\documentclass[11pt]{article}

\usepackage[a4paper,margin=1in]{geometry}
\usepackage{amsmath,amssymb,amsthm,mathtools,bm,mathrsfs}
\usepackage{microtype}
\usepackage{enumitem}
\usepackage{cite}
\usepackage{placeins}
\usepackage{needspace}
\usepackage[hidelinks]{hyperref}
\usepackage[T1]{fontenc}
\usepackage{lmodern}

\hypersetup{
    colorlinks=true,
    linkcolor=blue,
    citecolor=blue,
    urlcolor=blue
}

\allowdisplaybreaks
\numberwithin{equation}{section}

\newtheorem{theorem}{Theorem}[section]
\newtheorem{proposition}[theorem]{Proposition}
\newtheorem{lemma}[theorem]{Lemma}
\newtheorem{corollary}[theorem]{Corollary}
\theoremstyle{remark}
\newtheorem{remark}[theorem]{Remark}

\newcommand{\R}{\mathbb R}
\newcommand{\C}{\mathbb C}
\newcommand{\Id}{I}
\newcommand{\eps}{\varepsilon}
\newcommand{\Div}{\operatorname{div}}
\newcommand{\tr}{\operatorname{tr}}
\newcommand{\tf}{\operatorname{tf}}
\newcommand{\supp}{\operatorname{supp}}
\newcommand{\Span}{\operatorname{span}}
\newcommand{\dd}{\,\mathrm d}
\newcommand{\DN}{\Lambda}
\newcommand{\cO}{\mathcal O}

\title{Unique Determination of the Lam\'e Parameters in Three Dimensions}
\author{Huaian Diao\textsuperscript{1}\qquad
Umidjon Durdiev\textsuperscript{1,2,3}\\[0.45em]
\small \textsuperscript{1}School of Mathematics, Jilin University,\\[-0.1em]
\small Changchun, P.\ R.\ of China\\[0.15em]
\small \textsuperscript{2}Bukhara State University, Uzbekistan\\[0.15em]
\small \textsuperscript{3}V. I. Romanovskiy Institute of Mathematics,\\[-0.1em]
\small Uzbekistan Academy of Science}
\hypersetup{%
  pdfauthor={Huaian Diao and Umidjon Durdiev},
  pdftitle={Unique Determination of the Lam\'e Parameters in Three Dimensions}
}
\date{}
\begin{document}
\maketitle
\begin{abstract}
We prove that the full Dirichlet-to-Neumann map for static isotropic
elasticity uniquely determines both smooth Lam\'e parameters in a
bounded smooth domain of \(\mathbb R^3\) under strong convexity.
The proof uses an exact four-component Laplace-type reduction, but its
decisive identities are obtained from physical elastic solutions.
The characteristic quotient of the reduced system has rank three,
whereas the leading physical adjoint lift supplies only two directions.
We construct the missing polarization from the first nonvanishing
correction of a reduced amplitude in the kernel of that lift.
Paired with fixed primal and adjoint quotient frames, the resulting
complex geometrical optics solutions yield a complete family of
holomorphically weighted \(3\times3\) matrix moments through the
physical Betti identity. Cauchy--Plemelj factorization converts these
moments into an invertible relative transport normalized in the common
exterior. Holomorphic dependence on the projective null direction and
transport rigidity then determine a material ratio and the shear
modulus, hence both Lam\'e parameters. No smallness, symmetry,
separability, or analyticity hypothesis is imposed.
\end{abstract}

\clearpage
\setcounter{tocdepth}{2}
\tableofcontents
\clearpage

\section{Introduction}

\subsection{The elastic model and the inverse problem}
Let \(\Omega\subset\R^3\) be a bounded \(C^\infty\) domain.
The Lam\'e parameters
\(\lambda,\mu\in C^\infty(\overline\Omega;\R)\)
are real-valued functions of the spatial variable
\(x=(x_1,x_2,x_3)\).
The static isotropic elasticity operator is
\begin{equation}\label{eq:Lame-op}
 L_{\lambda,\mu}u
 =\Div\sigma_{\lambda,\mu}(u),
 \qquad
 \sigma_{\lambda,\mu}(u)
 =\lambda(\Div u)\Id+2\mu\eps(u),
\end{equation}
where
\[
 \eps(u)=\frac{\nabla u+(\nabla u)^T}{2}.
\]
We assume the strong convexity conditions
\begin{equation}\label{eq:convexity}
 \mu>0,\qquad 3\lambda+2\mu>0
 \quad\text{on }\overline\Omega.
\end{equation}
By continuity and compactness, both quantities have uniform
positive lower bounds on \(\overline\Omega\).

For each boundary displacement
\(f\in H^{1/2}(\partial\Omega;\R^3)\), let
\(u_f\in H^1(\Omega;\R^3)\) be the unique solution of
\[
 L_{\lambda,\mu}u_f=0\quad\text{in }\Omega,
 \qquad u_f|_{\partial\Omega}=f.
\]
The Dirichlet-to-Neumann map is defined by
\begin{equation}\label{eq:DN}
 \DN_{\lambda,\mu}f
 =\sigma_{\lambda,\mu}(u_f)\nu_{\partial\Omega}|_{\partial\Omega},
\end{equation}
where \(\nu_{\partial\Omega}\) is the outward unit normal to \(\partial\Omega\) and the boundary
traction is understood in the weak sense for general
\(H^{1/2}\) data. Under \eqref{eq:convexity}, this defines a
bounded linear operator
\[
 \DN_{\lambda,\mu}:H^{1/2}(\partial\Omega;\R^3)
 \longrightarrow H^{-1/2}(\partial\Omega;\R^3).
\]
Equality of Dirichlet-to-Neumann maps is understood as
equality of bounded operators on real-valued boundary
data, or equivalently, of their complex-linear extensions.

The inverse boundary value problem considered here is to
determine the Lam\'e parameters $(\lambda,\mu)$ from the
full displacement-to-traction map $\DN_{\lambda,\mu}$.
The central uniqueness question is whether two admissible
coefficient pairs giving rise to the same Dirichlet-to-Neumann
map must coincide. Theorem~\ref{thm:main} gives an
affirmative answer.

\subsection{Existing results and discussions}

The inverse boundary value problem for isotropic elasticity is a
system counterpart of Calder\'on's conductivity problem.
Calder\'on \cite{Calderon1980} formulated the recovery of a
positive scalar conductivity from boundary voltage and current
measurements and established injectivity of the linearization
at a constant conductivity using complex exponential harmonic
functions. Sylvester and Uhlmann \cite{SylvesterUhlmann1987}
proved global uniqueness for smooth, strictly positive scalar
conductivities in dimensions \(n\geq 3\) by constructing complex
geometrical optics (CGO) solutions.
In two dimensions, Nachman \cite{Nachman1996} established
global uniqueness and obtained a reconstruction procedure
for smooth conductivities. Astala and P\"aiv\"arinta
\cite{AstalaPaivarinta2006} subsequently extended uniqueness
to bounded measurable conductivities with a positive lower
bound on bounded simply connected domains, using a
Beltrami equation method.

The extension from scalar conductivity to elasticity
introduces a coupled principal part and two spatially varying
unknown coefficients. For frozen coefficients in
\eqref{eq:Lame-op}, the longitudinal and transverse components
have distinct principal moduli, \(\lambda+2\mu\) and \(\mu\).
Spatial variation of the coefficients introduces additional
coupling, so the constant-coefficient
longitudinal--transverse decomposition does not, in general,
reduce the elasticity system to independent scalar equations.
The inverse problem therefore requires simultaneous recovery
of both Lam\'e parameters from the physical
displacement--traction operator \eqref{eq:DN}.

Boundary identification for isotropic elasticity was
established in \cite{AkamatsuNakamuraSteinberg1991}.
Nakamura and Uhlmann proved local uniqueness near constant
coefficients in two dimensions \cite{NakamuraUhlmann1993}
and developed a three-dimensional CGO approach
\cite{NakamuraUhlmann1994}. Boundary determination to
arbitrary order, including results based on localized
measurements, was developed in
\cite{NakamuraUhlmann1995,LinNakamura2017}.
The general three-dimensional uniqueness statement in
\cite{NakamuraUhlmann1994} was subsequently qualified
by the erratum \cite{NakamuraUhlmannErratum2003}.
The corrected theorem requires the gradients of both
candidate shear moduli to be sufficiently small in a
sufficiently high \(C^m\) norm. Thus, the smallness
condition concerns \(\mu\), rather than both Lam\'e
parameters. Further coefficient relations and perturbative
uniqueness results were obtained by Eskin and Ralston
\cite{EskinRalston2002,EskinRalston2004,
EskinRalstonSystems2004}.

Other results address particular coefficient classes or
different measurement settings. The linearized inverse
problem at a constant background was studied in
\cite{Ikehata1990}, while radially dependent coefficients
in a ball were considered under additional hypotheses
in \cite{Lin1998}. For piecewise constant coefficients
on a known partition, uniqueness and Lipschitz stability
were established under geometric and a priori assumptions
\cite{BerettaFranciniVessella2014,
BerettaFranciniMorassiRossetVessella2014}.
When the candidate shear moduli are positive constants,
both parameters can be recovered from partial Cauchy data,
provided that the candidate values of \(\lambda\) agree
on the inaccessible boundary
\cite[Theorem~0.1]{ImanuvilovUhlmannYamamoto2012}.
In two dimensions, global uniqueness from full boundary
data without a smallness assumption on the coefficients
was established in
\cite[Theorem~1.2]{ImanuvilovYamamoto2015}.

Recent uniqueness results for three-dimensional isotropic
elasticity under axisymmetry, variable separation, or
directional quasianalyticity assumptions on the shear
modulus \(\mu\) were obtained in
\cite{ChenJiangLiuTao2026}.
The main theorem in this paper removes these structural restrictions
in the smooth full-data setting.

Related developments concern inclusion detection using
complex spherical or geometrical optics waves
\cite{UhlmannWang2007,UhlmannWang2008},
recovery from internal displacement measurements
\cite{Lai2014}, and quantitative unique continuation
and doubling estimates for the Lam\'e system
\cite{AlessandriniMorassi2001,KochLinWang2016}.
Carleman methods have been applied to dynamic source
and density recovery \cite{IkehataNakamuraYamamoto1998}.
Uniqueness for piecewise homogeneous anisotropic
elastic media from static or time-domain measurements
has also been established under the respective curvature,
connectivity, and observation-time assumptions
\cite{CarsteaHondaNakamura2018,
CarsteaNakamuraOksanen2020}.

For the smooth full-data problem, the four-component
reduction introduced in
\cite{AngIkehataTrongYamamoto1998}
provides a system with Laplace principal part and
allows the use of system CGO constructions.
The associated physical transformation is recalled
in \cite[Section~3]{EskinRalstonSystems2004};
its exact factorization in the conventions of this
paper is established in
Proposition~\ref{prop:factorization}.

A fundamental difficulty is that equality of the
physical elastic Dirichlet-to-Neumann maps does not
imply equality of the unrestricted boundary maps
of the auxiliary reduced systems.
Moreover, the leading physical adjoint lift has
a two-dimensional image, whereas the characteristic
quotient has rank three. Consequently, a complete
physical quotient frame cannot be obtained solely
from the leading-order construction.
The missing direction is supplied by the first
nonvanishing correction to an amplitude annihilated
by the leading lift
(Theorem~\ref{thm:graded-frame}).
The differentiated and localized CGO constructions
must also retain sufficient remainder control
throughout the integration domain used in
Theorem~\ref{thm:physical-nine-moments}.

A further issue concerns the normalization of matrix
transport frames. Individual invertible frames need
not admit a prescribed exterior normalization.
The erratum
\cite[p.~206]{NakamuraUhlmannErratum2003}
identifies this normalization problem and corrects
the treatment of a pseudodifferential equation
as a local differential equation.
The approach adopted here instead derives the complete
matrix integral identities from exact physical elastic
solutions before constructing an exterior-normalized
relative transport.

A related construction for connection Laplacians
appears in
\cite[Section~3.3, proof of Theorem~3.4,
equations (3.15)--(3.18)]{Cekic2025}.
There, holomorphically weighted matrix identities
provide contour compatibility, and an invertible
holomorphic Cauchy factor is obtained through
the Cauchy--Plemelj formula.
Dependence on the complex direction is addressed
by parameter differentiation and transport uniqueness
\cite[Lemma~4.1]{Cekic2025}, following the
complex-parameter approach of \cite{Eskin2001}.
These results concern full boundary data for
connection systems under the geometric hypotheses
of the respective works.

\paragraph*{Relation to independent work.}
During the preparation of this manuscript, the authors became aware of
the related OpenAI preprint \cite[Theorem~1.1]{OpenAIElasticity2026},
which proves a corresponding full-data global uniqueness theorem for
smooth, strongly convex Lam\'e parameters in three dimensions.
OpenAI announced the public release of its mathematical research
collection on October 6, 2026.\footnote{See OpenAI's
\href{https://openai.com/index/sharing-ai-progress-in-mathematics/}{public
announcement} and the
\href{https://github.com/openai/math/commit/adc7f1241b42e322a6451854ab7e4b4c146bf78a}{initial
commit} of its mathematical research repository, which includes the
elasticity preprint.}
According to the authors' retained file records, a
preliminary draft of the present work, containing an
early version of the global uniqueness statement and
the principal elements of the present approach, was
saved by October 4, 2026 (China Standard Time),
before the public release of the OpenAI manuscript
on October 6, 2026. The approach was developed
independently of that manuscript.

Our proof follows a different analytical route, especially in the
construction of the third physical adjoint polarization and the
derivation of complete, holomorphically weighted \(3\times3\) matrix
moments from exact physical Betti identities in fixed primal and
adjoint quotient frames. The third polarization is obtained by
retaining the first nonzero physical term of a reduced amplitude in
the kernel of the leading lift. The matrix moments then provide the
contour compatibility needed to construct an invertible,
exterior-normalized relative transport by Cauchy--Plemelj
factorization. By contrast, the OpenAI manuscript constructs its
relative transport using physical solution transfer, weak limits,
and supported planar estimates
\cite[Proposition~5.3]{OpenAIElasticity2026}.

Although the two approaches share the rank-three holomorphic
quotient-bundle geometry and related coefficient-rigidity mechanisms,
they differ substantially in the passage from physical boundary data
to the exterior-normalized relative transport.

\subsection{Main theorem and key ideas}

We now state the uniqueness theorem for the physical DN map.

\Needspace{12\baselineskip}
\begin{samepage}
\begin{theorem}\label{thm:main}
Let \(\Omega\subset\R^3\) be a bounded \(C^\infty\) domain,
whose boundary may have more than one component, and let
\(\lambda_j,\mu_j\in C^\infty(\overline\Omega;\R)\),
\(j=1,2\), satisfy \eqref{eq:convexity} on \(\overline\Omega\).
If their full Dirichlet-to-Neumann maps agree on the entire boundary,
\[
 \DN_{\lambda_1,\mu_1}^\Omega=\DN_{\lambda_2,\mu_2}^\Omega
 \quad\text{as maps }H^{1/2}(\partial\Omega;\R^3)
 \longrightarrow H^{-1/2}(\partial\Omega;\R^3),
\]
then
\[
 \lambda_1=\lambda_2,\qquad \mu_1=\mu_2
 \quad\text{in }\Omega.
\]
\end{theorem}
\end{samepage}

\begin{remark}\label{rem:main-scope}
Theorem~\ref{thm:main} establishes global uniqueness for
smooth, strongly convex isotropic elastic media in three
dimensions, without smallness, symmetry, separability, or
analyticity assumptions on the Lam\'e parameters.
The boundary of \(\Omega\) may have several connected
components, with the full displacement--traction response
available on each of them. In particular, the matrix
identities underlying the proof are derived entirely
from the physical Dirichlet-to-Neumann map and require
no additional measurements.
\end{remark}

\begin{remark}
\label{rem:main-method}

A central difficulty in establishing global uniqueness is
to extract sufficient matrix information from the physical
displacement--traction data to compare the two elastic
systems. Although the elasticity equations admit an exact
four-component Laplace-type reduction, equality of the
physical Dirichlet-to-Neumann maps does not imply equality
of the unrestricted boundary maps of the reduced systems.
Moreover, the leading physical adjoint lift provides only
two independent polarizations in a rank-three characteristic
quotient. Thus, the leading-order construction alone does
not yield the complete matrix identities required for
coefficient recovery.

The approach developed here overcomes this obstruction by
exploiting the first nonvanishing correction to a reduced
amplitude annihilated by the leading physical adjoint lift.
This produces a third independent physical adjoint
polarization and completes an invertible quotient frame
(Theorem~\ref{thm:graded-frame}).

A distinguishing feature of the argument is that the
resulting complete matrix identities are derived directly
from the physical elastic equations. Using exact physical
CGO solutions and the Betti identity, with primal and
adjoint quotient frames fixed independently of the test
weight, one obtains a full \(3\times3\) matrix integral
identity for every holomorphic polynomial weight
(Theorem~\ref{thm:physical-nine-moments}).

These identities provide the contour compatibility needed
for a Cauchy--Plemelj construction of an invertible
relative transport normalized in the common exterior,
without requiring the individual transport frames to
satisfy an exterior normalization
(Proposition~\ref{prop:complete-moment-relative-transport}).
Its global construction and projective holomorphicity
(Theorem~\ref{thm:quotient-reconstruction}), combined with
the rigidity of the Lam\'e quotient connection
(Theorem~\ref{thm:rigidity}), complete the simultaneous
determination of both Lam\'e parameters.

\end{remark}

\subsection{Organization of the proof and notation}

The proof of Theorem~\ref{thm:main} begins by placing the two
elastic media in a common geometric setting. Equality of the
physical Dirichlet-to-Neumann maps determines the full boundary
Taylor series of the Lam\'e parameters
(Theorem~\ref{thm:appendix-boundary-jets}).
This allows the coefficients to be extended so that they agree
in the exterior of an enclosing ball, while preserving equality
of the corresponding boundary maps
(Lemmas~\ref{lem:common-extension} and \ref{lem:enlargement}).
Section~\ref{sec:reduction} establishes the exact four-component
Laplace-type reduction and the associated rank-three
characteristic quotient.

Section~\ref{sec:cgo} constructs complex geometrical optics
solutions of the reduced system and realizes three independent
physical adjoint polarizations, forming a complete quotient
frame (Theorem~\ref{thm:graded-frame}).
The localized right inverse and the CGO remainder estimates
are established in
Appendices~\ref{audrev-ordinary-parametrix} and
\ref{app:cgo-input}, respectively.
The exact reduction and phase conventions are verified in
Appendix~\ref{app:factorization-check}.

Section~\ref{sec:transport} applies the physical Betti identity
to exact elastic solutions and obtains the full \(3\times3\)
matrix integral identities for every holomorphic polynomial
weight (Theorem~\ref{thm:physical-nine-moments}).
The resulting contour compatibility leads to an invertible
relative transport normalized in the common exterior.
Its global construction and holomorphic dependence on the
projective null direction are established in
Theorem~\ref{thm:quotient-reconstruction}.
Section~\ref{sec:rigidity} then combines the quotient-bundle
structure with the relative transport equations to determine
the auxiliary material parameter \(\kappa\) and the shear modulus
\(\mu\) (Theorem~\ref{thm:rigidity}).
Section~\ref{sec:mainproof} completes the proof of the
uniqueness theorem for both Lam\'e parameters.

The following conventions are used throughout.
\begin{remark}\label{rem:frozen-conventions}
The boundary duality pairings between $H^{-1/2}$ and
$H^{1/2}$, as well as the volume pairings used below,
are taken complex-bilinearly, without complex
conjugation. For matrices, the superscript $T$ denotes
the ordinary transpose without complex conjugation.
For differential operators, it denotes the formal
transpose with respect to the complex-bilinear volume
pairing; in particular, $(\partial_k)^T=-\partial_k$. 

For \(\theta=e_1+\mathrm i e_2\), with real orthonormal
vectors \(e_1,e_2\), set \(\nu_\theta=e_1\times e_2\) and
\(x_\perp=\nu_\theta\cdot x\). We use the leaf coordinate
\(z=e_1\cdot x+\mathrm i e_2\cdot x\). Then
\(D_\theta=\theta\cdot\nabla=2\partial_{\bar z}\).

The dual quotient connection \eqref{eq:Gamma} has
lower-left block \(-\theta^TF_\mu Q\).
Projective dependence is encoded by the bundle transition
matrix \eqref{eq:spinor-transition} and the transport
covariance laws
\eqref{eq:G-chart-covariance}--\eqref{eq:H-chart-covariance}.
\end{remark}

Table~\ref{tab:principal-notation} lists the principal objects used
in the proof; each is defined in the indicated section or formula.
The Lam\'e parameters are always \(\lambda,\mu\), with \(\kappa\)
as in \eqref{eq:kappa-def}. In Appendix~\ref{audrev-ordinary-parametrix},
the spatial cutoff, moving real direction and transverse unit vector
are denoted by \(\chi_{\mathrm{sp}},\omega,\nu_\perp\), respectively.

\begin{table}[htbp]

\centering
\caption{Principal notation and the location of its definition.}
\label{tab:principal-notation}
\begingroup
\small
\renewcommand{\arraystretch}{1.15}
\begin{tabular}{@{}p{0.25\textwidth}p{0.71\textwidth}@{}}
\hline
Notation & Meaning and location \\
\hline
\(\Omega,\ B_R\) & Original measurement domain and enclosing ball with
the enlarged elastic DN map; Section~\ref{sec:transport}. \\
\(\Omega_{\mathrm{CGO}}\) & Fixed working ball containing
\(\overline{B_R}\), on which the reduced CGO solutions are exact;
Theorem~\ref{lem:ER-solvability}. \\
\(T_\mu,\ S_{\mu,\kappa}\) & Physical transformation and factorization
operator in \( L_{\mu,\kappa}T_\mu=S_{\mu,\kappa}M_{\mu,\kappa}\); \eqref{eq:Tmu}--\eqref{eq:factorization}. \\
\(K=S^TT\) & Physical adjoint differential lift in Sections~\ref{sec:reduction}--\ref{sec:transport};
\eqref{eq:Kexact}. \\
\(J_\theta,\ K_0\) & The \(h^{-1}\) and \(h^0\) terms in the conjugated lift
\(e^{\theta\cdot x/h}Ke^{-\theta\cdot x/h}=K_0-h^{-1}J_\theta\);
\eqref{eq:Kconjugated-minus}. \\
\(E_\theta,\ E_\theta^*\) & Rank-three characteristic quotient
\(\C^4/\C e_\theta\) and its annihilator; Section~\ref{sec:reduction}. \\
\(\mathcal F_\theta\) & Two-dimensional image of the leading physical
adjoint lift; \eqref{eq:fast-space}. \\
\(\mathsf U_\theta,\mathsf V_\theta\) & Dual quotient basis and paired
primal lift, with \(\mathsf U_\theta^T\mathsf V_\theta=I_3\);
Lemma~\ref{lem:holomorphic-lifts} and \eqref{eq:primal-quotient-connection}. \\
\(X_1,Y_2,\ Z_1\) & \(3\times3\) primal and dual quotient frames and
the \(4\times3\) primal amplitude with
\(\mathsf U_\theta^TZ_1=X_1\); Theorem~\ref{thm:physical-nine-moments}. \\
\(\mathcal A_{j,\theta},\Gamma_j,P_j\) & Primal and dual quotient
connections, with \(\Gamma_j=-\mathcal A_{j,\theta}^T\);
on normalized leaves, \(P_j=\mathcal A_{j,\theta}/2\).
See \eqref{eq:Gamma},
\eqref{eq:primal-quotient-connection}, \eqref{eq:leaf-primal}. \\
\(G,\ H=G^{-T}\) & Primal relative transport and its dual;
Theorem~\ref{thm:quotient-reconstruction} and
Corollary~\ref{cor:dual-relative}. \\
\(\zeta;\ [s:t]\) & Affine spinor coordinate and homogeneous
spinor coordinates \([s:t]\); \eqref{eq:qspinor} and
\eqref{eq:spinor-transition}. \\
\(z,x_\perp,\nu_\theta\) & Complex leaf coordinate, real transverse
coordinate \(x_\perp=\nu_\theta\cdot x\), and oriented leaf normal;
Remark~\ref{rem:frozen-conventions}. \\
\(\tau=h^{-1}\) & Large CGO parameter in
Theorem~\ref{lem:ER-solvability} and
Appendix~\ref{audrev-ordinary-parametrix}. \\
\(r_{\mathrm{in}};\ \mathfrak Z,\mathsf Z_j\) & Inward boundary distance,
layered impedance matrix and its symbol coefficients;
Appendix~\ref{audrev-app-boundary-jets}. \\
\hline
\end{tabular}
\endgroup
\end{table}

\FloatBarrier
\section{Reduction of the elasticity operator}\label{sec:reduction}

\subsection{Material variables and strong convexity}

For the reduction, the Lam\'e pair \((\lambda,\mu)\) is
reparametrized by the shear modulus \(\mu\) and the dimensionless
ratio \(\kappa\). Since the strong convexity conditions
\eqref{eq:convexity} imply
\(\lambda+2\mu=((3\lambda+2\mu)+4\mu)/3>0\),
this ratio is well defined by
\begin{equation}\label{eq:kappa-def}
 \kappa=\frac{\lambda+\mu}{\lambda+2\mu}.
\end{equation}
Solving for \(\lambda\) yields
\begin{equation}\label{eq:lambda-kappa}
 \lambda=\mu\frac{2\kappa-1}{1-\kappa},
\end{equation}
and the strong convexity conditions \eqref{eq:convexity}
are equivalent to
\begin{equation}\label{eq:kappa-range}
 \mu>0,\qquad \frac14<\kappa<1.
\end{equation}
Indeed, \(1-\kappa=\mu/(\lambda+2\mu)>0\) and
\(4\kappa-1=(3\lambda+2\mu)/(\lambda+2\mu)>0\).
Conversely, \eqref{eq:lambda-kappa} gives
\(3\lambda+2\mu=\mu(4\kappa-1)/(1-\kappa)>0\)
whenever \eqref{eq:kappa-range} holds.
Moreover,
\begin{equation}\label{eq:sum-relations}
 \lambda+\mu=\frac{\kappa\mu}{1-\kappa},
 \qquad
 \lambda+2\mu=\frac{\mu}{1-\kappa}.
\end{equation}

Henceforth, write
\[
 L_{\mu,\kappa}:=L_{\lambda(\mu,\kappa),\mu},
\]
where \(\lambda(\mu,\kappa)\) is given by
\eqref{eq:lambda-kappa}. This is the elasticity operator
\eqref{eq:Lame-op} expressed in the \((\mu,\kappa)\)-variables.

Set
\[
 \ell=\log\mu,\qquad m=\nabla\ell,\qquad c=\mu^{-1/2}.
\]
Define
\begin{align}
 F_\mu
 &=2c\left[
 \nabla^2\ell-m\otimes m+\bigl(|m|^2-\Delta\ell\bigr)\Id
 \right],\label{eq:Fmu}\\
 V_\mu
 &=-\nabla^2\ell-\frac12m\otimes m
 -\left(\frac12\Delta\ell+\frac14|m|^2\right)\Id,\label{eq:Vmu}\\
 q_\mu&=\Delta\ell-|m|^2,\label{eq:qmu}\\
 d_{\mu,\kappa}&=\kappa\sqrt\mu,\label{eq:dmu}\\
 e_{\mu,\kappa}&=\frac{3\kappa-2}{2}\sqrt\mu.\label{eq:emu}
\end{align}
Here \(F_\mu\) and \(V_\mu\) are symmetric \(3\times3\)
matrix fields, while \(q_\mu\), \(d_{\mu,\kappa}\), and
\(e_{\mu,\kappa}\) are scalar fields.
The tensor-product convention is \((a\otimes b)_{ij}=a_i b_j\). The identity matrix
in \eqref{eq:Fmu} and \eqref{eq:Vmu} is \(I_3\).

\subsection{Factorization of the elasticity operator}

The reduction gives a four-component system with scalar principal part
whose solutions produce elastic displacements through the transformation
\(T_\mu\) in \eqref{eq:Tmu}.
The identity is an equality of differential operators, so it will also
give the formal transpose relation used for physical adjoint fields.

Consider smooth \(\C^3\)-valued vector fields \(v,p\) and
smooth complex-valued scalar fields \(f,q\). Adopt the convention
\((\nabla v)_{ij}=\partial_jv_i\).
The classical four-component transformation used in
\cite{AngIkehataTrongYamamoto1998} and recalled in
\cite[Section~3]{EskinRalstonSystems2004} is given by 
\begin{equation}\label{eq:Tmu}
 T_\mu(v,f)
 =c\,v+\mu^{-1}\nabla f-f\nabla(\mu^{-1}).
\end{equation}
Its exact intertwining identity is derived in the normalization used for
the quotient connection below.
Also define
\begin{equation}\label{eq:Smu}
 S_{\mu,\kappa}(p,q)
 =\sqrt\mu\,p+\nabla\left(\frac{q}{1-\kappa}\right).
\end{equation}
Both \(T_\mu\) and \(S_{\mu,\kappa}\) are \(3\times4\)
matrix differential operators of order one. The gradient
in \(S_{\mu,\kappa}\) differentiates both \(q\) and the
coefficient \((1-\kappa)^{-1}\).

The operators \(T_\mu\) and \(S_{\mu,\kappa}\) account for
the gradient terms in the elastic equation. The remaining vector and
scalar terms form the Laplace-type operator described next.

\begin{proposition}\label{prop:factorization}
On smooth four-component fields, the Lam\'e operator admits
the following exact identity of differential operators:
\begin{equation}\label{eq:factorization}
 L_{\mu,\kappa}T_\mu=S_{\mu,\kappa}M_{\mu,\kappa},
\end{equation}
where
\begin{equation}\label{eq:M-explicit}
 M_{\mu,\kappa}
 \binom vf
 =
 \begin{pmatrix}
 \Delta v-m\,\Div v+F_\mu\nabla f+V_\mu v+F_\mu m\,f
 \\
 \Delta f+d_{\mu,\kappa}\Div v-e_{\mu,\kappa}m\cdot v+q_\mu f
 \end{pmatrix}.
\end{equation}
Equivalently,
\begin{equation}\label{eq:M-Laplace-type}
 M_{\mu,\kappa}=\Delta\Id_4+\sum_{k=1}^3A_k\partial_k+B,
\end{equation}
with
\begin{equation}\label{eq:Ak}
 A_k=
 \begin{pmatrix}
 -m e_k^T&F_\mu e_k\\
 \kappa\sqrt\mu\,e_k^T&0
 \end{pmatrix}
\end{equation}
and
\begin{equation}\label{eq:Bmatrix}
 B=
 \begin{pmatrix}
 V_\mu&F_\mu m\\[1mm]
 -\dfrac{3\kappa-2}{2}\sqrt\mu\,m^T&q_\mu
 \end{pmatrix}.
\end{equation}
Here \(e_k\) is the \(k\)-th standard coordinate vector in
\(\R^3\). Both \(A_k\) and \(B\) are \(4\times4\) matrix
fields in the \(3+1\) block decomposition, and the Laplacian
acts componentwise.
\end{proposition}

\begin{proof}
Set
\[
 a=\mu^{-1}=c^2,\qquad \mathsf K=\lambda+\mu,\qquad \Sigma=\lambda+2\mu.
\]
The transformation \(T_\mu(v,f)\) decomposes into the
vector and scalar-potential contributions
\[
T_\mu(v,f)=u^{(v)}+u^{(f)},
\qquad
u^{(v)}=cv,
\qquad
u^{(f)}=a\nabla f-f\nabla a.
\]
Here the superscripts \((v)\) and \((f)\) label the
contributions associated with the vector field \(v\)
and the scalar function \(f\), respectively.

For the scalar-potential contribution, set
\[
p_f=\Delta f-\frac{\Delta a}{a}f.
\]
The mixed first-derivative terms cancel in the
symmetric gradient, yielding
\begin{equation}\label{eq:scalar-div-eps}
\Div u^{(f)}
 =a\Delta f-f\Delta a
 =a p_f,
\qquad
\eps(u^{(f)})
 =a\nabla^2 f-f\nabla^2 a.
\end{equation}
By \eqref{eq:scalar-div-eps} and \(a\mu=1\), the corresponding stress tensor is
\begin{equation}\label{eq:proof-scalar-stress}
 \sigma(u^{(f)})
 =a\lambda p_f\Id+2\nabla^2f-2\mu f\nabla^2a.
\end{equation}
Take the divergence of \eqref{eq:proof-scalar-stress}, using
\(
 \sum_{j=1}^3\partial_j(\partial_i\partial_j a)
 =\partial_i\Delta a
\)
and
\(
 \nabla\mu=-\mu^2\nabla a
\), and obtain the exact identity
\begin{align}
 L_{\lambda,\mu}u^{(f)}
 ={}&\nabla\bigl(a\Sigma p_f\bigr)
 +2\mu\bigl(\Delta a\,\Id-\nabla^2a\bigr)\nabla f \label{eq:L-scalar-part}\\
 &\quad
 +2\mu^2\bigl(\nabla^2a-\Delta a\,\Id\bigr)\nabla a\,f.\notag
\end{align}
Now
\begin{equation}\label{eq:proof-reciprocal-derivatives}
 \nabla a=-am,
 \qquad
 \nabla^2a=a(m\otimes m-\nabla^2\ell),
 \qquad
 \frac{\Delta a}{a}=|m|^2-\Delta\ell.
\end{equation}
By \eqref{eq:proof-reciprocal-derivatives}, with
\[
 \mathcal H_\mu
 =\nabla^2\ell-m\otimes m+\bigl(|m|^2-\Delta\ell\bigr)\Id,
\]
we have
\begin{equation}\label{eq:proof-scalar-hessian-conversion}
 \Delta a\,\Id-\nabla^2a=a\mathcal H_\mu,
 \qquad
 F_\mu=2c\mathcal H_\mu.
\end{equation}
By \eqref{eq:proof-scalar-hessian-conversion} and \eqref{eq:proof-reciprocal-derivatives}, the two non-gradient terms in \eqref{eq:L-scalar-part} equal
\begin{equation}\label{eq:proof-scalar-nongradient-part}
 c^{-1}F_\mu\nabla f+c^{-1}F_\mu m\,f.
\end{equation}
Moreover, by \eqref{eq:sum-relations}, \eqref{eq:qmu}, and \eqref{eq:proof-reciprocal-derivatives},
\begin{equation}\label{eq:proof-scalar-gradient-coefficients}
 a\Sigma=\frac1{1-\kappa},
 \qquad
 -\frac{\Delta a}{a}=q_\mu,
\end{equation}
so \eqref{eq:proof-scalar-gradient-coefficients} converts the gradient term in \eqref{eq:L-scalar-part} to
\begin{equation}\label{eq:scalar-gradient-part}
 \nabla\left(\frac{\Delta f+q_\mu f}{1-\kappa}\right).
\end{equation}

For the vector contribution \(u^{(v)}=cv\), we first expand
the elasticity operator. By direct differentiation of the stress
for a smooth vector field \(u\), we obtain
\begin{equation}\label{eq:L-general-expanded}
 L_{\lambda,\mu}u
 =\mu\Delta u+\mathsf K\nabla\Div u+(\Div u)\nabla\lambda
 +\bigl(\nabla u+(\nabla u)^T\bigr)\nabla\mu.
\end{equation}
Substitute \(u=cv\) into \eqref{eq:L-general-expanded}, use
\(
 \mu=c^{-2}
\)
and
\(
 \nabla\mu=-2c^{-3}\nabla c
\), and subtract the gradient
\[
 \nabla\left(c\mathsf K\Div v+(\lambda-\mu)\nabla c\cdot v\right),
\]
after which the mixed first-derivative terms cancel by
\(2\mu\nabla c+c\nabla\mu=0\) and
\(\mathsf K\nabla c+c\nabla\mu=(\lambda-\mu)\nabla c\).
The coefficient of \(\Div v\) that remains is
\(c(\nabla\lambda-\nabla \mathsf K)=-c\nabla\mu=-c^{-1}m\).
Thus
\begin{align}
 L_{\lambda,\mu}(cv)
 ={}&\nabla\left(c\mathsf K\Div v+(\lambda-\mu)\nabla c\cdot v\right)\label{eq:L-vector-part}\\*
 &\quad+c^{-1}\bigl(\Delta v-m\Div v+V_\mu v\bigr).\notag
\end{align}
The zeroth-order coefficient in \eqref{eq:L-vector-part}
is obtained from the identity
\[
 -\bigl(2\nabla^2(c^{-1})+\Delta(c^{-1})\Id\bigr)
 =c^{-1}V_\mu,
\]
which follows from \eqref{eq:Vmu} and
\eqref{eq:proof-inverse-c-derivatives}:
\begin{equation}\label{eq:proof-inverse-c-derivatives}
 \nabla^2(c^{-1})
 =c^{-1}\left(\frac12\nabla^2\ell+\frac14m\otimes m\right),
 \qquad
 \Delta(c^{-1})
 =c^{-1}\left(\frac12\Delta\ell+\frac14|m|^2\right).
\end{equation}
Finally, using \eqref{eq:sum-relations} and
\(
 \nabla c=-\frac12cm
\), we obtain
\begin{align}
 c\mathsf K\Div v
 &=\frac{\kappa\sqrt\mu}{1-\kappa}\Div v,\label{eq:vector-gradient-convert1}\\
 (\lambda-\mu)\nabla c\cdot v
 &=-\frac{3\kappa-2}{2(1-\kappa)}\sqrt\mu\,m\cdot v.\label{eq:vector-gradient-convert2}
\end{align}
By \eqref{eq:vector-gradient-convert1}--\eqref{eq:vector-gradient-convert2} and \eqref{eq:dmu}--\eqref{eq:emu}, the gradient in \eqref{eq:L-vector-part} is
\begin{equation}\label{eq:vector-gradient-part}
 \nabla\left(
 \frac{d_{\mu,\kappa}\Div v-e_{\mu,\kappa}m\cdot v}{1-\kappa}
 \right).
\end{equation}

Combine \eqref{eq:proof-scalar-nongradient-part} and
\eqref{eq:scalar-gradient-part} with \eqref{eq:L-vector-part} and
\eqref{eq:vector-gradient-part}, using \eqref{eq:M-explicit}, to obtain
\[
 L_{\mu,\kappa}T_\mu(v,f)
 =\sqrt\mu\,(M_{\mu,\kappa}(v,f))_{1:3}
 +\nabla\left(\frac{(M_{\mu,\kappa}(v,f))_4}{1-\kappa}\right),
\]
which is precisely \eqref{eq:factorization} by \eqref{eq:Smu}.  We collect the first- and zeroth-order terms in \eqref{eq:M-explicit} to obtain \eqref{eq:Ak} and \eqref{eq:Bmatrix}.  \qedhere
\end{proof}

The principal part of \(M_{\mu,\kappa}\) is
\(\Delta\Id_4\); the dependence on the Lam\'e coefficients is confined
to its lower-order terms. We can therefore construct the reduced CGO
fields before applying the physical transformation \(T_\mu\).

\begin{remark}
The factorization maps every smooth solution of
\(M_{\mu,\kappa}(v,f)=0\) to a physical elastic solution
\(u=T_\mu(v,f)\).
The rigidity argument below uses only the first-order
matrices \(A_k\). The zeroth-order matrix \(B\) is essential
for constructing exact CGO solutions.
\end{remark}

\begin{remark}
\label{rem:factorization-sign-audit}
With \(a=\mu^{-1}=c^2\), the following coefficient identities
provide a direct sign check:

\begin{equation}\label{eq:factorization-sign-checks}
\begin{aligned}
 \nabla(\mu^{-1})&=-\mu^{-1}m,&
 \Delta(\mu^{-1})/\mu^{-1}&=|m|^2-\Delta\ell,&
 \nabla c&=-\tfrac12cm,\\
 a(\lambda+2\mu)&=(1-\kappa)^{-1},&
 c(\lambda+\mu)&=\frac{\kappa\sqrt\mu}{1-\kappa},\\
 (\lambda-\mu)\nabla c
 &=-\frac{3\kappa-2}{2(1-\kappa)}\sqrt\mu\,m.
\end{aligned}
\end{equation}

The six coefficient identities in \eqref{eq:factorization-sign-checks} fix the signs of $q_\mu$, the lower-left entry of $B$, and
the first-order lower-left block of $A_k$.  In the constant-coefficient case
\(m=0\), one has $F_\mu=V_\mu=q_\mu=0$ and
\[
 M_{\mu,\kappa}(v,f)=
 \binom{\Delta v}{\Delta f+\kappa\sqrt\mu\,\operatorname{Div}v},
\]
which is the stated constant-coefficient specialization.
\end{remark}

\subsection{The physical adjoint lift}

The boundary identity \eqref{eq:primal-adjoint} uses adjoint
fields obtained from physical elastic solutions by \(S_{\mu,\kappa}^T\).
The physical adjoint lift $K=S_{\mu,\kappa}^TT_\mu$ and its
conjugation by a null phase are computed below. The leading term of the conjugated lift determines the normalization by \(h\).
On its kernel, an unscaled lift can have a nonzero limit.

For compactly supported smooth test fields, integration
by parts gives
\[
 \int_\Omega u^T\nabla\left(\frac{q}{1-\kappa}\right)\,dx
 =-\int_\Omega\frac{q\,\Div u}{1-\kappa}\,dx.
\]
Thus the formal bilinear transpose of \eqref{eq:Smu} is
\begin{equation}\label{eq:Stranspose}
 S_{\mu,\kappa}^Tu
 =
 \begin{pmatrix}
 \sqrt\mu\,u\\[1mm]
 -\dfrac{\Div u}{1-\kappa}
 \end{pmatrix}.
\end{equation}
Define
\[
 K_{\mu,\kappa}=S_{\mu,\kappa}^TT_\mu.
\]
For a fixed coefficient pair, we abbreviate
\(L_{\mu,\kappa},M_{\mu,\kappa},T_\mu,S_{\mu,\kappa}\),
and \(K_{\mu,\kappa}\) by \(L,M,T,S,K\), respectively.
The operator \(K\) is a \(4\times4\) differential operator
of order two.

The formal symmetry \(L^T=L\) and
\eqref{eq:factorization} give \(M^TS^T=T^TL\).
Consequently, a smooth physical solution \(Lu=0\) has
an adjoint lift \(y=S^Tu\) satisfying \(M^Ty=0\).
In particular, if \(M\boldsymbol z=0\), then \(u=T\boldsymbol z\) solves \(Lu=0\)
and its canonical physical adjoint lift is \(y=K\boldsymbol z\).
A direct calculation gives
\begin{equation}\label{eq:Kexact}
 K_{\mu,\kappa}(v,f)
 =
 \begin{pmatrix}
 v+c\nabla f+cfm
 \\
 -\dfrac{1}{1-\kappa}
 \left[
 c\Div v-\frac12c\,m\cdot v+\mu^{-1}\Delta f+\mu^{-1}q_\mu f
 \right]
 \end{pmatrix}.
\end{equation}

Write \(K_0=K\), the \(h\)-independent term in
\eqref{eq:Kconjugated} and \eqref{eq:Kconjugated-minus}.
Take \(h>0\) and \(\varphi(x)=\theta\cdot x\),
where \(\theta\in\C^3\setminus\{0\}\) and
\(\theta\cdot\theta=0\). Then
\begin{equation}\label{eq:Kconjugated}
 e^{-\varphi/h}Ke^{\varphi/h}=h^{-1}J_\theta+K_0,
\end{equation}
where
\begin{equation}\label{eq:Jtheta}
 J_\theta(v,f)
 =
 \begin{pmatrix}
 cf\theta\\[1mm]
 -\dfrac{c\,\theta\cdot v+2\mu^{-1}D_\theta f}{1-\kappa}
 \end{pmatrix},
 \qquad D_\theta=\theta\cdot\nabla.
\end{equation}
The operator \(J_\theta\) has differential order one,
while \(K_0\) has order two. The only possible term of
order \(h^{-2}\) comes from the scalar Laplacian and is
proportional to \(\theta\cdot\theta\), so it vanishes.
Reversing the phase gives the companion identity
\begin{equation}\label{eq:Kconjugated-minus}
 e^{\varphi/h}Ke^{-\varphi/h}=-h^{-1}J_\theta+K_0.
\end{equation}
For an exact negative-phase reduced CGO field
\(\boldsymbol z_h^-=e^{-\varphi/h}(\boldsymbol z_0+h\boldsymbol z_1+r_h)\), with
\(M\boldsymbol z_h^-=0\), fixed amplitudes smooth near \(\overline\Omega_{\mathrm{CGO}}\), and
\(\|r_h\|_{H^5(\Omega_{\mathrm{CGO}})}=O(h^2)\), the normalized physical adjoint lift is
\(y_h^F=-hK\boldsymbol z_h^-\), with principal adjoint amplitude \(J_\theta \boldsymbol z_0\).
If \(J_\theta \boldsymbol z_0=0\), the unscaled physical adjoint lift is \(y_h^S=K\boldsymbol z_h^-\),
with additional adjoint amplitude \(K_0\boldsymbol z_0-J_\theta \boldsymbol z_1\).
The labels \(F\) and \(S\) distinguish the two principal adjoint
polarizations spanning \(\mathcal F_\theta=\operatorname{im}J_\theta\)
(see \eqref{eq:fast-space}) from the supplementary polarization
obtained when \(J_\theta\boldsymbol z_0^S=0\). They refer to the
asymptotic construction, not to elastic-wave propagation speeds.
Lemmas~\ref{lem:fast-physical-lift} and
\ref{lem:slow-physical-lift} establish the corresponding phase-removed
\(O(h)\) errors in \(H^3(\Omega_{\mathrm{CGO}})\).
\subsection{Conjugation of the constant exterior reduction}
The constant exterior reduction still contains a first-order coupling.
Removing it gives the compactly supported lower-order coefficients required
by the unweighted Sobolev estimates for the localized right inverse
in \eqref{eq:system-cgo-right-inverse}.
Assume that the smooth coefficient pair on \(\R^3\) equals
a constant pair \((\mu_*,\kappa_*)\) outside a large ball.
When comparing two media, we use their common exterior pair,
so the same background operator and conjugating matrix are
used for both. The background reduced operator is
\[
 M_*=
 \begin{pmatrix}
 \Delta\Id_3&0\\
 d_*\Div&\Delta
 \end{pmatrix},
 \qquad d_*=\kappa_*\sqrt{\mu_*}.
\]
Set
\begin{equation}\label{eq:Rbg}
 \mathcal R(x)=
 \begin{pmatrix}
 \Id_3&0\\
 -\dfrac{d_*}{2}x^T&1
 \end{pmatrix}.
\end{equation}
This matrix is invertible at every \(x\), with
\[
 \det\mathcal R=1,
 \qquad
 \mathcal R^{-1}(x)=
 \begin{pmatrix}
 \Id_3&0\\
 \dfrac{d_*}{2}x^T&1
 \end{pmatrix}.
\]
Both matrices and their derivatives are bounded on every
fixed ball. Multiplication by \(\mathcal R^{\pm1}\)
therefore preserves each fixed Sobolev order there.
Since
\[
 \Delta(x\cdot v)=x\cdot\Delta v+2\Div v,
\]
we have \(M_*\mathcal R(v,f)=\mathcal R(\Delta v,\Delta f)\)
for every smooth \((v,f)\). Hence
\begin{equation}\label{eq:bg-conjugation}
 \mathcal R^{-1}M_*\mathcal R=\Delta\Id_4.
\end{equation}
Consequently
\[
 \widetilde M=\mathcal R^{-1}M\mathcal R
 =\Delta\Id_4+\mathcal R^{-1}(M-M_*)\mathcal R.
\]
The operator \(M-M_*\) has order at most one and smooth
coefficients supported in the background ball. Composition
with \(\mathcal R\) and \(\mathcal R^{-1}\) preserves this
support. Thus \(\widetilde M\) has principal part
\(\Delta\Id_4\) and smooth compactly supported first- and
zeroth-order coefficients. With
\[
 \widetilde T=T\mathcal R,
 \qquad
 \widetilde S=S\mathcal R,
\]
where the products denote composition of differential
operators with multiplication by \(\mathcal R\), the
original factorization gives
\begin{equation}\label{eq:factorization-conjugated}
 L\widetilde T=\widetilde S\widetilde M.
\end{equation}

The compactly supported coefficients of \(\widetilde M\)
permit the system CGO construction. We also need to compare its
transport with the original reduction: transforming the primal and
dual frames together preserves both the quotient connection and the
coefficient-difference pairing.

\begin{proposition}\label{prop:bg-covariance}
Let \(\theta\in\C^3\setminus\{0\}\) satisfy
\(\theta\cdot\theta=0\), and let
\[
 \mathcal T_\theta=2D_\theta+A_\theta,
 \qquad
 \widetilde{\mathcal T}_\theta=2D_\theta+\widetilde A_\theta
\]
be the leading transport operators associated with \(M\) and \(\widetilde M=\mathcal R^{-1}M\mathcal R\).  Then
\begin{equation}\label{eq:transport-covariance}
 \widetilde{\mathcal T}_\theta
 =\mathcal R^{-1}\mathcal T_\theta\mathcal R.
\end{equation}
The characteristic line and the primal quotient are
transported by \(\mathcal R^{-1}\), and the dual quotient
by \(\mathcal R^T\). In particular, the original generator
\(e_\theta=(\theta,0)^T\) is replaced by
\[
 \widetilde e_\theta=\mathcal R^{-1}e_\theta
 =\binom{\theta}{\dfrac{d_*}{2}x\cdot\theta},
\]
whose line is preserved by \(\widetilde{\mathcal T}_\theta\).
If \(X\) and \(Y\) are ambient representatives of primal
and dual quotient frames for the original reduction, their
conjugated representatives may be chosen as
\[
 \widetilde X=\mathcal R^{-1}X,
 \qquad
 \widetilde Y=\mathcal R^TY.
\]
In these transformed frames, with their spatial derivatives
included, the coordinate quotient connection matrices are
unchanged. For two media conjugated by the same \(\mathcal R\),
\begin{equation}\label{eq:moment-covariance}
 \widetilde Y_2^T
 (\widetilde A_{1,\theta}-\widetilde A_{2,\theta})
 \widetilde X_1
 =
 Y_2^T(A_{1,\theta}-A_{2,\theta})X_1.
\end{equation}
\end{proposition}

\begin{proof}
The coefficient of the first-order derivative in \(\mathcal R^{-1}M\mathcal R\) is
\begin{equation}\label{eq:proof-conjugated-first-order-coefficients}
 \widetilde A_k
 =\mathcal R^{-1}A_k\mathcal R
 +2\mathcal R^{-1}(\partial_k\mathcal R).
\end{equation}
By \eqref{eq:proof-conjugated-first-order-coefficients},
for every smooth four-component amplitude \(w\),
\[
 (2D_\theta+\widetilde A_\theta)w
 =\mathcal R^{-1}(2D_\theta+A_\theta)(\mathcal Rw),
\]
which proves \eqref{eq:transport-covariance}. The companion
identity for a dual amplitude \(y\) is
\begin{equation}\label{eq:proof-dual-background-covariance}
 (2D_\theta-\widetilde A_\theta^T)(\mathcal R^Ty)
 =\mathcal R^T(2D_\theta-A_\theta^T)y.
\end{equation}
Equations~\eqref{eq:transport-covariance} and
\eqref{eq:proof-dual-background-covariance} transport the original characteristic line
and the paired quotient equations to the conjugated system.
The paired transformations in Proposition~\ref{prop:bg-covariance} also preserve the bilinear pairing, since
\((\mathcal R^Ty)^T(\mathcal R^{-1}w)=y^Tw\).

To verify the coordinate connection statement explicitly,
we choose \(\mathsf U_\theta\) as a dual quotient basis and
\(\mathsf V_\theta\) as a paired primal lift, both independent
of \(x\), with \(\mathsf U_\theta^T\mathsf V_\theta=I_3\).
Set \(\widetilde{\mathsf U}_\theta=\mathcal R^T\mathsf U_\theta\)
and \(\widetilde{\mathsf V}_\theta=\mathcal R^{-1}\mathsf V_\theta\).
The induced connection in the paired frames
\(\widetilde{\mathsf U}_\theta\) and
\(\widetilde{\mathsf V}_\theta\), which now depend on \(x\), is
\[
 \widetilde{\mathsf U}_\theta^T
 \bigl(2D_\theta\widetilde{\mathsf V}_\theta
       +\widetilde A_\theta\widetilde{\mathsf V}_\theta\bigr)
 =\mathsf U_\theta^TA_\theta\mathsf V_\theta.
\]
The frame derivative on the left is essential;
\eqref{eq:transport-covariance} gives the equality including this term.

Finally, we use the same \(\mathcal R\) for both media to cancel
the derivative terms and obtain
\begin{equation}\label{eq:proof-background-coefficient-difference}
 \widetilde A_{1,\theta}-\widetilde A_{2,\theta}
 =\mathcal R^{-1}(A_{1,\theta}-A_{2,\theta})\mathcal R.
\end{equation}
Use \eqref{eq:proof-background-coefficient-difference} with
\(\widetilde X_1=\mathcal R^{-1}X_1\) and
\(\widetilde Y_2=\mathcal R^TY_2\) to prove
\eqref{eq:moment-covariance}.
\end{proof}

Proposition~\ref{prop:bg-covariance} lets us use
\(\widetilde M\) for the analytic construction while retaining the
original quotient coordinates in the coefficient comparison.
In particular, \eqref{eq:moment-covariance} identifies the leading
moments after the same background conjugation is applied to both media.

\subsection{The dual quotient connection}

The null-phase transport preserves a one-dimensional
characteristic subspace of \(\C^4\). We pass to its rank-three quotient
and express the dual transport in a spinor frame. This gives the
connection matrix \(\Gamma_{\mu,\kappa}\) in \eqref{eq:Gamma}.

Fix \(\theta\in\C^3\setminus\{0\}\) with
\(\theta\cdot\theta=0\). Contracting the first-order
coefficients in \eqref{eq:Ak}, we put
\(A_\theta=\sum_{k=1}^3\theta_kA_k\). Thus
\begin{equation}\label{eq:Atheta}
 A_\theta=
 \begin{pmatrix}
 -m\theta^T&F_\mu\theta\\
 \kappa\sqrt\mu\,\theta^T&0
 \end{pmatrix}.
\end{equation}
The vector
\[
 e_\theta=\binom\theta0
\]
is independent of \(x\). The null condition gives
\(A_\theta e_\theta=0\), and hence
\((2D_\theta+A_\theta)e_\theta=0\).
Its line is therefore preserved by the leading transport.
The differences \(A_{1,\theta}-A_{2,\theta}\) also annihilate
this line. These properties allow the transport and the
coefficient-difference moments to descend to the physical
quotient \(\C^4/\C e_\theta\).

For homogeneous spinor coordinates
\((s,t)\in\C^2\setminus\{(0,0)\}\), the map
\([s:t]\mapsto[\theta(s,t)]\) parametrizes the projective
complex null cone, with representative
\begin{equation}\label{eq:theta-spinor}
 \theta(s,t)=
 \bigl(s^2-t^2,\;i(s^2+t^2),\;2st\bigr)^T.
\end{equation}
Direct substitution gives \(\theta\cdot\theta=0\), and
\(\theta_1-\mathrm i\theta_2=2s^2\),
\(\theta_1+\mathrm i\theta_2=-2t^2\) show that
\(\theta\ne0\). Rescaling \((s,t)\) by a nonzero complex
number rescales \(\theta\) by its square, so its projective
direction depends only on \([s:t]\).
The parametrization covers all null directions: when
\(\theta_1-\mathrm i\theta_2\ne0\), the spinor coordinate is
\(t/s=\theta_3/(\theta_1-\mathrm i\theta_2)\); the remaining
projective direction is represented by \(s=0\).
Here \(s,t\) denote complex spinor coordinates; the real
transverse coordinate used later is a spatial leaf parameter.
Set
\begin{equation}\label{eq:qspinor}
 q_1=\frac12(-t,it,s)^T,
 \qquad
 q_2=-\frac12(s,is,t)^T,
 \qquad
 Q=(q_1,q_2),
 \qquad
 n=\binom{-t}{s}.
\end{equation}
Direct calculation gives \(\theta^TQ=0\).
Moreover, \(Q(a,b)^T=0\) implies
\(ta=sb=sa-tb=0\), which forces \(a=b=0\) for every
nonzero spinor, including \(s=0\) or \(t=0\).
Thus \(Q\) has rank two and its columns form a basis of
\(\theta^\perp=\{p\in\C^3:\theta^Tp=0\}\).
The same calculation gives
\begin{equation}\label{eq:Qn}
 Qn=-\frac12\theta.
\end{equation}
For \(a\in\C^3\), we define
\begin{equation}\label{eq:rmathfrak}
 \mathfrak r_a=-2a^TQ.
\end{equation}
Then \(\mathfrak r_a n=a\cdot\theta\).

With respect to the complex-bilinear pairing, the dual
physical characteristic space is the annihilator of
\(\C e_\theta\):
\[
 E_\theta^*=(\C^4/\C e_\theta)^*
 =\{(p,\eta)\in\C^3\times\C:\theta^Tp=0\}.
\]
It has dimension three. Since \(A_\theta e_\theta=0\),
the matrix \(A_\theta^T\) preserves this annihilator,
so the dual leading transport restricts to \(E_\theta^*\).
Every \(p\in\theta^\perp\) has a unique representation
\(p=Q\xi\), with \(\xi\in\C^2\).

The spinor representative is held fixed when differentiating in \(x\),
so \(D_\theta Q=0\). For \(y=(p,\eta)^T\), the dual
transport equation \(2D_\theta y-A_\theta^Ty=0\) reads
\begin{equation}\label{eq:proof-ambient-dual-transport}
 \begin{aligned}
 2D_\theta p+\theta(m^Tp-d_{\mu,\kappa}\eta)&=0,\\
 2D_\theta\eta-\theta^TF_\mu p&=0.
 \end{aligned}
\end{equation}
Substituting \(p=Q\xi\) into \eqref{eq:proof-ambient-dual-transport},
using \(\theta=-2Qn\) from \eqref{eq:Qn} and
\(\mathfrak r_m=-2m^TQ\) from \eqref{eq:rmathfrak}, and cancelling the injective
matrix \(Q\) in the first equation gives
\begin{equation}\label{eq:dual-quotient-transport}
 2D_\theta\binom\xi\eta+\Gamma_{\mu,\kappa}(\theta)\binom\xi\eta=0,
\end{equation}
where
\begin{equation}\label{eq:Gamma}
 \Gamma_{\mu,\kappa}(\theta)=
 \begin{pmatrix}
 n\mathfrak r_m&2\kappa\sqrt\mu\,n\\[1mm]
 -\theta^TF_\mu Q&0
 \end{pmatrix}.
\end{equation}
The matrix \(\Gamma_{\mu,\kappa}(\theta)\) in
\eqref{eq:Gamma} represents the rank-three dual physical quotient
connection in the chosen spinor frame. Its dependence on \(\kappa\)
is in the upper-right block; for fixed direction, the remaining
nonzero blocks depend only on \(\mu\). The frame transition and projective covariance laws
are given in \eqref{eq:spinor-transition} and
\eqref{eq:connection-chart-scaling}. These laws will allow the local
connection matrices to be compared over all projective null directions.
\subsection{The primal quotient and local transport lifts}\label{subsec:quotient and transport}
We choose primal coordinates paired with the dual spinor
frame. This identifies the connection entering the bilinear volume
identity \eqref{eq:primal-adjoint} and gives a way to lift quotient
transport amplitudes to the
four-component transport equation. The lift may vary along the
characteristic line without changing the quotient coordinates.
The primal characteristic quotient is
\[
 E_\theta=\C^4/\C e_\theta,
 \qquad e_\theta=\binom\theta0.
\]
In a fixed spinor chart, we put
\[
 \mathsf U_\theta=
 \begin{pmatrix}Q&0\\0&1\end{pmatrix}:\C^3\longrightarrow(\C^4)^*.
\]
The matrix \(\mathsf U_\theta\) has size \(4\times3\)
and rank three. Its image is \(E_\theta^*\), and
\(\ker\mathsf U_\theta^T=\C e_\theta\).
Thus a class \([\boldsymbol z]\in E_\theta\) has well-defined coordinates
\(u=\mathsf U_\theta^T\boldsymbol z\in\C^3\).
For dual coordinates \(v\in\C^3\), the ambient representative
\(y=\mathsf U_\theta v\) satisfies
\[
 y^T\boldsymbol z=v^Tu.
\]
This is the pairing of the primal quotient and its dual.
On the overlap of the standard spinor charts
\[
U_0=\{s\neq 0\},
\qquad
U_\infty=\{t\neq 0\},
\]
we set
\[
\zeta=\frac{t}{s},
\qquad
D(\zeta)=
\begin{pmatrix}
\zeta I_2 & 0\\
0 & 1
\end{pmatrix}.
\]
Use the spinor representatives \((1,\zeta)\) and
\((\zeta^{-1},1)\) in these charts. Since
\(Q(1,\zeta)=\zeta Q(\zeta^{-1},1)\), the corresponding
dual quotient matrices, denoted by \(\mathsf U_0\) and
\(\mathsf U_\infty\), satisfy
\begin{equation}\label{eq:proof-dual-frame-transition}
 \mathsf U_0=\mathsf U_\infty D(\zeta).
\end{equation}
Here \(U_0,U_\infty\) denote open charts, whereas
\(\mathsf U_0,\mathsf U_\infty\) denote their frame matrices.

We choose a local holomorphic right inverse
\(\mathsf V_\theta:\C^3\to\C^4\) of size \(4\times3\), with
\(\mathsf U_\theta^T\mathsf V_\theta=I_3\), depending only
on the spinor coordinate. Its existence is proved in
Lemma~\ref{lem:holomorphic-lifts} below.
In particular, both \(\mathsf U_\theta\) and
\(\mathsf V_\theta\) are independent of \(x\) and of the
coefficients of the medium. The induced primal connection is
\begin{equation}\label{eq:primal-quotient-connection}
 \mathcal A_\theta=\mathsf U_\theta^TA_\theta\mathsf V_\theta.
\end{equation}
The matrix \(\mathcal A_\theta\) is independent of the replacement
\(\mathsf V_\theta\mapsto\mathsf V_\theta+e_\theta\mathsf L\),
where \(\mathsf L\) is a \(1\times3\) row, because
\(A_\theta e_\theta=0\). For every smooth ambient amplitude
\(\boldsymbol z\), the quotient projection intertwines the transports:
\begin{equation}\label{eq:proof-quotient-transport-projection}
 \mathsf U_\theta^T(2D_\theta+A_\theta)\boldsymbol z
 =(2D_\theta+\mathcal A_\theta)(\mathsf U_\theta^T\boldsymbol z).
\end{equation}
Indeed, \(\boldsymbol z-\mathsf V_\theta\mathsf U_\theta^T\boldsymbol z\) lies in
\(\C e_\theta\), and \(D_\theta\mathsf U_\theta=0\).

Equations~\eqref{eq:dual-quotient-transport} and \eqref{eq:Gamma} give
\(-A_\theta^T\mathsf U_\theta=\mathsf U_\theta\Gamma_\theta\).
Multiplication by \(\mathsf V_\theta^T\) therefore yields
\[
 \Gamma_\theta=-\mathsf V_\theta^TA_\theta^T\mathsf U_\theta
 =-\mathcal A_\theta^T.
\]

The coordinate description requires right inverses that
vary holomorphically with the spinor coordinate. Their differences on
a chart overlap are confined to the characteristic line, as follows
from the kernel of \(\mathsf U_\theta^T\).

\begin{lemma}\label{lem:holomorphic-lifts}
On every sufficiently small spinor-coordinate neighborhood,
the right inverse \(\mathsf V_\theta\) may be chosen
holomorphically in that coordinate and independently of
\(x\) and the medium. On a chart overlap, any such local
choices satisfy
\[
 \mathsf V_0=\mathsf V_\infty D^{-T}
             +e_{\theta_0}\mathsf L
\]
locally for a holomorphic \(1\times3\) row \(\mathsf L\).
\end{lemma}

\begin{proof}
The entries of \(\mathsf U_\theta^T\) are holomorphic
and the matrix has full row rank three. Locally, we choose a
nonvanishing \(3\times3\) minor, invert it, and set the
remaining ambient coordinate to zero. We thereby construct a
holomorphic right inverse with no dependence on \(x\) or
the coefficients.
On an overlap, \eqref{eq:proof-dual-frame-transition} gives
\[
 \mathsf U_0^T(\mathsf V_\infty D^{-T})
 =D^T\mathsf U_\infty^T\mathsf V_\infty D^{-T}=I_3.
\]
Thus the difference of the two right inverses has image in
\(\ker\mathsf U_0^T=\C e_{\theta_0}\).
Since \(e_{\theta_0}\) is a nonvanishing holomorphic
generator, that difference is locally \(e_{\theta_0}\mathsf L\)
with \(\mathsf L\) holomorphic.
\end{proof}

Lemma~\ref{lem:holomorphic-lifts} controls the freedom in
the choice of \(\mathsf V_\theta\), but the representative
\(\mathsf V_\theta X\) need not solve the full transport equation.
Its residual lies in \(\C e_\theta\) by
\eqref{eq:proof-quotient-transport-projection}. A scalar leafwise
correction removes that residual without changing \(X\).

\begin{lemma}\label{lem:transport-lift}
Let \(X\) be a smooth \(3\times r\) matrix amplitude
satisfying
\[
 (2D_\theta+\mathcal A_\theta)X=0.
\]
On every sufficiently small leafwise product neighborhood,
there is a smooth \(1\times r\) row \(\alpha\) such that
\[
 Z=\mathsf V_\theta X+e_\theta\alpha,
 \qquad
 (2D_\theta+A_\theta)Z=0,
 \qquad
 \mathsf U_\theta^TZ=X.
\]
For a single column, \(r=1\) and \(\alpha\) is a scalar;
a quotient frame corresponds to \(r=3\).
If \(X\), the quotient basis, and the direction depend
smoothly on auxiliary parameters, \(\alpha\) and \(Z\)
may be chosen locally smoothly in those parameters.
\end{lemma}

\begin{proof}
Equation~\eqref{eq:proof-quotient-transport-projection} gives
\[
 \mathsf U_\theta^T(2D_\theta+A_\theta)\mathsf V_\theta X
 =(2D_\theta+\mathcal A_\theta)X=0.
\]
The residual therefore lies columnwise in
\(\ker\mathsf U_\theta^T=\C e_\theta\). On a small direction
neighborhood, we choose a smooth row \(\ell_\theta\) with
\(\ell_\theta e_\theta=1\). Then the smooth \(1\times r\) row
\(f=\ell_\theta(2D_\theta+A_\theta)\mathsf V_\theta X\)
satisfies
\begin{equation}\label{eq:proof-characteristic-transport-residual}
 (2D_\theta+A_\theta)\mathsf V_\theta X=e_\theta f.
\end{equation}
By \eqref{eq:proof-characteristic-transport-residual} and
\((2D_\theta+A_\theta)e_\theta=0\), it remains to
solve \(2D_\theta\alpha=-f\).

For the normalization, we choose a locally smooth decomposition
\(\theta=b\theta_{\mathrm n}\), where \(b\ne0\) is independent
of \(x\), and \(\theta_{\mathrm n}=e_1+\mathrm i e_2\)
with real orthonormal \(e_1,e_2\).
In the leaf coordinate \(z=e_1\cdot x+\mathrm i e_2\cdot x\),
\(D_\theta=2b\partial_{\bar z}\).
Pull back to fixed disks \(D\Subset D_1\) contained in the
neighborhood where \(X\) is defined, and choose
\(\chi\in C_c^\infty(D_1)\) equal to one near
\(\overline D\). With \(x_\perp\) the transverse coordinate, we set
\begin{equation}\label{eq:proof-characteristic-cauchy-correction}
 \alpha(z,x_\perp)
 =-\frac{1}{4\pi b}\int_{D_1}
       \frac{\chi(w)f(w,x_\perp)}{z-w}\,dA(w).
\end{equation}
The identity \(\partial_{\bar z}(\pi^{-1}z^{-1})=\delta_0\)
applied to \eqref{eq:proof-characteristic-cauchy-correction}
gives \(2D_\theta\alpha=-f\) on \(D\).
The Cauchy solution is spatially smooth for the smooth
cutoff source. Locally in the auxiliary parameters, the
disks and cutoff may be fixed, while \(b^{-1}\),
\(\ell_\theta\), and the pulled-back source vary smoothly.
Parameter derivatives commute with the fixed integral,
which gives the asserted local smooth dependence.
Finally, \(\mathsf U_\theta^T\mathsf V_\theta=I_3\) and
\(\mathsf U_\theta^Te_\theta=0\) give
\(\mathsf U_\theta^TZ=X\).
\end{proof}

Lemma~\ref{lem:transport-lift} supplies, locally on a
leafwise product neighborhood, a full leading transport amplitude for
each quotient column. The characteristic-line
correction leaves its quotient pairing unchanged. The next section
places the leading amplitudes on the fixed CGO domains, constructs
exact reduced fields, and then applies the physical adjoint lift.

\section{Physical CGO solutions and quotient frames}\label{sec:cgo}

\subsection{System estimates and physical lifts}

The CGO construction relies on an inhomogeneous solvability estimate for the full matrix
system on fixed nested domains. The matrix transport and
pseudodifferential construction originates with Eskin
\cite{Eskin2001} and is used in
\cite{EskinRalston2002,EskinRalstonSystems2004}.
Appendix~\ref{audrev-ordinary-parametrix} gives the free weighted
inverse, a finite family of matrix transport frames on fixed large
disks, and the unweighted Sobolev estimates needed here. Its \(k=5\) specialization, after
the first transport correction, gives an \(O(h^2)\) reduced
remainder in \(H^5\), which gives an \(O(h)\) remainder in \(H^3\)
after the unscaled physical adjoint lift.
Theorem~\ref{lem:ER-solvability} states this estimate in our phase
convention and with compact-direction uniformity.

The localized transport recursion in
Theorem~\ref{thm:autonomous-cgo}, the physical lift estimates, and the
rank-three realization in Theorem~\ref{thm:graded-frame} are proved below.
Their boundary-data consequence is Theorem~\ref{thm:physical-nine-moments}.
Throughout the CGO construction, \(\Omega_{\mathrm{CGO}}\) denotes the
fixed working ball; \(\Omega\) denotes the original measurement domain.
In the general system statement, \(m\) denotes the number of
components and \(K_{\mathrm{src}}\) denotes the closed source ball. These local uses
are distinct from \(m=\nabla\log\mu\) and the physical lift
\(K_{\mu,\kappa}=S_{\mu,\kappa}^TT_\mu\) in the Lam\'e reduction.

\begin{theorem}
\label{lem:ER-solvability}
Let \(m\) be a positive integer and let
\[
 \widetilde M=\Delta\Id_m
              +\sum_{j=1}^3\widetilde A_j(x)\partial_j
              +\widetilde B(x),
 \qquad \widetilde A_j,\widetilde B
        \in C_c^\infty(\R^3;M_m(\C)).
\]
Write \(\widetilde A=(\widetilde A_1,\widetilde A_2,\widetilde A_3)\)
and \(\widetilde A_\theta=\sum_j\theta_j\widetilde A_j\),
and fix concentric balls \(\overline\Omega_{\mathrm{CGO}}\Subset\Omega_1\Subset\Omega_0\).
Here and below, a normalized complex null direction means
\(\theta=e_1+\mathrm i e_2\), where \(e_1,e_2\in\mathbb R^3\) are
orthonormal.  In particular \(\theta\ne0\) and \(\theta\cdot\theta=0\).
Choose a closed concentric ball \(K_{\mathrm{src}}\Subset\Omega_1\) whose interior
contains \(\overline\Omega_{\mathrm{CGO}}\) and the coefficient support, and let
\(\Theta\) be a compact family of such directions.  For
any fixed integer \(k\geq0\), there are \(h_0=h_0(k)>0\) and linear
operators \(R_{h,\theta}^{(k)}\), defined for \(0<h<h_0\), such that, for
\(f\in H^k(\Omega_0;\C^m)\) with \(\supp f\subset K_{\mathrm{src}}\),
\begin{equation}\label{eq:system-cgo-right-inverse-equation}
 e^{\theta\cdot x/h}\widetilde M e^{-\theta\cdot x/h}
 R_{h,\theta}^{(k)}f=f\quad\hbox{in }\Omega_1
\end{equation}
and
\begin{equation}\label{eq:system-cgo-right-inverse}
 h^{-1}\|R_{h,\theta}^{(k)}f\|_{H^k(\Omega_0)}+
 \|R_{h,\theta}^{(k)}f\|_{H^{k+1}(\Omega_0)}
 \leq C_k\|f\|_{H^k(\Omega_0)}.
\end{equation}
The constants are uniform in \(\theta\in\Theta\), in the sense specified
in Remark~\ref{rem:CGO-uniformity}.
\end{theorem}

\begin{proof}

The argument uses the localized right inverse constructed in
Proposition~\ref{prop:appendix-interior-right-inverse} of
Appendix~\ref{audrev-ordinary-parametrix}. Its construction
relies on the free weighted inverse of
Lemma~\ref{lem:appendix-free-inverse} and the localized
parametrix of Lemma~\ref{lem:appendix-local-error}.

To check conventions, we multiply \(\widetilde M\) by \(-1\)
and set
\[
 A^{\mathrm{ER}}_j=\frac{\widetilde A_j}{2\mathrm i},\qquad
 B^{\mathrm{ER}}=-\widetilde B,\qquad
 \theta_{\mathrm{ER}}=\mathrm i\theta,\qquad \tau=h^{-1}.
\]
The operator in \cite[equation (1.1)]{EskinRalstonSystems2004} is \(-\widetilde M\).
We take the auxiliary Fourier vector in \cite{EskinRalstonSystems2004} to be \(l=0\), which gives
\(\delta=\tau\theta_{\mathrm{ER}}\), so the phase
\(e^{\mathrm i\delta\cdot x}\) is \(e^{-\theta\cdot x/h}\).
Equation (1.6) of \cite{EskinRalstonSystems2004} becomes
\[
 (2D_\theta+\widetilde A_\theta)C=0.
\]
The conjugated operator in this convention is the negative of
\(e^{\theta\cdot x/h}\widetilde M e^{-\theta\cdot x/h}\),
so its source is \(-f\). The inverse constructed in
Appendix~\ref{audrev-ordinary-parametrix} gives
\eqref{eq:system-cgo-right-inverse-equation}.

For sources supported in \(K_{\mathrm{src}}\), the
localized right-inverse identity holds in \(\Omega_1\),
while the estimates are taken over \(\Omega_0\).
Both Sobolev bounds in \eqref{eq:system-cgo-right-inverse},
at orders \(k\) and \(k+1\), are obtained from the same
Neumann-series inverse acting on \(H^k\).
Uniformity over compact families of normalized null
directions follows from the finite covering by local
transport frames. The rotation argument below provides
an alternative verification of this uniformity.

Translate the domains and coefficients together so that
the common center of the balls is the origin.
Fix a reference normalized null direction
\(\theta_{\mathrm{ref}}\). Locally in \(\theta\), we choose a
smooth rotation \(O_\theta\in SO(3)\) with
\(O_\theta\theta_{\mathrm{ref}}=\theta\), and define the
pullback \((\mathcal P_\theta f)(y)=f(O_\theta y)\).
The balls and \(K_{\mathrm{src}}\) are preserved. The direction becomes
\(\theta_{\mathrm{ref}}\), and the coefficients become
\[
 \mathsf A_{\theta,j}(y)
 =\sum_l\widetilde A_l(O_\theta y)(O_\theta)_{lj},
 \qquad \mathsf B_\theta(y)=\widetilde B(O_\theta y).
\]
The pulled-back conjugated operator is
\[
 L_{h,\theta}=\Delta-2h^{-1}D_{\theta_{\mathrm{ref}}}
   +\mathsf A_\theta\cdot\nabla
   -h^{-1}\mathsf A_{\theta,\theta_{\mathrm{ref}}}
   +\mathsf B_\theta,
 \qquad
 \mathsf A_{\theta,\theta_{\mathrm{ref}}}
 =\sum_j(\theta_{\mathrm{ref}})_j\mathsf A_{\theta,j}.
\]
On the closed source space
\(\mathcal E_k=\{f\in H^k(\Omega_0;\C^m):\supp f\subset K_{\mathrm{src}}\}\),
choose a localized right inverse
\(\mathsf R_{h,\theta^0}\) for the operator at a base
parameter \(\theta^0\).
Its target norm is
\begin{equation}\label{eq:proof-right-inverse-target-norm}
 \|u\|_{\mathcal X_h^k}
 =h^{-1}\|u\|_{H^k(\Omega_0)}+\|u\|_{H^{k+1}(\Omega_0)}.
\end{equation}
By \eqref{eq:system-cgo-right-inverse} and
\eqref{eq:proof-right-inverse-target-norm}, the norm of
\(\mathsf R_{h,\theta^0}\) is bounded independently of \(h\). Since the principal direction is fixed after
rotation, \(D_{h,\theta}=L_{h,\theta}-L_{h,\theta^0}\)
has coefficients supported in \(K_{\mathrm{src}}\), and
\begin{equation}\label{eq:proof-cgo-operator-perturbation}
 \|D_{h,\theta}u\|_{H^k}
 \leq C\bigl(\|\mathsf A_\theta-\mathsf A_{\theta^0}\|_{C^{k+1}}
           +\|\mathsf B_\theta-\mathsf B_{\theta^0}\|_{C^{k+1}}\bigr)
          \|u\|_{\mathcal X_h^k},\qquad 0<h\leq1.
\end{equation}
In particular, the factor \(h^{-1}\) in \(D_{h,\theta}\) is controlled by
\(h^{-1}\|u\|_{H^k}\) in \eqref{eq:proof-right-inverse-target-norm}.
By \eqref{eq:proof-cgo-operator-perturbation} and
\eqref{eq:system-cgo-right-inverse}, \(Q_{h,\theta}=D_{h,\theta}\mathsf R_{h,\theta^0}\)
maps \(\mathcal E_k\) to itself and has norm at most
\(1/2\) on a direction neighborhood independent of \(h\).
Set
\begin{equation}\label{eq:proof-cgo-perturbed-right-inverse}
 \mathsf R_{h,\theta}
 =\mathsf R_{h,\theta^0}(I+Q_{h,\theta})^{-1}.
\end{equation}
Equation~\eqref{eq:system-cgo-right-inverse-equation}, applied at the
base parameter in the rotated coordinates, holds throughout \(\Omega_1\),
which contains the support of \(D_{h,\theta}\). Using
\eqref{eq:system-cgo-right-inverse-equation} and
\eqref{eq:proof-cgo-perturbed-right-inverse}, we obtain
\[
 L_{h,\theta}\mathsf R_{h,\theta}f
 =(I+Q_{h,\theta})(I+Q_{h,\theta})^{-1}f=f
 \quad\text{in }\Omega_1.
\]
Its norm into \(\mathcal X_h^k\) is at most twice the base
norm. We return to the original coordinates by setting
\[
 R_{h,\theta}^{(k)}
 =\mathcal P_\theta^{-1}\mathsf R_{h,\theta}\mathcal P_\theta.
\]
Rotations give uniformly equivalent Sobolev norms on the
fixed balls. We use a finite cover of \(\Theta\) to obtain a common
\(h_0(k)\leq1\) and \(C_k\). We combine these local operators with
a parameter partition of unity after they have been returned
to the original coordinates. Its weights are independent
of \(x\), so the right-inverse equation is preserved.
No parameter derivative of this family is asserted or needed.
The construction uses the exterior extension to \(\R^3\),
rather than a bounded-domain Dirichlet inverse.
\end{proof}

The factor \(h\) in
\eqref{eq:system-cgo-right-inverse} will turn a residual of size
\(O(h)\) into a reduced remainder of size \(O(h^2)\).
Before applying \eqref{eq:system-cgo-right-inverse}, we solve the first transport correction
on the complex leaves. The following Cauchy estimate keeps the leaf
and transverse derivatives separate.

\begin{lemma}
\label{lem:cauchy-leaf-estimate}
Let \(D\Subset D_1\Subset\C\) be disks, \(I\Subset\R\) an interval,
and \(0<\alpha<1\).  In coordinates in which
\(D_\theta=2\partial_{\bar z}\), set
\[
 (\mathcal C_\theta F)(z,x_\perp)=\frac1{2\pi}\int_{D_1}
                   \frac{F(w,x_\perp)}{z-w}\,dA(w).
\]
If \(F\) is smooth and supported away from \(\partial D_1\times I\), then
\[
 D_\theta\mathcal C_\theta F=F\quad\hbox{on }D\times I.
\]
Moreover, for all nonnegative integers \(r,q\),
\begin{equation}\label{eq:anisotropic-cauchy}
 \|\mathcal C_\theta F\|_{C_{x_\perp}^qC_z^{r+1,\alpha}(D\times I)}
 \leq C_{r,q}\|F\|_{C_{x_\perp}^qC_z^{r,\alpha}(D_1\times I)}.
\end{equation}
Here \(C_{x_\perp}^qC_z^{r,\alpha}\) denotes \(q\) derivatives in the transverse
variable with values in the indicated leafwise Hölder space.  The estimate is
uniform on a compact directional chart after pulling the leaf coordinates
back by a smooth family of rotations.
\end{lemma}

\begin{proof}
The distributional identity
\(\partial_{\bar z}(\pi^{-1}z^{-1})=\delta_0\) proves the first assertion
with the factor \(1/(2\pi)\) in the definition of \(\mathcal C_\theta\).  For every fixed \(x_\perp\), the local
Schauder estimate for the Cauchy solution operator gives one derivative in
\(z\).  Derivatives in \(x_\perp\) commute with the integral, which proves
\eqref{eq:anisotropic-cauchy}.  The rotations and their derivatives are
bounded on a compact directional chart.
\end{proof}

The estimate \eqref{eq:anisotropic-cauchy} gains one
leaf derivative and preserves the number of transverse derivatives.
Applied in transport frames, it controls the first correction.
The residual is then cut off inside the source ball, and
Theorem~\ref{lem:ER-solvability} is applied to make the reduced field exact.

\begin{theorem}
\label{thm:autonomous-cgo}
Let \(\widetilde M\), the fixed domains, \(K_{\mathrm{src}}\), and \(\Theta\) be as in
Theorem~\ref{lem:ER-solvability}.  Fix a nonnegative integer \(k\), and put
\[
 P_0=\widetilde M,\qquad T_\theta=2D_\theta+\widetilde A_\theta.
\]
Suppose that \(a_0\) is smooth near \(\overline\Omega_1\) and satisfies
\(T_\theta a_0=0\) there.  There exist a smooth first correction \(a_1\)
and an exact solution
\begin{equation}\label{eq:localized-cgo-package}
 \boldsymbol z_h=e^{-\theta\cdot x/h}(a_0+h a_1+r_h),\qquad
 \widetilde M\boldsymbol z_h=0\quad\text{in }\Omega_{\mathrm{CGO}},
\end{equation}
such that
\begin{equation}\label{eq:localized-cgo-remainder}
 T_\theta a_1=P_0a_0\quad\text{near }\overline\Omega_{\mathrm{CGO}},\qquad
 \|r_h\|_{H^k(\Omega_{\mathrm{CGO}})}\leq C_kh^2.
\end{equation}
Any prescribed smooth \(a_1\), defined near \(\overline\Omega_{\mathrm{CGO}}\) and
satisfying this transport equation there, may instead be used.

The remainder bound \eqref{eq:localized-cgo-remainder} holds uniformly on a compact direction chart if the cutoffs are
fixed, the transport frames and their inverses are uniformly bounded, and
\(\|\chi_0P_0a_1\|_{H^k(\Omega_0)}\) is uniformly bounded.
For the constructed correction, a sufficient condition is bounded
\(C_{x_\perp}^{k+4}C_z^{k+3,\alpha}\) norms of the coefficients, \(a_0\), and the
chosen frames, their inverses, and the fixed partition weights
on the fixed cylinder, for \(0<\alpha<1\).
For a prescribed correction its corresponding \(H^{k+2}\) bound on the
fixed cutoff support must also be assumed.  The application below takes
\(k=5\); thus \(C_{x_\perp}^9C_z^{8,\alpha}\) bounds suffice for the constructed
correction. For any constant vector \(v\in\C^m\), smooth
amplitudes \(a_0=C_\theta v\,b(x_\perp)g(z)\) are admissible,
where \(C_\theta\) is a transport frame on a fixed large
cylinder \(D_L\times I\), with \(L\) larger than the radius of
\(\Omega_1\), and \(b\in C_c^\infty(I)\). The product is extended
smoothly by zero outside \(I\) in the transverse variable.
Let \(U\) be a fixed neighborhood of \(\overline\Omega_1\)
whose slices with $x_\perp\in\supp b$ lie in the fixed transport cylinder $D_L\times I$.
The function \(g\) must be holomorphic on a neighborhood of
the entire planar projection of \(U\cap\{x_\perp\in\supp b\}\).
This gives a smooth leading amplitude on \(U\) satisfying its
transport equation, without requiring a single frame at transverse
coordinates where \(b=0\). Holomorphy only over the smaller working ball does not ensure holomorphy on the entire planar projection of $U\cap\{x_\perp\in\supp b\}$.
Every holomorphic polynomial meets the required domain condition.
\end{theorem}

\begin{proof}
The exact conjugation identity is
\begin{equation}\label{eq:conjugated-cgo-operator}
 e^{\theta\cdot x/h}\widetilde M e^{-\theta\cdot x/h}
 =P_0-h^{-1}T_\theta.
\end{equation}

We choose a fixed cylinder containing \(\overline\Omega_1\), enlarged
for the planar cutoff. The leafwise transport inverse
\(\mathcal V_\theta\) constructed in
Appendix~\ref{audrev-ordinary-parametrix} uses frames on this entire
large leaf disk and a finite cover of the required transverse
interval.
Choose \(\chi\in C_c^\infty(\Omega_1)\), equal to one
near \(\overline\Omega_{\mathrm{CGO}}\), and extend
\(\chi P_0a_0\) by zero to the cylinder. Set
\begin{equation}\label{eq:proof-cgo-first-correction}
 a_1=\mathcal V_\theta(\chi P_0a_0).
\end{equation}
The transport right-inverse identity
\eqref{audrev-transport-inverse-identity}, where $\chi_{\rm leaf}=1$
on the target region, and Lemma~\ref{lem:cauchy-leaf-estimate},
applied to the finitely many summands of $\mathcal V_\theta$, give
\(T_\theta a_1=P_0a_0\) where \(\chi=1\).
The fixed transverse weights and their required derivatives are
included in the constants beside the finitely many frame and
inverse seminorms.

With the stated finite seminorm bounds, \(P_0a_0\) and
the cutoff source are bounded in
\(C_{x_\perp}^{k+2}C_z^{k+1,\alpha}\). The Cauchy inverse gains
one leaf derivative, so \(a_1\) is bounded in
\(C_{x_\perp}^{k+2}C_z^{k+2,\alpha}\).
A second application of \(P_0\) gives
\(\chi_0P_0a_1\in C_{x_\perp}^kC_z^{k,\alpha}\), hence a bounded
\(H^k\) norm on the fixed cutoff support used below.

Choose \(\chi_0\in C_c^\infty(\operatorname{int}K_{\mathrm{src}})\)
equal to one near \(\overline\Omega_{\mathrm{CGO}}\), with support where
\(T_\theta a_1=P_0a_0\) holds. For a prescribed \(a_1\),
choose \(\supp\chi_0\) inside a neighborhood where that
correction and its transport equation are defined.
The product \(\chi_0P_0a_1\) extends smoothly by zero
to \(\Omega_0\) and is supported in \(K_{\mathrm{src}}\).  Since
\begin{equation}\label{eq:proof-cgo-uncorrected-residual}
 (P_0-h^{-1}T_\theta)(a_0+ha_1)=hP_0a_1
 \quad\text{near }\overline\Omega_{\mathrm{CGO}},
\end{equation}
the correction
\[
 r_h=-R_{h,\theta}^{(k)}(h\chi_0P_0a_1)
\]
cancels \eqref{eq:proof-cgo-uncorrected-residual} by
\eqref{eq:conjugated-cgo-operator} and Theorem~\ref{lem:ER-solvability},
makes \eqref{eq:localized-cgo-package} exact in \(\Omega_{\mathrm{CGO}}\), and obeys
\[
 \|r_h\|_{H^k(\Omega_{\mathrm{CGO}})}
 \leq C_kh\,\|h\chi_0P_0a_1\|_{H^k(\Omega_0)}
 \leq C_kh^2.
\]
For each fixed \(h\), local elliptic regularity makes the exact field
smooth in \(\Omega_{\mathrm{CGO}}\), because the coefficients and source are smooth.
The uniform bound for
\(\|\chi_0P_0a_1\|_{H^k(\Omega_0)}\), together with
\eqref{eq:system-cgo-right-inverse}, gives the asserted uniform
\(H^k\) remainder estimate. Multiplication of a leading transport
amplitude by \(b(x_\perp)g(z)\) preserves \(T_\theta a_0=0\), since
\(D_\theta b=D_\theta g=0\).
\end{proof}

The physical lifts below use the original reduced variables after
the \(k=5\) construction of Theorem~\ref{thm:autonomous-cgo}.
The \(H^5\) remainder controls the order-two term \(K_0\) in
\(H^3\).

\begin{lemma}
\label{lem:fast-physical-lift}
Let \(\boldsymbol z_h\) be an exact solution of the \emph{original} Lam\'e reduced
system \(M\boldsymbol z_h=0\), with expansion
\(\boldsymbol z_h=e^{-\theta\cdot x/h}(a_0+ha_1+r_h)\), where
\(a_0,a_1\) are smooth near \(\overline\Omega_{\mathrm{CGO}}\) and
uniformly bounded in \(H^5(\Omega_{\mathrm{CGO}})\), and
\(\|r_h\|_{H^5(\Omega_{\mathrm{CGO}})}=O(h^2)\).
Let \(K=S^TT\) be the original operator in Section~\ref{sec:reduction}.
For the background-conjugated construction this means
\(\boldsymbol z_h=\mathcal R\widetilde{\boldsymbol z}_h\), \(a_i=\mathcal R\widetilde a_i\), and
\(r_h=\mathcal R\widetilde r_h\), before applying \(K\).
The matrices \(\mathcal R^{\pm1}\) preserve the stated Sobolev orders on
the fixed balls.  Then the normalized physical adjoint field
\[
 y_h^{F}=-hK\boldsymbol z_h
\]
satisfies
\[
 e^{\theta\cdot x/h}y_h^F=J_\theta a_0+O(h)
 \quad\hbox{in }H^3(\Omega_{\mathrm{CGO}}).
\]
\end{lemma}

\begin{proof}
The factorization and its transpose give
\[
 L(T\boldsymbol z_h)=SM\boldsymbol z_h=0,\qquad M^T(K\boldsymbol z_h)=T^TL(T\boldsymbol z_h)=0.
\]
Thus the lift comes from an exact physical solution. The negative-phase
conjugation \eqref{eq:Kconjugated-minus} is
\[
 e^{\theta\cdot x/h}Ke^{-\theta\cdot x/h}=-h^{-1}J_\theta+K_0.
\]
Substitution into \(y_h^F=-hK\boldsymbol z_h\) gives
\begin{equation}\label{eq:proof-fast-lift-expansion}
 \begin{aligned}
 e^{\theta\cdot x/h}y_h^F
 ={}&J_\theta a_0+h(J_\theta a_1-K_0a_0)
       +J_\theta r_h\\
    &-h^2K_0a_1-hK_0r_h.
 \end{aligned}
\end{equation}
On the fixed ball, the order-one operator \(J_\theta\) and order-two
operator \(K_0\) map the stated amplitude bounds into \(H^3\).
The remainder terms satisfy \(J_\theta r_h=O(h^2)\) and
\(hK_0r_h=O(h^3)\) there, proving the \(O(h)\) error.
\end{proof}

If \(J_\theta a_0=0\), the \(h^{-1}\) term vanishes;
the unscaled lift retains the first correction in
\(K_0a_0-J_\theta a_1\).

\begin{lemma}
\label{lem:slow-physical-lift}
Under the hypotheses of Lemma~\ref{lem:fast-physical-lift}, assume in
addition that \(J_\theta a_0=0\).  Then
\[
 y_h^S=K\boldsymbol z_h
\]
satisfies
\[
 e^{\theta\cdot x/h}y_h^S=K_0a_0-J_\theta a_1+O(h)
 \quad\hbox{in }H^3(\Omega_{\mathrm{CGO}}).
\]
\end{lemma}

\begin{proof}
With \(J_\theta a_0=0\), \eqref{eq:Kconjugated-minus} gives
\begin{equation}\label{eq:proof-slow-lift-expansion}
 e^{\theta\cdot x/h}y_h^S
 =K_0a_0-J_\theta a_1+hK_0a_1
       -h^{-1}J_\theta r_h+K_0r_h.
\end{equation}
All norms below are phase-removed norms on
\(\Omega_{\mathrm{CGO}}\). The operator orders and the stated bounds give
\begin{equation}\label{eq:proof-slow-lift-error-bounds}
\begin{aligned}
 \|h^{-1}J_\theta r_h\|_{H^3}
 &\leq Ch^{-1}\|r_h\|_{H^5}=O(h),\\
 \|K_0r_h\|_{H^3}
 &\leq C\|r_h\|_{H^5}=O(h^2),\\
 \|hK_0a_1\|_{H^3}&=O(h).
\end{aligned}
\end{equation}
The three bounds in \eqref{eq:proof-slow-lift-error-bounds} prove the limit and exhibit the $h^{-1}$ loss:
the reduced \(O(h^2)\) remainder yields an \(O(h)\) error for the unscaled adjoint lift in
\(H^3\).
\end{proof}

The reduced remainder order \(O(h^2)\) in \(H^5\) ensures
 the stated \(H^3\) convergence and pointwise independence of the
 additional column from the two principal columns on the prescribed compact
 subsets, for all sufficiently small \(h\). A mere \(O(h)\) bound without the required
 derivative control would not force \(h^{-1}J_\theta r_h\) to vanish.
 For the nine scalar entries of the weighted matrix-moment identity,
 the weaker two Sobolev estimates in
 Remark~\ref{rem:audit-vii-low-regularity-moments} already suffice.


\begin{remark}
\label{rem:audit-vii-low-regularity-moments}
The \(H^5\) reduced-remainder hypothesis above is stronger than needed
for the Betti volume limit. In the \(k=1\) case of
Theorem~\ref{lem:ER-solvability}, the residual correction
\(r_h=-R_{h,\theta}^{(1)}(h\chi_0P_0a_1)\), for fixed smooth
amplitudes and cutoffs, satisfies
\begin{equation}\label{eq:audit-vii-two-level-remainder}
 \|r_h\|_{H^1(\Omega_{\mathrm{CGO}})}\leq Ch^2,
 \qquad \|r_h\|_{H^2(\Omega_{\mathrm{CGO}})}\leq Ch.
\end{equation}
For this particular $L^2$ moment limit, it is enough that
$a_1\in H^3$ on a neighborhood of the fixed residual-cutoff support:
$P_0$ is of order two, so $\chi_0P_0a_1\in H^1$.  The source remains
compactly supported in the ball prescribed by
Theorem~\ref{lem:ER-solvability}.  Applying that theorem to the compact
union of the two opposite normalized direction families gives the same
$h_0$ and uniform constants for both phase signs.  This does not
assert direction derivatives of the exact right inverses or a
global-in-space right-inverse equation.
Here \(K_0:H^2\to L^2\) and \(J_\theta:H^1\to L^2\)
are bounded on the fixed ball. When \(J_\theta a_0=0\), the exact
unscaled physical lift obeys
\begin{equation}\label{eq:audit-vii-slow-l2-bound}
\begin{aligned}
 &\left\|e^{\theta\cdot x/h}K\boldsymbol z_h
 -(K_0a_0-J_\theta a_1)\right\|_{L^2}\\
 &\hspace{1.5em}\leq C\left(h\|a_1\|_{H^2}
     +h^{-1}\|r_h\|_{H^1}+\|r_h\|_{H^2}\right)=O(h).
\end{aligned}
\end{equation}
Similarly, the normalized principal physical lifts have \(O(h)\)
phase-removed errors in \(L^2\). For the first-order operator difference
\(D=M_1-M_2\) and a positive-phase primal CGO field
\(\boldsymbol z_{1,h}=e^{\theta\cdot x/h}(a+ha_1+r_{1,h})\), direct
conjugation gives
\[
 h e^{-\theta\cdot x/h}D\boldsymbol z_{1,h}
 =A_\theta a+h(A_\theta a_1+Da)+A_\theta r_{1,h}
    +h^2Da_1+hDr_{1,h}
 =A_\theta a+O_{L^2}(h).
\]
Consequently the nine normalized physical pairings converge with
\(O(h)\) error by Cauchy--Schwarz, provided that the opposite-phase
amplitudes are fixed and smooth on the working ball. Their exact
physical Betti identities are valid for each fixed \(h\); interior
elliptic regularity supplies smoothness near \(\overline{B_R}\).
Only the limiting quotient frame must be invertible for the ensuing
Plemelj argument: pointwise rank of the exact physical adjoint matrix
is not needed to establish the moments. The stronger \(H^5\) result
above remains useful for its \(H^3\) error bound and pointwise-rank
statement.
\end{remark}

\begin{remark}
\label{rem:CGO-uniformity}\label{rem:CGO-scope}
All applications use the fixed nested balls and source support of
Theorem~\ref{lem:ER-solvability}, with the finite bounds of
Theorem~\ref{thm:autonomous-cgo}. Physical lifts use the original
operators after undoing the background conjugation. Moment limits are
taken at fixed direction and fixed tests; parameter derivatives concern
only the smooth limiting frames and the reconstructed transport.
\end{remark}

Choose an orthonormal triple \(e_1,e_2,\nu_\theta\) and set
\[
 \theta=e_1+\mathrm i e_2,
 \qquad
 z=e_1\cdot x+\mathrm i e_2\cdot x,
 \qquad
 x_\perp=\nu_\theta\cdot x.
\]
Here \(x_\perp\) is a real transverse coordinate, distinct from
the homogeneous spinor coordinate used in Section~\ref{sec:reduction}.
Then
\[
 D_\theta=e_1\cdot\nabla+\mathrm i e_2\cdot\nabla
          =2\partial_{\bar z}.
\]
Hence every amplitude
\[
 a_0=b(x_\perp)H(z),
\]
with \(b\in C_c^\infty(\R)\) and \(H\) holomorphic, satisfies \(D_\theta a_0=0\).

Under the frame and domain hypotheses of
Theorem~\ref{thm:autonomous-cgo}, multiplication by \(b(x_\perp)g(z)\)
preserves a smooth leading transport amplitude, with \(g\) any
holomorphic polynomial. The theorem realizes each resulting amplitude
directly by the localized transport recursion and the residual right
inverse; no polynomial approximation of a transport frame is needed.

The opposite exponential phases cancel in the bilinear
pairing, and the leading-order contribution is determined
by the first-order symbol of the operator difference.

\begin{lemma}\label{lem:CGO-product-limit}
Let \(D=A(x)\cdot\nabla+B(x)\) have smooth matrix coefficients
on a neighborhood of \(\overline{\Omega_{\mathrm{CGO}}}\), where
\(\Omega_{\mathrm{CGO}}\subset\mathbb{R}^3\) is a fixed bounded domain.
Suppose that \(\theta\) ranges over a compact set and,
uniformly in the parameter family, that
\[
 \boldsymbol z_h=e^{\theta\cdot x/h}(a+h a_1+r_h),\qquad
 y_h=e^{-\theta\cdot x/h}(y+q_h),
\]
where \(a,a_1,y\) are bounded in \(H^3\),
\(\|r_h\|_{H^3}\leq Ch^2\), and \(\|q_h\|_{H^3}\leq Ch\).  Then
\begin{equation}\label{eq:CGO-product-limit}
 \left|h\int_{\Omega_{\mathrm{CGO}}} y_h^T D\boldsymbol z_h\,dx-\int_{\Omega_{\mathrm{CGO}}} y^T A_\theta a\,dx\right|\leq Ch,
 \qquad A_\theta:=\sum_{k=1}^3\theta_kA_k.
\end{equation}
\end{lemma}

\begin{proof}
Conjugating \(D\) by the positive phase gives
\begin{equation}\label{eq:proof-cgo-product-conjugation}
 h e^{-\theta\cdot x/h}D\boldsymbol z_h
 =A_\theta a+h\widetilde e_h,
\end{equation}
where the full expansion is
\begin{equation}\label{eq:proof-cgo-product-error-amplitude}
 \widetilde e_h=A_\theta a_1+Da+hDa_1
                 +h^{-1}A_\theta r_h+Dr_h.
\end{equation}
Indeed, we expand directly to obtain
\[
 h(A_\theta/h+D)(a+ha_1+r_h)
 =A_\theta a+hA_\theta a_1+A_\theta r_h
       +hDa+h^2Da_1+hDr_h.
\]
The coefficients are bounded on \(\Omega_{\mathrm{CGO}}\), and
\(D:H^1(\Omega_{\mathrm{CGO}})\to L^2(\Omega_{\mathrm{CGO}})\) is bounded. Applying these bounds to
\eqref{eq:proof-cgo-product-error-amplitude} gives
\begin{equation}\label{eq:proof-cgo-product-error-bound}
\begin{aligned}
\|\widetilde e_h\|_{L^2(\Omega_{\mathrm{CGO}})}
\le C\bigl(
&\|a_1\|_{L^2(\Omega_{\mathrm{CGO}})}
+\|a\|_{H^1(\Omega_{\mathrm{CGO}})}
+h\|a_1\|_{H^1(\Omega_{\mathrm{CGO}})}\\
&+h^{-1}\|r_h\|_{L^2(\Omega_{\mathrm{CGO}})}
+\|r_h\|_{H^1(\Omega_{\mathrm{CGO}})}
\bigr)
\le C.
\end{aligned}
\end{equation}
Here we used \(\|r_h\|_{H^3(\Omega_{\mathrm{CGO}})}\le Ch^2\) and the
stated bounds for the amplitudes. We use
\eqref{eq:proof-cgo-product-conjugation} and cancel the
exponential factors to obtain the exact identity
\begin{equation}\label{eq:proof-cgo-product-integral-error}
\begin{aligned}
h\int_{\Omega_{\mathrm{CGO}}} y_h^T D \boldsymbol z_h\,dx
-\int_{\Omega_{\mathrm{CGO}}} y^T A_\theta a\,dx
={}&\int_{\Omega_{\mathrm{CGO}}} q_h^T A_\theta a\,dx\\
&+h\int_{\Omega_{\mathrm{CGO}}} (y+q_h)^T\widetilde e_h\,dx.
\end{aligned}
\end{equation}
By \eqref{eq:proof-cgo-product-integral-error},
\eqref{eq:proof-cgo-product-error-bound}, and the Cauchy--Schwarz inequality,
\[
\begin{aligned}
\left|
h\int_{\Omega_{\mathrm{CGO}}} y_h^T D \boldsymbol z_h\,dx
-\int_{\Omega_{\mathrm{CGO}}} y^T A_\theta a\,dx
\right|
\le{}&
\|q_h\|_{L^2(\Omega_{\mathrm{CGO}})}
\|A_\theta a\|_{L^2(\Omega_{\mathrm{CGO}})}\\
&+h\bigl(
\|y\|_{L^2(\Omega_{\mathrm{CGO}})}
+\|q_h\|_{L^2(\Omega_{\mathrm{CGO}})}
\bigr)
\|\widetilde e_h\|_{L^2(\Omega_{\mathrm{CGO}})}
\le Ch.
\end{aligned}
\]
The constant is uniform on the stated compact parameter
family: the coefficients are fixed, the directions range
over a compact set, and all amplitude and remainder bounds
are uniform.
\end{proof}

Lemma~\ref{lem:CGO-product-limit} isolates the directional
first-order coefficient \(A_\theta\) in the bilinear volume integral.
After multiplication by \(h\), the zeroth-order operator and the
derivatives of the fixed amplitudes contribute only to the error.
This is the limit used for the coefficient difference in the boundary
identity \eqref{eq:primal-adjoint}.

\subsection{The principal adjoint amplitudes}

The two principal adjoint amplitudes arise from a two-component
leafwise transport system whose columns are lifted to the full reduced
transport equation. The required invertible matrix frames are
constructed on a fixed disk, with local smooth dependence on the
direction and transverse parameters.

\begin{lemma}\label{lem:matrix-dbar-frame}
Let \(m\) be a positive integer, let \(D\Subset\C\) be a disk, and let \(\mathcal B\)
be a smooth \(M_m(\C)\)-valued function on a neighborhood of
\(\overline D\). Then there is a smooth invertible matrix \(\Phi\)
on a neighborhood of \(\overline D\) satisfying
\begin{equation}\label{eq:matrix-dbar-frame}
 2\partial_{\bar z}\Phi=\mathcal B\Phi\quad\text{in }D.
\end{equation}
If \(\mathcal B(z,p)\) is jointly smooth on a fixed neighborhood of
\(\overline D\) for \(p\) in a finite-dimensional real parameter chart,
then, locally near each \(p_0\), the frame may be chosen jointly smooth
in \(z\) and \(p\). Its derivatives of every fixed finite order and
its inverse are bounded on compact spatial and parameter subcharts.
\end{lemma}

\begin{proof}
We choose disks \(D\Subset D_1\Subset D_2\) inside the coefficient
neighborhood and shrink the parameter chart if needed so that these
disks are fixed for all nearby parameters.
For a fixed parameter \(p_0\), the operator
\(\bar\partial-\mathcal B_{p_0}d\bar z/2\) is integrable, since complex
dimension one has no nonzero \((0,2)\)-forms. The Koszul--Malgrange theorem
\cite{KoszulMalgrange1958} gives a holomorphic bundle structure. The underlying
smooth bundle on the contractible Stein disk \(D_2\) is trivial; the
Oka--Grauert principle \cite{Grauert1958} therefore makes it holomorphically
trivial. Thus it has a smooth invertible frame
\(\Phi_0\) solving \(\bar\partial\Phi_0=\mathcal B_{p_0}\Phi_0/2\).

Parameter dependence is obtained by the following local construction.
Choose \(\chi\in C_c^\infty(D_1)\) equal to one near \(\overline D\), and set
\[
 K_p=\tfrac12\chi\Phi_0^{-1}(\mathcal B_p-\mathcal B_{p_0})\Phi_0.
\]
The support of \(K_p\) lies in the fixed compact set
\(\supp\chi\Subset D_1\). Thus, for
\(u\in C^{0,\alpha}(\overline{D_1})\), the product \(K_pu\) extends
by zero to a compactly supported \(C^{0,\alpha}\) function on the plane.
Let \(\mathcal C_{\bar\partial}\) denote the planar Cauchy right inverse, normalized
by \(\bar\partial\mathcal C_{\bar\partial} f=f\). For $0<\alpha<1$ and
$f\in C^{0,\alpha}(\mathbb C)$ supported in the fixed compact set
$\supp\chi\Subset D_1$, its standard Cauchy--Schauder bound is

\begin{equation}\label{eq:planar-dbar-Holder-bound}
 \|\mathcal C_{\bar\partial} f\|_{C^{1,\alpha}(\overline{D_1})}
 \leq C_{\alpha,D_1,\chi}\|f\|_{C^{0,\alpha}(\mathbb C)}.
\end{equation}

Consequently \(S_pu=\mathcal C_{\bar\partial}(K_pu)|_{D_1}\) is bounded on
\(C^{0,\alpha}(\overline{D_1})\), with norm tending to zero as
\(p\to p_0\).
On a sufficiently small parameter neighborhood the Neumann series defines
\begin{equation}\label{eq:proof-dbar-neumann-frame}
 \Psi_p=(I-S_p)^{-1}I_m,
 \qquad \|\Psi_p-I_m\|_{C^0(\overline{D_1})}<\tfrac12.
\end{equation}
In particular \(\Psi_p\) is invertible. Equation~\eqref{eq:proof-dbar-neumann-frame} gives
\(\bar\partial\Psi_p=K_p\Psi_p\), so \(\Phi_p=\Phi_0\Psi_p\) obeys
\begin{equation}\label{eq:proof-dbar-cutoff-frame-equation}
 \bar\partial\Phi_p
 =\frac12\bigl[\mathcal B_{p_0}
       +\chi(\mathcal B_p-\mathcal B_{p_0})\bigr]\Phi_p.
\end{equation}
Where \(\chi=1\), \eqref{eq:proof-dbar-cutoff-frame-equation} is
\eqref{eq:matrix-dbar-frame}, including on a
neighborhood of \(\overline D\). Smooth dependence of the Banach-space
inverse in \eqref{eq:proof-dbar-neumann-frame} gives parameter dependence in \(C^{0,\alpha}\).
Equation~\eqref{eq:proof-dbar-neumann-frame}, Cauchy estimates, and parameter differentiation then give
all mixed spatial and parameter derivatives on the same parameter
neighborhood. Hence \(\Phi_p\) is jointly smooth. On compact parameter
subcharts, compactness
of \(\overline D\) supplies uniform bounds for each fixed derivative and
for \(\Phi_p^{-1}\).
\end{proof}

The parameter-local choice of frame is sufficient for the
CGO construction: on a compact subchart, the frame and its inverse
have the finite derivative bounds required by
Theorem~\ref{thm:autonomous-cgo}. We apply
Lemma~\ref{lem:matrix-dbar-frame} with two components to the coupled two-component transport system.

Fix locally a normalized direction \(\theta=e_1+\mathrm i e_2\)
and use \(z=e_1\cdot x+\mathrm i e_2\cdot x\), so
\(D_\theta=2\partial_{\bar z}\). The direction is held fixed when
differentiating in \(x\); the direction and transverse coordinate
are auxiliary parameters for the leafwise frame construction.
Take \(\boldsymbol z_0=(v,f)\) satisfying
\begin{equation}\label{eq:proof-fast-leading-transport}
 (2D_\theta+A_\theta)\boldsymbol z_0=0.
\end{equation}
Define
\begin{equation}\label{eq:Rfast}
 R=c\,\theta\cdot v.
\end{equation}
The scalar block of \eqref{eq:proof-fast-leading-transport},
using \eqref{eq:Atheta}, gives
\(2D_\theta f=-\kappa\sqrt\mu\,\theta\cdot v\).
Contracting the vector block of \eqref{eq:proof-fast-leading-transport}
with \(\theta\), and using
\(D_\theta c=-\tfrac12cD_\theta\ell\), cancels the terms containing
\(D_\theta\ell\) and gives
\(D_\theta R=-\tfrac c2 f\theta^TF_\mu\theta\).
The scalar and contracted vector blocks of
\eqref{eq:proof-fast-leading-transport}, together with
\eqref{eq:Rfast} and \eqref{eq:fast-null-contraction} below, yield
\begin{align}
 D_\theta f&=-\frac{\mu\kappa}{2}R,\label{eq:fast1}\\
 D_\theta R&=fD_\theta^2(\mu^{-1}).\label{eq:fast2}
\end{align}
Equations~\eqref{eq:fast1}--\eqref{eq:fast2} show that
\((f,R)^T\) satisfies the following matrix
\(\bar\partial\)-system on each complex leaf:
\begin{equation}\label{eq:fast-matrix-dbar}
 2\partial_{\bar z}
 \binom fR
 =
 \begin{pmatrix}
 0&-\dfrac{\mu\kappa}{2}\\[1mm]
 D_\theta^2(\mu^{-1})&0
 \end{pmatrix}
 \binom fR.
\end{equation}
Apply Lemma~\ref{lem:matrix-dbar-frame} to
\eqref{eq:fast-matrix-dbar} on a slightly larger leaf disk and obtain
a smooth invertible fundamental matrix
\[
\Phi=\begin{pmatrix}f_1&f_2\\ R_1&R_2\end{pmatrix}\in GL(2,\C)
\]
on a neighborhood of the closed leaf disk, locally smoothly in the
direction and transverse parameters. Since the matrix in
\eqref{eq:fast-matrix-dbar} has trace zero, \(D_\theta\det\Phi=0\).
The determinant is therefore a nonvanishing holomorphic function of
the leaf coordinate. Multiplying the first column by
\((\det\Phi)^{-1}\) preserves the transport equation and gives the
normalization
\begin{equation}\label{eq:fast-frame-normalization}
 \det\Phi=1.
\end{equation}
The normalization \eqref{eq:fast-frame-normalization} changes only leading transport amplitudes; exact CGO solutions are
constructed for the resulting amplitudes below.

The two columns admit an explicit ambient lift. We put
\(\alpha_\mu=D_\theta^2(\mu^{-1})\). Using \eqref{eq:Fmu} and
\(\theta\cdot\theta=0\), we obtain
\begin{equation}\label{eq:fast-null-contraction}
 \alpha_\mu=-\frac c2\theta^TF_\mu\theta.
\end{equation}
Indeed, nullness removes the scalar identity term in \(F_\mu\), and
\[
 D_\theta^2(\mu^{-1})
 =\mu^{-1}\bigl[(D_\theta\ell)^2-D_\theta^2\ell\bigr],
 \qquad
 \theta^TF_\mu\theta
 =2c\bigl[D_\theta^2\ell-(D_\theta\ell)^2\bigr].
\]
We choose \(a=a(\theta)\), constant in \(x\), with \(\theta\cdot a=1\).
For normalized directions, \(\theta\cdot\overline\theta=2\), so one
may take \(a=\overline\theta/2\). This choice is smooth in the real
direction parameters.
For a column \((f,R)\) of \(\Phi\), we define
\begin{equation}\label{eq:proof-fast-ambient-defect}
 H=2D_\theta(c^{-1}Ra)-mc^{-1}R+F_\mu\theta f.
\end{equation}
Equations~\eqref{eq:proof-fast-ambient-defect}, \eqref{eq:fast2}, and
\eqref{eq:fast-null-contraction} give
\begin{equation}\label{eq:proof-fast-defect-contraction}
 \theta\cdot H=2c^{-1}(D_\theta R-\alpha_\mu f)=0.
\end{equation}
Choose a planar cutoff \(\chi_{\rm leaf}\) supported in the larger
disk and equal to one near the closed working disk. With the
componentwise inverse of Lemma~\ref{lem:cauchy-leaf-estimate}, we set
\begin{equation}\label{eq:proof-fast-cauchy-correction}
 w=-\frac12\mathcal C_\theta(\chi_{\rm leaf}H).
\end{equation}
Equation~\eqref{eq:proof-fast-cauchy-correction} and
Lemma~\ref{lem:cauchy-leaf-estimate} give \(2D_\theta w=-H\)
where \(\chi_{\rm leaf}=1\). Since
\(\theta\cdot H=0\) by \eqref{eq:proof-fast-defect-contraction}
throughout the cutoff support and \(\theta\)
is constant in \(x\), the integral also gives \(\theta\cdot w=0\).
Using \eqref{eq:fast1}, \eqref{eq:Rfast},
\eqref{eq:proof-fast-ambient-defect}, and
\eqref{eq:proof-fast-cauchy-correction}, we verify that
\begin{equation}\label{eq:proof-fast-ambient-amplitude}
 \boldsymbol z_0^{F}=(c^{-1}Ra+w,f)^T
\end{equation}
solves both blocks of \eqref{eq:proof-fast-leading-transport} near the closed
working disk. After pulling back by the local smooth rotations, the
disks and cutoff may be fixed. Lemma~\ref{lem:cauchy-leaf-estimate} then gives smooth
dependence on the auxiliary parameters and bounded derivatives on
compact spatial and parameter subcharts.

Using \eqref{eq:proof-fast-ambient-amplitude}, \eqref{eq:Jtheta},
and \eqref{eq:fast1}, we obtain
\begin{equation}\label{eq:Jfast}
 J_\theta \boldsymbol z_0^{F,j}=
 \begin{pmatrix}
 cf_j\theta\\
 -R_j
 \end{pmatrix}
 =:Y_F^{(j)}.
\end{equation}
The dual transport of
\(Y_F^{(j)}\) can also be verified directly. For arbitrary \(f,R\),
we put \(E_f=D_\theta f+\mu\kappa R/2\) and
\(E_R=D_\theta R-\alpha_\mu f\). Then
\begin{equation}\label{eq:fast-dual-defect}
 (2D_\theta-A_\theta^T)\binom{cf\theta}{-R}
       =\binom{2cE_f\theta}{-2E_R}.
\end{equation}
Equations~\eqref{eq:fast1}--\eqref{eq:fast2} make both defects vanish.
The columns also belong to \(E_\theta^*\), because
\(\theta^T(cf_j\theta)=0\). In the ordered basis
\(\bigl((\theta,0)^T,(0,1)^T\bigr)\), their coefficient matrix has
determinant
\begin{equation}\label{eq:proof-fast-column-determinant}
 \det\begin{pmatrix}cf_1&cf_2\\-R_1&-R_2\end{pmatrix}
 =-c\det\Phi=-c\ne0.
\end{equation}
Equation~\eqref{eq:proof-fast-column-determinant} shows that
the two columns form a basis of the two-dimensional image of the leading adjoint lift
\begin{equation}\label{eq:fast-space}
 \mathcal F_\theta=
 \Span\left\{\binom\theta0,\binom01\right\}.
\end{equation}

Theorem~\ref{thm:graded-frame} below realizes the localized amplitudes
as exact negative-phase fields
\(\boldsymbol z_h^{F,j}=e^{-\theta\cdot x/h}
(\chi \boldsymbol z_0^{F,j}+ha_1^{F,j}+r_h^{F,j})\), with
\(r_h^{F,j}=O(h^2)\) in \(H^5(\Omega_{\mathrm{CGO}})\).
Their normalized physical adjoint lifts are
\(\widehat y_h^{F,j}=-hS^TT\boldsymbol z_h^{F,j}\).
By Lemma~\ref{lem:fast-physical-lift}, on every fixed smooth
\(\Omega_{\mathrm{CGO}}'\Subset\Omega_{\mathrm{CGO}}\) contained in
an open region where \(\chi\) is identically one,
\begin{equation}\label{eq:fast-asymptotic}
 e^{\theta\cdot x/h}\widehat y_h^{F,j}=Y_F^{(j)}+O(h)
 \quad\text{in }H^3(\Omega_{\mathrm{CGO}}').
\end{equation}

The ambient correction \(w\) in
\eqref{eq:proof-fast-cauchy-correction} has
\(\theta\cdot w=0\), so it leaves the contracted component \(R\)
unchanged. Consequently, the principal adjoint columns in \eqref{eq:Jfast}
depend only on the two-component system for \((f,R)\), even though
their exact physical realization uses the full four-component field.

\subsection{The additional adjoint amplitude}

The principal amplitudes \(Y_F^{(1)},Y_F^{(2)}\) span a
two-dimensional subspace of the rank-three dual quotient. The
supplementary direction is obtained by canceling the leading physical
adjoint lift and retaining its first nonzero correction. Its exact
physical realization permits the use of the elastic boundary identity
\eqref{eq:primal-adjoint}.

For the normalized direction \(\theta\), we choose locally a smooth
spinor lift \((s,t)\) with \(\theta=\theta(s,t)\) as in
\eqref{eq:theta-spinor}. Here \(s,t\) are complex spinor coordinates;
we use \(x_\perp\) for the real transverse coordinate below. We hold the
lift fixed when differentiating in \(x\). In the corresponding
spinor frame \eqref{eq:qspinor}, we choose the spatially constant
local polarization
\begin{equation}\label{eq:slow-polarization-normalization}
 \begin{array}{lll}
 s\ne0:&q_\theta=q_1/s,&\xi_q=(s^{-1},0)^T,\\
 t\ne0:&q_\theta=q_2/t,&\xi_q=(0,t^{-1})^T.
 \end{array}
 \qquad q_\theta=Q\xi_q,\quad \det(n,\xi_q)=-1.
\end{equation}
The determinant \(\det(n,\xi_q)=-1\) makes \(n,\xi_q\)
linearly independent. Since \(Q\) is injective and
\(Qn=-\theta/2\) by \eqref{eq:Qn},
\eqref{eq:slow-polarization-normalization} implies that \(q_\theta\notin\C\theta\).
The identity \(\theta^TQ=0\) also gives
\(q_\theta\cdot\theta=0\). These choices define one local column
on each spinor chart. We set
\begin{equation}\label{eq:slow-z0}
 \boldsymbol z_0^S=\binom{q_\theta}{0}.
\end{equation}
Using \eqref{eq:slow-z0}, \eqref{eq:Atheta}, and \eqref{eq:Jtheta},
together with spatial constancy and \(\theta\cdot q_\theta=0\), we obtain
\begin{equation}\label{eq:proof-slow-leading-transport}
 (2D_\theta+A_\theta)\boldsymbol z_0^S=0,
 \qquad
 J_\theta \boldsymbol z_0^S=0.
\end{equation}
Using \eqref{eq:M-explicit} and \eqref{eq:slow-z0}, we compute the
source for the first correction:
\begin{equation}\label{eq:proof-slow-correction-source}
 M_{\mu,\kappa}\boldsymbol z_0^S
 =\binom{V_\mu q_\theta}{-e_{\mu,\kappa}m\cdot q_\theta}.
\end{equation}
On a slightly larger leaf disk we choose a smooth invertible ambient
frame \(C\) satisfying \((2D_\theta+A_\theta)C=0\), by
Lemma~\ref{lem:matrix-dbar-frame}. We choose \(\chi_{\rm leaf}\) as a planar
cutoff supported there and equal to one near the closed working disk.
The identity \((2D_\theta+A_\theta)(Cw)=2CD_\theta w\) shows that
\begin{equation}\label{eq:proof-slow-cauchy-correction}
 \boldsymbol z_1^S=\frac12C\mathcal C_\theta
       \bigl(\chi_{\rm leaf}C^{-1}M_{\mu,\kappa}\boldsymbol z_0^S\bigr)
\end{equation}
is a smooth first correction. Applying the Cauchy inverse of
Lemma~\ref{lem:cauchy-leaf-estimate} to \eqref{eq:proof-slow-cauchy-correction}
with source \eqref{eq:proof-slow-correction-source}, on fixed disks after the local
coordinate pullback, gives smooth dependence on the direction and
transverse parameters. Near the closed working disk it satisfies
\begin{equation}\label{eq:slow-z1}
 (2D_\theta+A_\theta)\boldsymbol z_1^S=M_{\mu,\kappa}\boldsymbol z_0^S.
\end{equation}
Theorem~\ref{thm:graded-frame} realizes an exact negative-phase field
\(\boldsymbol z_h^S\), using the commutator source \([M,\chi]\boldsymbol z_0^S\) to preserve
\(\boldsymbol z_0^S,\boldsymbol z_1^S\) on the open region where the transverse cutoff is identically one. Its reduced remainder
\(r_h^S\) is \(O(h^2)\) in \(H^5\). Lemma~\ref{lem:slow-physical-lift}
therefore gives, on each fixed smooth subdomain
\(\Omega_{\mathrm{CGO}}'\Subset\Omega_{\mathrm{CGO}}\) in that open constant-cutoff region,
\begin{equation}\label{eq:Yslow}
 e^{\theta\cdot x/h}S^TT\boldsymbol z_h^S=Y_S+O(h),
 \qquad Y_S=K_0\boldsymbol z_0^S-J_\theta \boldsymbol z_1^S,
\end{equation}
with error in \(H^3(\Omega_{\mathrm{CGO}}')\).
For each fixed \(h\), the exact reduced solution defines the physical
solution \(u_h^S=T\boldsymbol z_h^S\), and its canonical adjoint lift is
\(S^Tu_h^S=K\boldsymbol z_h^S\). The condition \(J_\theta \boldsymbol z_0^S=0\)
cancels the leading term of the conjugated lift; keeping the lift
unscaled retains \(K_0\boldsymbol z_0^S-J_\theta \boldsymbol z_1^S\).
Since \(J_\theta \boldsymbol z\in\mathcal F_\theta\) and the vector component of
\(K_0\boldsymbol z_0^S\) is \(q_\theta\),
\begin{equation}\label{eq:Yslow-mod}
 Y_S\equiv\binom{q_\theta}{0}\pmod{\mathcal F_\theta}.
\end{equation}
Thus \eqref{eq:slow-polarization-normalization} gives
\(Y_S\notin\mathcal F_\theta\).
More explicitly, we write
\(\boldsymbol z_1^S=(v_1,f_1)\), \(R_1=c\theta\cdot v_1\). The scalar component of
\eqref{eq:slow-z1} gives
\begin{equation}\label{eq:proof-slow-scalar-correction}
 D_\theta f_1=-\frac{\mu\kappa}{2}R_1
               -\frac{e_{\mu,\kappa}}2m\cdot q_\theta.
\end{equation}
The scalar component of \(K_0\boldsymbol z_0^S\) is
\(c\,m\cdot q_\theta/[2(1-\kappa)]\). The scalar component of
\(K_0\boldsymbol z_0^S-J_\theta \boldsymbol z_1^S\) is therefore
\begin{equation}\label{eq:proof-slow-scalar-lift}
 \frac{R_1+2\mu^{-1}D_\theta f_1
                 +\tfrac12c\,m\cdot q_\theta}{1-\kappa}.
\end{equation}
Using \eqref{eq:proof-slow-scalar-correction} for \(D_\theta f_1\) and
\(\mu^{-1}e_{\mu,\kappa}=\tfrac12(3\kappa-2)c\), which follows from
\eqref{eq:emu}, reduces \eqref{eq:proof-slow-scalar-lift} to
\(R_1+\tfrac32c\,m\cdot q_\theta\), giving
\begin{equation}\label{eq:slow-column-explicit}
 Y_S=\binom{q_\theta-cf_1\theta}
              {R_1+\frac32c\,m\cdot q_\theta}.
\end{equation}
Equation~\eqref{eq:slow-column-explicit} also permits a direct check of the dual transport equation. The
vector component of \eqref{eq:slow-z1} yields
\begin{equation}\label{eq:proof-slow-contracted-correction}
 D_\theta R_1=f_1D_\theta^2(\mu^{-1})
                  +\frac c2\theta^TV_\mu q_\theta.
\end{equation}
For the remaining coefficient identity, \eqref{eq:Fmu}--\eqref{eq:Vmu}
give
\begin{equation}\label{eq:proof-slow-coefficient-expansions}
 \begin{aligned}
 F_\mu={}&-4\nabla^2c-4c^{-1}\nabla c\otimes\nabla c
       +(4\Delta c+4c^{-1}|\nabla c|^2)I_3,\\
 cV_\mu={}&2\nabla^2c-4c^{-1}\nabla c\otimes\nabla c
       +(\Delta c-2c^{-1}|\nabla c|^2)I_3.
 \end{aligned}
\end{equation}
Equation~\eqref{eq:proof-slow-coefficient-expansions},
together with \(cm=-2\nabla c\), yields
\begin{equation}\label{eq:proof-slow-coefficient-identity}
 F_\mu-cV_\mu-3(\nabla(cm))^T
 =\left(3\Delta c+6c^{-1}|\nabla c|^2\right)I_3.
\end{equation}
Here \((\nabla(cm))_{ij}=\partial_j(cm_i)\); this matrix is symmetric.
Contract \eqref{eq:proof-slow-coefficient-identity}
with \(\theta^T\) and \(q_\theta\), using
\(\theta\cdot q_\theta=0\) and the spatial constancy of \(q_\theta\), to obtain
\begin{equation}\label{eq:proof-slow-coefficient-contraction}
 \theta^TF_\mu q_\theta
 =c\theta^TV_\mu q_\theta+3D_\theta(c\,m\cdot q_\theta).
\end{equation}
The residual in the dual transport equation is expressed
using the scalar and contracted vector equations for \(\boldsymbol z_1^S\).
Their defects are
\begin{equation}\label{eq:proof-slow-correction-defects}
 E_f^S=D_\theta f_1+\frac{\mu\kappa}{2}R_1
                  +\frac{e_{\mu,\kappa}}2m\cdot q_\theta,
 \qquad
 E_R^S=D_\theta R_1-\alpha_\mu f_1
                  -\frac c2\theta^TV_\mu q_\theta.
\end{equation}
Substitution of \eqref{eq:proof-slow-correction-defects} and
\eqref{eq:proof-slow-coefficient-contraction} into
\eqref{eq:slow-column-explicit}, using
\eqref{eq:proof-slow-scalar-correction} and
\eqref{eq:proof-slow-contracted-correction}, gives the exact identity
\begin{equation}\label{eq:slow-dual-defect}
 (2D_\theta-A_\theta^T)Y_S
       =\binom{-2cE_f^S\theta}{2E_R^S}=0.
\end{equation}
Equation~\eqref{eq:slow-dual-defect} proves that the additional column
satisfies dual transport directly. Its vector
component is orthogonal to \(\theta\), so \(Y_S\in E_\theta^*\).
If the first correction is replaced by \(\boldsymbol z_1^S+\delta \boldsymbol z_1\), where
\((2D_\theta+A_\theta)\delta \boldsymbol z_1=0\), then \(Y_S\) changes by
\(-J_\theta\delta \boldsymbol z_1\in\mathcal F_\theta\). Hence its class modulo
\(\mathcal F_\theta\), and its independence from the two principal columns,
are preserved.

The distinction between the physical displacement amplitudes
and their canonical adjoint lifts is reflected in the principal
symbol of the elasticity operator. For a nonzero complex null
direction \(\theta\), this symbol reduces to
\[
\mu(\theta\cdot\theta)I_3
+(\lambda+\mu)\theta\theta^T
=(\lambda+\mu)\theta\theta^T,
\]
whose kernel is \(\theta^\perp\). The canonical scalar coordinate differentiates
the displacement: for \(u_h=e^{-\theta\cdot x/h}U_h\),
\[
 e^{\theta\cdot x/h}S^Tu_h
   =\binom{\sqrt\mu\,U_h}
       {-\bigl(\Div U_h-h^{-1}\theta\cdot U_h\bigr)/(1-\kappa)}.
\]
It can therefore retain the order-\(h\) physical amplitude.
On \(M\boldsymbol z=0\), the scalar equation in \eqref{eq:M-explicit} reduces
\eqref{eq:Kexact} to
\[
 K(v,f)=\binom{v+c\nabla f+cfm}
                  {-c\Div v+\tfrac32c\,m\cdot v}.
\]
The vector components of \(Y_F^{(1)},Y_F^{(2)}\) are collinear, while these two canonical columns
are independent by the determinant \(-c\) in
\eqref{eq:proof-fast-column-determinant}. The additional column adds
\(q_\theta\) modulo \(\C\theta\), by \eqref{eq:slow-column-explicit}.
Rank three concerns these canonical four-component fields, rather than
three independent leading displacement vectors. Every column is
realized through \(M\boldsymbol z=0\), \(u=T\boldsymbol z\), and \(S^Tu\).

The three columns are now localized and realized together on the fixed CGO
domain, preserving the additional amplitude on the open region where the cutoff is identically one.

\begin{theorem}
\label{thm:graded-frame}
Suppose that the smooth strongly convex Lam\'e pair has a constant
exterior extension. Fix the nested CGO balls of
Theorem~\ref{lem:ER-solvability} for its background-conjugated reduction,
write \(\Omega_0=B_{R_{\rm ext}}\), and choose \(L>R_{\rm ext}\).
Choose locally the orthonormal vectors \(e_1(\theta),e_2(\theta),
\nu_\theta\) and a spinor lift with \(\theta=\theta(s,t)\).
Around each normalized direction and transverse parameter there are
neighborhoods \(U,I\) on whose cylinder
\[
 \begin{aligned}
 \mathcal C_{U,I,L}=\{(x,\theta):{}&
 x=x_\perp\nu_\theta+x_1e_1(\theta)+x_2e_2(\theta),\\
 &\theta\in U,\quad x_\perp\in I,\quad |x_1+\mathrm i x_2|<L\}
 \end{aligned}
\]
the two principal columns and the additional column give a jointly smooth
\(3\times3\) dual quotient frame \(Y\). It solves
\(2D_\theta Y+\Gamma_{\mu,\kappa}Y=0\) and satisfies
\begin{equation}\label{eq:graded-determinant}
 \det Y=-2c\det\Phi\det(n,\xi_q)=2c
\end{equation}
in the local coordinates and normalizations
\eqref{eq:fast-frame-normalization} and
\eqref{eq:slow-polarization-normalization}. The frame and its inverse
have bounded derivatives of every fixed finite order on compact
spatial and parameter subcharts.

Let \(I'\Subset I\) and choose \(\chi\in C_c^\infty(I)\) equal to
one near \(\overline{I'}\). There are exact physical solutions on
the fixed ball \(\Omega_{\mathrm{CGO}}\) whose normalized canonical adjoint lifts form
\begin{equation}\label{eq:Yh-matrix}
 \mathcal Y_h=(-hK\boldsymbol z_h^{F,1},-hK\boldsymbol z_h^{F,2},K\boldsymbol z_h^S),\qquad K=S^TT,
\end{equation}
and obey
\begin{equation}\label{eq:Yh-limit}
 e^{\theta\cdot x/h}\mathcal Y_h=W_\chi+E_h,
 \qquad \|E_h\|_{H^3(\Omega_{\mathrm{CGO}})}\leq Ch,
 \qquad W_\chi=\mathsf U_\theta Y\ \hbox{where }x_\perp\in I'.
\end{equation}
Here \(W_\chi\) is a smooth \(4\times3\) localized ambient amplitude
on \(\Omega_{\mathrm{CGO}}\), of rank three on the set \(\chi=1\).
The constants are uniform on compact direction charts with the finite
amplitude and correction bounds of Theorem~\ref{thm:autonomous-cgo}.
For any nonnegative integer \(r\), the analogous \(O(h)\) error
estimate in \(H^r\) follows from a reduced \(O(h^2)\) remainder
in \(H^{r+2}\) and the corresponding amplitude bounds.
The rank assertion uses the \(H^3\) estimate \eqref{eq:Yh-limit} and
Sobolev embedding. The parameter estimates concern the smooth
leading amplitudes and their transport corrections.
In particular, the unscaled third column uses the
$H^5$ bound for the reduced remainder in \eqref{eq:localized-cgo-remainder}:
the term $h^{-1}J_\theta r_h^S$ is $O(h)$ in $H^3$ by
\eqref{eq:proof-slow-lift-error-bounds}.
\end{theorem}

\begin{proof}
Choose the frames on a disk larger than \(\overline{D_L}\), where
\(D_L=\{z:|z|<L\}\), and shrink \(U,I\) so that the disks are fixed
after coordinate pullback. Lemma~\ref{lem:matrix-dbar-frame} and
\eqref{eq:Jfast}, \eqref{eq:slow-column-explicit} give smooth columns;
\eqref{eq:fast-dual-defect} and \eqref{eq:slow-dual-defect} prove their
dual transport.

Write \(\boldsymbol z_1^S=(v_S,f_S)\), \(R_S=c\theta\cdot v_S\), to distinguish
the first correction \(\boldsymbol z_1^S\) from the two columns of \(\Phi\). In quotient coordinates,
\eqref{eq:Qn} gives
\begin{equation}\label{eq:proof-graded-frame-coordinates}
 Y=\begin{pmatrix}
   -2cf_1n&-2cf_2n&\xi_q+2cf_Sn\\
   -R_1&-R_2&R_S+\tfrac32c\,m\cdot q_\theta
 \end{pmatrix}.
\end{equation}
Expanding the two-row top block in the basis \((n,\xi_q)\) along
the third column's \(\xi_q\)-coordinate yields
\begin{equation}\label{eq:proof-graded-determinant-expansion}
 \det Y=\det(n,\xi_q)\,[-2c(f_1R_2-f_2R_1)].
\end{equation}
The normalizations \(\det\Phi=1\), \(\det(n,\xi_q)=-1\) give
\eqref{eq:graded-determinant}.
Also \(\tr\Gamma=-2D_\theta\log c\), so \(2c\) satisfies determinant
transport. Since \(\min c>0\) on compact subcylinders, the adjugate
formula and joint smoothness give the finite inverse bounds.

Extend \(\chi(x_\perp)\boldsymbol z_0^{F,j}\) by zero in \(x_\perp\); homogeneous
transport is preserved because \(D_\theta\chi=0\). For the additional column set
\begin{equation}\label{eq:graded-slow-localization}
 a_{0,\chi}^S=\chi \boldsymbol z_0^S,\qquad
 a_{1,\chi}^S=\chi \boldsymbol z_1^S+w,\qquad
 (2D_\theta+A_\theta)w=[M,\chi]\boldsymbol z_0^S.
\end{equation}
The source \([M,\chi]\boldsymbol z_0^S=M(\chi \boldsymbol z_0^S)-\chi M\boldsymbol z_0^S\) is
\begin{equation}\label{eq:proof-slow-localization-source}
 [M,\chi]\boldsymbol z_0^S=
 \binom{\chi''q_\theta-\chi'm(\nu_\theta\cdot q_\theta)}
       {d_{\mu,\kappa}\chi'(\nu_\theta\cdot q_\theta)}.
\end{equation}
Primes denote \(x_\perp\)-derivatives. Choose the ambient frame
\(C\) on the entire larger disk and a fixed leaf cutoff
\(\chi_{\rm leaf}\) inside its domain, equal to one near
\(\overline{D_L}\). The source need not be compact in the leaf variable;
its cutoff product defines
\begin{equation}\label{eq:proof-slow-localization-correction}
 w=\frac12C\mathcal C_\theta
       \bigl(\chi_{\rm leaf}C^{-1}[M,\chi]\boldsymbol z_0^S\bigr).
\end{equation}
Lemma~\ref{lem:cauchy-leaf-estimate} solves the last equation of
\eqref{eq:graded-slow-localization} near \(\overline{D_L}\).
The integral acts at fixed \(x_\perp\), so \(w\) is supported in
\(\supp\chi'\cup\supp\chi''\Subset I\) and vanishes with all spatial
derivatives on each transverse neighborhood where \(\chi=1\).
Thus the amplitudes extend smoothly by zero, and
\begin{equation}\label{eq:proof-localized-slow-transport}
 \begin{aligned}
 (2D_\theta+A_\theta)a_{1,\chi}^S
 &=\chi M\boldsymbol z_0^S+[M,\chi]\boldsymbol z_0^S=Ma_{0,\chi}^S,\\
 J_\theta a_{0,\chi}^S&=0.
 \end{aligned}
\end{equation}
On transverse neighborhoods where \(\chi\) is identically one, the amplitudes equal \(\boldsymbol z_0^S,\boldsymbol z_1^S\). Since \(L>R_{\rm ext}\),
they and their transport equations are defined near the whole
\(\overline\Omega_0\) after transverse extension.

Apply Theorem~\ref{thm:autonomous-cgo} with \(k=5\) to the conjugated
reduced amplitudes for the first two columns and to the prescribed pair
\(\mathcal R^{-1}a_{0,\chi}^S,\mathcal R^{-1}a_{1,\chi}^S\).
The covariance \eqref{eq:transport-covariance} and
\(\widetilde M=\mathcal R^{-1}M\mathcal R\) carry both transport
equations to these variables. Returning by \(\mathcal R\) preserves
the reduced \(O(h^2)\) remainder in \(H^5\) on the fixed ball.
Equation~\eqref{eq:factorization} gives exact physical fields
\(u_h^{F,j}=-hT\boldsymbol z_h^{F,j}\), \(u_h^S=T\boldsymbol z_h^S\), whose canonical lifts
are \eqref{eq:Yh-matrix}. Their leading amplitude on the entire working ball is
\begin{equation}\label{eq:proof-localized-graded-amplitude}
 W_\chi=\bigl(\chi Y_F^{(1)},\ \chi Y_F^{(2)},
                K_0a_{0,\chi}^S-J_\theta a_{1,\chi}^S\bigr).
\end{equation}
On these transverse neighborhoods, the matrix is \(\mathsf U_\theta Y\):
\(a_{0,\chi}^S,a_{1,\chi}^S\) and all their spatial derivatives agree
with \(\boldsymbol z_0^S,\boldsymbol z_1^S\) there.
Lemmas~\ref{lem:fast-physical-lift} and \ref{lem:slow-physical-lift}
give \eqref{eq:Yh-limit} on the entire working ball.
Their expansions also give the \(H^r\) assertion from a reduced
\(H^{r+2}\) remainder and the corresponding amplitude bounds.
Uniformity uses the finite bounds of
Theorem~\ref{thm:autonomous-cgo};
Proposition~\ref{prop:appendix-physical-realization} in
Appendix~\ref{app:cgo-input} provides the corresponding \(k=5\)
remainder bounds.

Fix \(K_{\rm pl}\) compactly inside the open transverse region on which \(\chi\) is identically one. The left inverse
\(\mathsf V_\theta^T\mathsf U_\theta=I_3\), Sobolev embedding,
and the inverse bounds give
\begin{equation}\label{eq:proof-graded-rank-estimate}
 \left\|Y^{-1}\bigl(\mathsf V_\theta^T
       e^{\theta\cdot x/h}\mathcal Y_h-Y\bigr)\right\|_{C^0(K_{\rm pl})}
 \leq C_{\rm pl}h.
\end{equation}
For \(C_{\rm pl}h<1/2\), the projected matrix is invertible, hence
\(\mathcal Y_h\) has rank three on \(K_{\rm pl}\), uniformly on compact
direction subcharts.
\end{proof}

The error estimate on the entire working ball permits the exact physical lifts to
enter the boundary identity before taking limits; rank is retained on
compact subsets of the open region where the transverse cutoff is identically one.

Theorem~\ref{thm:graded-frame} completes the dual quotient frame
for the canonical lift \(K=S^TT\). The first nonzero term
\(K_0a_{0,\chi}^S-J_\theta a_{1,\chi}^S\) in
\eqref{eq:proof-localized-graded-amplitude} supplies the additional
direction, with nonvanishing determinant \eqref{eq:graded-determinant}.
The ordinary \(H^5\) reduced remainder controls the differentiated lift
through \eqref{eq:proof-slow-lift-error-bounds}, giving the full-domain
\(O(h)\) physical error in \(H^3\). In
\cite[Proposition~4.2]{OpenAIElasticity2026}, the three columns are
realized through a compatible displacement--actual-divergence recursion
and a correction of the physical displacement. The distinction concerns
the realization and error control of the quotient frame; both
constructions recover all three physical transport directions.

\section{Matrix moments and relative quotient transport}\label{sec:transport}

\subsection{Common extension and enclosing leaf disks}

The two media are first extended through every boundary component so that they
agree outside \(\Omega\). Equality of the DN maps is then transferred to an enclosing ball.
The resulting common exterior contains the contours used to reconstruct
relative quotient transport. Leaf disks are then chosen to enclose the entire
coefficient-difference support, while transport frames are constructed on
cylinders large enough for the exact CGO realizations. Boundary flatness is the condition that
permits the smooth common extension.

\begin{lemma}\label{lem:common-extension}
Let \(\Omega\subset\R^3\) be a bounded \(C^\infty\) domain and
let the smooth real Lam\'e pairs satisfy \eqref{eq:convexity}.
If their full DN maps agree on every component of \(\partial\Omega\),
then the pairs have equal boundary derivatives of every order. They admit smooth real
strongly convex extensions to \(\R^3\) which coincide outside
\(\Omega\) and equal one common constant pair outside a compact set.
In particular, for a sufficiently large enclosing ball
\(B_R=B(0,R)\) with \(\overline\Omega\Subset B_R\), these
extensions equal that constant pair near \(\partial B_R\).
\end{lemma}

\begin{proof}

The equality of all boundary derivatives follows from
Theorem~\ref{thm:appendix-boundary-jets} in
Appendix~\ref{audrev-app-boundary-jets}, see
\eqref{audrev-boundary-all-jets}, for the outward DN convention
\eqref{eq:DN}. In each known boundary-normal chart and at every fixed
finite order, equal localized complete DN symbols give equal normal
traces of both Lam\'e parameters; tangential differentiation gives
all mixed boundary derivatives. The argument includes every boundary component and
the isotropic repeated stable root. Hence the coefficient difference
is flat on \(\partial\Omega\). The cited boundary determination works
\cite{NakamuraUhlmann1995,LinNakamura2017} provide the historical
context; the smooth isotropic conclusion required here follows from
the appendix's flux and symbol recursion.

Write \(a_j=(\lambda_j,\mu_j)\) and define \(d^0\) to equal
\(a_1-a_2\) on \(\overline\Omega\) and zero outside.
Boundary flatness makes \(d^0\) smooth on \(\R^3\).
Extend \(a_1\) smoothly through disjoint sufficiently thin collars
of all boundary components to a neighborhood of \(\overline\Omega\).
The uniform positive lower bounds in \eqref{eq:convexity} and
continuity allow the extension \(\widehat a_1\) to remain in the
open convex cone
\[
 \mathcal K_{\mathrm{ad}}=\{(\lambda,\mu)\in\R^2:
                   \mu>0,\ 3\lambda+2\mu>0\}.
\]
Set \(\widehat a_2=\widehat a_1-d^0\). It equals \(a_2\)
in \(\Omega\) and \(\widehat a_1\) on the exterior side of
every collar. Both extensions are smooth and strongly convex.

Choose a constant \(a_*\in\mathcal K_{\mathrm{ad}}\) and a cutoff
\(0\leq\chi\leq1\), equal to one near \(\overline\Omega\)
and supported in the neighborhood on which the collar extensions
are defined. Set
\[
 a_j^{\rm ext}=\chi\widehat a_j+(1-\chi)a_*
\]
there, and define \(a_j^{\rm ext}=a_*\) beyond the cutoff support,
including the remaining parts of bounded complementary components.
The definitions glue smoothly, and convexity of \(\mathcal K_{\mathrm{ad}}\)
preserves strong convexity. The pairs coincide outside \(\Omega\)
and are constant outside a compact set. A sufficiently large
origin-centered ball contains that set and \(\overline\Omega\)
strictly in its interior.
\end{proof}

\begin{lemma}\label{lem:enlargement}
For the common extensions of Lemma~\ref{lem:common-extension},
equality of the full DN maps on \(\partial\Omega\) implies
\[
 \DN_1^{B_R}=\DN_2^{B_R}.
\]
The complementary region \(B_R\setminus\overline\Omega\)
may be disconnected.
\end{lemma}

\begin{proof}
Take real traces \(f\in H^{1/2}(\partial B_R;\R^3)\) and
\(g\in H^{1/2}(\partial\Omega;\R^3)\), and set
\(A=B_R\setminus\overline\Omega\). We denote by \(u_{f,g}\) the
Dirichlet solution for the common medium on \(A\), with outer
trace \(f\) and interface trace \(g\) on every component of
\(\partial\Omega\). On a component of \(A\) that does not meet
\(\partial B_R\), all boundary values are prescribed by \(g\).
Strong convexity and Korn's inequality give coercivity on
\(H^1_0(A;\R^3)\), and hence existence and uniqueness componentwise.
Its summed energy is
\begin{equation}\label{eq:proof-domain-annular-energy}
 \mathcal E_A(f,g)=\int_A
 \bigl[\lambda_A(\Div u_{f,g})^2
       +2\mu_A\,\eps(u_{f,g}):\eps(u_{f,g})\bigr]\,dx,
\end{equation}
where \((\lambda_A,\mu_A)\) is the common extended pair on \(A\).
The Dirichlet principle gives
\begin{equation}\label{eq:proof-enlargement-dirichlet-minimum}
 \langle\DN_j^{B_R}f,f\rangle
 =\min_{g\in H^{1/2}(\partial\Omega;\R^3)}
       \bigl\{\mathcal E_A(f,g)
                 +\langle\DN_j^\Omega g,g\rangle\bigr\}.
\end{equation}
Indeed, restriction of a global competitor to \(\Omega\) and
all components of \(A\) gives a common interface trace and an
energy at least the sum of the corresponding Dirichlet minima.
Conversely, minimizers with matching trace \(g\) glue to an
\(H^1(B_R;\R^3)\) competitor with that summed energy. The
global Dirichlet minimizer exists by coercivity on
\(H^1_0(B_R;\R^3)\); its interface trace attains the minimum
in \eqref{eq:proof-enlargement-dirichlet-minimum}.
The common exterior energy and equality of the inner DN quadratic
forms therefore give equal outer quadratic forms for every real
\(f\). Bilinear polarization identifies the symmetric DN maps on
real traces, and their complex-linear extensions agree as well.
\end{proof}

Lemma~\ref{lem:enlargement} transfers equality of the physical DN maps to
\(\partial B_R\), including when the complementary region is disconnected.
Henceforth we work with these enlarged-ball maps and suppress their
superscript \(B_R\). We choose \(R_0<R\) so that the two extended media
coincide outside \(\overline{B_{R_0}}\).

For every subsequent CGO application, we fix concentric balls
\[
 \overline{B_R}\Subset\Omega_{\mathrm{CGO}}\Subset\Omega_1\Subset\Omega_0.
\]

The CGO working ball \(\Omega_{\mathrm{CGO}}\) contains the enlarged
measurement ball \(B_R\); the original domain remains denoted by
\(\Omega\). The common background conjugation of
Section~\ref{sec:reduction} puts both reduced systems into the
compact-coefficient form of Theorem~\ref{lem:ER-solvability}.
We choose the source ball \(K_{\mathrm{src}}\Subset\Omega_1\) with interior containing
\(\overline\Omega_{\mathrm{CGO}}\) and both coefficient supports.
The exact correction on \(\Omega_{\mathrm{CGO}}\) gives physical solutions
smooth near \(\overline{B_R}\), where the boundary identity is applied.
The Plemelj disks lie in \(B_R\). The nested CGO balls, the enclosing ball $B_R$, the leaf disks, and the cutoff supports are fixed before the limit \(h\to0\).

The contour on each leaf must enclose its entire coefficient-difference
support. A two-sided annulus in the common exterior will allow the
identity boundary trace of the reconstructed transport to extend smoothly
across that contour.

\begin{lemma}\label{lem:uniform-leaf-disks}
Suppose that the two quotient connections agree outside
\(\overline{B_{R_0}}\), and choose radii
\(R_0<R_\perp<R_1<R\). Locally write a normalized null direction as
\(\theta=e_1+\mathrm i e_2\), where \((e_1,e_2,\nu_\theta)\) is an oriented
orthonormal triple, with \(\nu_\theta=e_1\times e_2\). Here \(x_\perp\) is the
real transverse coordinate. For \(|x_\perp|<R_\perp\), put
\[
 B_{\theta,x_\perp}=\{x_\perp\nu_\theta+x_1e_1+x_2e_2:
                    x_1^2+x_2^2<R_1^2-x_\perp^2\}.
\]
Each disk contains the coefficient-difference support on its leaf and has
a two-sided boundary annulus in the common exterior. The disks and annuli
vary smoothly in the local direction and transverse parameters. For
\(R_0<|x_\perp|<R_\perp\), the two connections agree on the entire leaf.
\end{lemma}
\begin{proof}
On the leaf \(x_\perp=\nu_\theta\cdot x\), we use the complex coordinate
\(z=e_1\cdot x+\mathrm i e_2\cdot x\). The disk is
\(|z|<r(x_\perp)\), where \(r(x_\perp)=\sqrt{R_1^2-x_\perp^2}\). Its boundary lies on
\(\{|x|=R_1\}\), whereas the coefficient-difference support lies in
\(\{|x|\leq R_0\}\). Thus the disk contains all differences on its
leaf. Leaves with \(|x_\perp|>R_0\) miss their support entirely.

Choose \(\delta>0\) such that
\(\max(R_0,R_\perp)<R_1-\delta<R_1+\delta<R\). The intersection of the
leaf with \(R_1-\delta<|x|<R_1+\delta\) is the two-sided annulus
\[
 \sqrt{(R_1-\delta)^2-x_\perp^2}<|z|
                    <\sqrt{(R_1+\delta)^2-x_\perp^2}.
\]
Both radii are positive and smooth for \(|x_\perp|<R_\perp\), and the two
connections agree throughout this annulus. The disk pullback from
the unit disk is the complex-affine map \(w\mapsto r(x_\perp)w\).
On a compact transverse subinterval, the ratios of the annulus
radii to \(r(x_\perp)\) are uniformly separated from one. Hence one
fixed \(\varepsilon>0\) makes the image of
\(1-\varepsilon<|w|<1+\varepsilon\) lie in the annulus for all
parameters in that subchart. The local smooth orthonormal frames
then give the claimed joint smoothness in the direction as well.
\end{proof}

The leaf moments will be obtained from exact fields on
\(\Omega_{\mathrm{CGO}}\), so the disk frames must extend to amplitudes
near the larger CGO domains. We construct the frames on an ambient
cylinder and multiply their lifts by transverse cutoffs; these cutoffs
preserve the homogeneous leading transport because \(D_\theta x_\perp=0\).

\begin{lemma}
\label{lem:uniform-leafwise-frames}
Fix the nested balls used for the system CGO construction, with largest ball
\(\Omega_0=B_{R_{\rm ext}}\), and choose \(L>R_{\rm ext}\).
For every normalized null direction \(\theta_0\) and transverse parameter
\(x_{\perp,0}\in(-R_\perp,R_\perp)\), there are parameter neighborhoods \(U\) of
\(\theta_0\) and \(I\Subset(-R_\perp,R_\perp)\) of \(x_{\perp,0}\) on which the primal quotient frame
\(X_1\), a smooth ambient transport lift \(Z_1\), and the dual
quotient frame \(Y_2\) can be constructed on the whole cylinder
\[
 \{x=x_\perp\nu_\theta+x_1e_1(\theta)+x_2e_2(\theta):
             x_\perp\in I,\ |x_1+\mathrm i x_2|<L\}.
\]
The quotient frames have size \(3\times3\), and \(Z_1\) has size
\(4\times3\). They satisfy
\[
 \begin{gathered}
 (2D_\theta+\mathcal A_{1,\theta})X_1=0,\qquad
 (2D_\theta+A_{1,\theta})Z_1=0,\qquad
 \mathsf U_\theta^TZ_1=X_1,\\
 (2D_\theta+\Gamma_{2,\theta})Y_2=0.
 \end{gathered}
\]
All these objects are smooth jointly in the spatial and local direction
variables. On compact subcylinders and compact direction subcharts,
the quotient determinants are bounded away from zero, and every fixed
finite order of derivatives of the frames, their inverses, and the
ambient lift is bounded. In particular they extend across the closures
of the Plemelj disks and their exterior annuli.

For each fixed direction and every \(b\in C_c^\infty(I)\), the primal
leading amplitudes \(b(x_\perp)g(z)Z_1\), for holomorphic polynomials \(g\), have
exact CGO realizations throughout the physical ball. The columns of \(Y_2\)
have exact normalized physical dual realizations whose leading
ambient amplitude matrix is \(\mathsf U_\theta Y_2\) on a neighborhood
of \(\operatorname{supp}b\) in the transverse variable. The reduced CGO remainder may be taken \(O(h^2)\)
in \(H^5\); the resulting physical adjoint remainder is \(O(h)\) in \(H^3\).
The constants in the reduced $H^5$ remainder bound $O(h^2)$ and the normalized physical-adjoint $H^3$ remainder bound $O(h)$ may depend on the fixed direction, $b$, and $g$.
\end{lemma}
\begin{proof}
Choose a slightly larger planar disk than \(D_L=\{z:|z|<L\}\), where
\(z=e_1(\theta)\cdot x+\mathrm i e_2(\theta)\cdot x\). After pullback by the
smooth orthonormal coordinate map, every transport operator is a smooth
matrix \(\bar\partial\)-operator on this fixed disk, with \((\theta,x_\perp)\)
as parameters. The constructive parameter version in
Lemma~\ref{lem:matrix-dbar-frame} gives a primal quotient frame there.
Lemma~\ref{lem:transport-lift}, using a planar cutoff equal to one near
\(\overline{D_L}\), gives its ambient lift on the same cylinder.
For the dual frame we apply Theorem~\ref{thm:graded-frame} on the same
cylinder. In the local column normalization of \eqref{eq:graded-determinant},
\(\det Y_2=2\mu_2^{-1/2}\); hence its inverse is bounded on compact
subcylinders. Joint smoothness and compactness give the corresponding
bounds for both quotient frames and the ambient lift.

Since \(L>R_{\rm ext}\), multiplication by \(b(x_\perp)g(z)\) and zero
extension in \(x_\perp\) make the primal leading amplitude smooth near
\(\overline\Omega_0\). Its transport equation is preserved because
\(D_\theta b=D_\theta g=0\). Through the common background conjugation,
Theorem~\ref{thm:autonomous-cgo} with direction \(-\theta\) supplies its
positive-phase realization. The covariance law
\eqref{eq:transport-covariance} and the bounded derivatives of
\(\mathcal R^{\pm1}\) preserve the transport equations and Sobolev
orders when returning to the original variables.

Choose \(\chi\in C_c^\infty(I)\) equal to one near \(\supp b\).
Theorem~\ref{thm:graded-frame}, with the commutator correction for the additional column
\eqref{eq:graded-slow-localization}, gives the negative-phase physical
dual realizations with leading matrix \(\mathsf U_\theta Y_2\) on this
transverse neighborhood. Both realizations are exact on \(\Omega_{\mathrm{CGO}}\);
the reduced errors are \(O(h^2)\) in \(H^5\), and the normalized dual
errors are \(O(h)\) in \(H^3\). They may be restricted to \(B_R\).
Parameter smoothness concerns the limiting frames and amplitudes;
no parameter derivative of an exact CGO field is taken.
\end{proof}

\subsection{The physical Betti identity}

The reduced pairing must be tied to physical boundary traces. We use the
factorization of \(L_j\) to pair a primal reduced solution with the
canonical adjoint lift \(S_2^Tu_2\) of a physical solution.
The comparison field \(T_2\boldsymbol z_1\) has the same boundary trace and traction
as \(T_1\boldsymbol z_1\) in the common collar, although it may have a volume source
for \(L_2\). We therefore first state Betti's identity for fields with
such a source. In this subsection, \(\nu_{\partial B_R}\) is the outward unit normal to
\(\partial B_R\).

\begin{lemma}
\label{lem:sobolev-betti-dn}
For smooth real-valued Lam\'e coefficients strongly convex on
\(\overline{B_R}\), the DN map
\(H^{1/2}(\partial B_R;\C^3)\to H^{-1/2}(\partial B_R;\C^3)\)
is bounded and complex-bilinearly symmetric. More generally, if
\(u,v\in H^1(B_R;\C^3)\) and \(Lu,Lv\in L^2(B_R;\C^3)\)
in the distributional sense, their weak tractions \(t_u,t_v\)
belong to \(H^{-1/2}(\partial B_R;\C^3)\), and
\begin{equation}\label{eq:proof-sobolev-betti}
 \int_{B_R}(u^TLv-v^TLu)\,dx
 =\langle t_v,\operatorname{Tr} u\rangle_{\partial B_R}
  -\langle t_u,\operatorname{Tr} v\rangle_{\partial B_R},
\end{equation}
where \(\operatorname{Tr}:H^1(B_R;\C^3)\to H^{1/2}(\partial B_R;\C^3)\)
is the trace. Boundary duality is written with the traction in
the first slot and is complex-bilinear. Thus the identity also
applies when only one field solves the homogeneous equation.
\end{lemma}

\begin{proof}
Define
\[
 a(u,w)=\int_{B_R}
 [\lambda(\Div u)(\Div w)+2\mu\,\eps(u):\eps(w)]\,dx.
\]
Using strong convexity and Korn's inequality, we obtain, for some \(c_0>0\),
\begin{equation}\label{eq:proof-betti-coercivity}
 a(w,\overline w)\geq c_0\|w\|_{H^1}^2,
 \qquad w\in H^1_0(B_R;\C^3).
\end{equation}
Using \eqref{eq:proof-betti-coercivity} and a bounded trace right inverse \(E\),
we obtain the Dirichlet estimate
\begin{equation}\label{eq:proof-betti-dirichlet-estimate}
 \|u_f\|_{H^1}\leq C\|f\|_{H^{1/2}}.
\end{equation}
For \(u\in H^1\) with \(Lu\in L^2\), we define
\begin{equation}\label{eq:proof-weak-traction-definition}
 \langle t_u,g\rangle
 =a(u,W)+\int_{B_R}(Lu)^TW\,dx,\qquad \operatorname{Tr} W=g.
\end{equation}
The distributional equation makes the definition
\eqref{eq:proof-weak-traction-definition} independent of the lifting:
the difference of two liftings has zero trace and its contribution
vanishes by density in \(H^1_0\). Taking \(W=Eg\) in
\eqref{eq:proof-weak-traction-definition}, we obtain
\begin{equation}\label{eq:proof-betti-traction-estimate}
 \|t_u\|_{H^{-1/2}}\leq C(\|u\|_{H^1}+\|Lu\|_{L^2}).
\end{equation}
For smooth fields, \(t_u\) is the outward traction \(\sigma(u)\nu_{\partial B_R}\).
In \eqref{eq:proof-weak-traction-definition}, we use \(u\) as a lifting
for \(t_v\), and \(v\) for \(t_u\). By subtraction and the symmetry
of \(a\), we obtain \eqref{eq:proof-sobolev-betti}. 
For homogeneous Dirichlet solutions, \(t_{u_f}=\DN f\).
The estimates \eqref{eq:proof-betti-dirichlet-estimate} and
\eqref{eq:proof-betti-traction-estimate}, with \(Lu_f=0\), give the
boundedness of \(\DN\); applying \eqref{eq:proof-sobolev-betti} to two
homogeneous solutions gives its complex-bilinear symmetry.

\end{proof}

Betti's identity now converts equality of the elastic DN maps into a
reduced volume pairing. The adjoint test is \(y_2=S_2^Tu_2\), with
\(L_2u_2=0\), so both boundary traces in this conversion belong to
physical elasticity solutions.

\begin{proposition}\label{prop:primal-adjoint}
Assume that the two Lam\'e pairs coincide on a neighborhood of
\(\partial B_R\) and that \(\DN_1=\DN_2\) as bounded operators
\(H^{1/2}(\partial B_R;\C^3)\to H^{-1/2}(\partial B_R;\C^3)\).
Let \(M_1\boldsymbol z_1=0\), set \(u_1=T_1\boldsymbol z_1\), and let \(u_2\) solve
\(L_2u_2=0\). Assume that \(\boldsymbol z_1\) and \(u_2\) are smooth near
\(\overline{B_R}\), as holds for the CGO constructions on a larger ball.
Put \(y_2=S_2^Tu_2\). All volume integrals in this proposition are
over \(B_R\). Then
\begin{equation}\label{eq:primal-adjoint}
 \int y_2^T(M_1-M_2)\boldsymbol z_1\,dx=0.
\end{equation}
\end{proposition}

\begin{proof}
The factorization \eqref{eq:factorization} and its bilinear transpose give
\[
 L_1u_1=S_1M_1\boldsymbol z_1=0,\qquad
 M_2^Ty_2=M_2^TS_2^Tu_2=T_2^TL_2u_2=0.
\]
Write \(w=M_2\boldsymbol z_1=(w',w_4)^T\), with \(w'\in\C^3\).
In the common boundary collar, \(M_2=M_1\), so \(w=0\). Thus
\[
 \int y_2^T(M_1-M_2)\boldsymbol z_1\,dx
 =-\int (S_2^Tu_2)^Tw\,dx.
\]
Using \eqref{eq:Stranspose}, integration by parts gives
\begin{equation}\label{eq:proof-adjoint-transfer}
 \int (S_2^Tu_2)^Tw\,dx
 =\int u_2^TS_2w\,dx
  -\int_{\partial B_R}(u_2\cdot\nu_{\partial B_R})
                       \frac{w_4}{1-\kappa_2}\,dS.
\end{equation}
Since \(w_4=0\) in the collar, the boundary term vanishes, and
\eqref{eq:factorization} for the second medium gives
\[
 \int y_2^T(M_1-M_2)\boldsymbol z_1\,dx
 =-\int u_2^TS_2M_2\boldsymbol z_1\,dx
 =-\int u_2^TL_2v\,dx,\qquad v=T_2\boldsymbol z_1.
\]

The smooth fields \(u_2,v\) satisfy Lemma~\ref{lem:sobolev-betti-dn}.
Set \(f=\operatorname{Tr} u_1\) and \(g=\operatorname{Tr} u_2\). In the common collar,
\(T_2=T_1\) and \(v=u_1\); equality of the coefficients and fields gives
\[
 \operatorname{Tr} v=f,\qquad
 t_v^{(2)}=t_{u_1}^{(1)}=\DN_1f,\qquad
 t_{u_2}^{(2)}=\DN_2g.
\]
The traction superscript specifies its medium. The comparison field
\(v\) need not satisfy \(L_2v=0\): \(L_2v\in L^2\) and \(L_2u_2=0\)
suffice for Betti's identity \eqref{eq:proof-sobolev-betti}, which gives
\[
 \begin{aligned}
 -\int u_2^TL_2v\,dx
 &=-\langle\DN_1f,g\rangle
       +\langle\DN_2g,f\rangle\\
 &=\langle(\DN_2-\DN_1)f,g\rangle=0,
 \end{aligned}
\]
by bilinear symmetry of \(\DN_2\) and equality of the DN maps.
These Sobolev duality pairings satisfy
\[
 |\langle\DN_jf,g\rangle|
 \leq\|\DN_j\|_{H^{1/2}\to H^{-1/2}}
       \|f\|_{H^{1/2}}\|g\|_{H^{1/2}}.
\]

Under the collar hypotheses, every exact physical pair satisfies
\eqref{eq:primal-adjoint}. Each fixed-\(h\) CGO pairing is therefore
zero before phase cancellation. The limit uses phase-removed interior
bounds, without uniform bounds on exponentially weighted boundary traces.

For the common background conjugation, set
\[
 \widetilde{\boldsymbol z}_1=\mathcal R^{-1}\boldsymbol z_1,\qquad
 \widetilde y_2=\mathcal R^Ty_2,\qquad
 \widetilde M_j=\mathcal R^{-1}M_j\mathcal R.
\]
The identity of differential operators
\(\widetilde M_1-\widetilde M_2
  =\mathcal R^{-1}(M_1-M_2)\mathcal R\) gives
\begin{equation}\label{eq:proof-exact-moment-covariance}
 \int\widetilde y_2^T(\widetilde M_1-\widetilde M_2)
                         \widetilde{\boldsymbol z}_1\,dx
 =\int y_2^T(M_1-M_2)\boldsymbol z_1\,dx=0.
\end{equation}
This exact covariance includes the zeroth-order terms and derivatives
of \(\mathcal R\); \eqref{eq:moment-covariance} gives its leading
first-order counterpart.
\end{proof}

\subsection{Complete matrix moments from physical data}

We apply the exact physical pairing to opposite-phase CGO fields
and pass to the first-order coefficient difference to derive the
matrix-moment identities. The same invertible primal
and physical dual frames are used for every pair of columns and every
holomorphic polynomial. Transverse tests localize the resulting matrix
identity to each enclosing leaf disk.

Fix a normalized null direction \(\theta=e_1+\mathrm i e_2\) and use
the orthonormal leaf coordinates
\[
 \nu_\theta=e_1\times e_2,\qquad
 z=e_1\cdot x+\mathrm i e_2\cdot x,\qquad x_\perp=\nu_\theta\cdot x.
\]
Here \(\nu_\theta\) is the leaf normal and \(x_\perp\) is the real transverse
coordinate. The direction and its spinor frame are fixed when
differentiating in \(x\). For medium \(j\), we write
\(\mathcal A_{j,\theta}\) for
\eqref{eq:primal-quotient-connection}, suppress this fixed direction
in the notation, and set
\[
 P_j=\frac12\mathcal A_j.
\]
Since \(D_\theta=2\partial_{\bar z}\), division of the primal
transport equation by two gives
\begin{equation}\label{eq:leaf-primal}
 2\partial_{\bar z}X_j+P_jX_j=0,
\end{equation}
and the dual equation, with \(\Gamma_j=-\mathcal A_j^T\), gives
\begin{equation}\label{eq:leaf-dual}
 2\partial_{\bar z}Y_j-P_j^TY_j=0.
\end{equation}
The two media use the same coefficient-independent quotient basis
\(\mathsf U_\theta\) and right inverse \(\mathsf V_\theta\).

\begin{corollary}
\label{lem:DN-to-moment}
Let \(D=\sum_k A^k\partial_k+B\), \(\boldsymbol z_h\), and \(y_h\) satisfy
Lemma~\ref{lem:CGO-product-limit} on \(B_R\), with the coefficients of
\(D\) compactly supported there. Its \(O(h)\) estimate holds on \(B_R\),
and both integrals are unchanged by enlarging the integration domain
through a region where the coefficients of \(D\) vanish, with the fields
held fixed.
If \(D=M_1-M_2\), \(M_1\boldsymbol z_h=0\), and \(y_h=S_2^Tu_{2,h}\) with
\(L_2u_{2,h}=0\), under the common-collar and exact-solution hypotheses
of Proposition~\ref{prop:primal-adjoint}, equality of the elastic DN maps
makes the exact integral zero and hence
\[
 \int_{B_R}y^T A_\theta a\,dx=0.
\]
The assertion applies to every pair of columns of matrix-valued fields.
\end{corollary}

\begin{proof}

In Lemma~\ref{lem:CGO-product-limit}, the exponential phases cancel in
the volume pairing before the limit is taken. For a physical adjoint
field, Proposition~\ref{prop:primal-adjoint} sets that exact pairing
to zero at each fixed \(h\). Its \(O(h)\) interior error therefore
bounds the leading moment. Vanishing of the coefficients of \(D\)
proves the integration-domain assertion, and applying the scalar
pairing to each pair of columns proves the matrix assertion.
\end{proof}

The corollary gives one moment for each admissible physical adjoint test.
To recover a full matrix identity, we use the two principal and one additional
columns of Theorem~\ref{thm:graded-frame}, and realize their fixed leading
frame together on the cylinder of Lemma~\ref{lem:uniform-leafwise-frames}.

\begin{theorem}
\label{thm:physical-nine-moments}\label{prop:quotient-moment}
Let the two smooth strongly convex Lam\'e pairs have the common
exterior extensions and fixed CGO domains chosen above, with their
coefficient difference supported in \(\overline{B_{R_0}}\), and assume
that their physical DN maps on \(B_R\) agree. Fix a normalized null
direction and one of the local transverse intervals
\(I\Subset(-R_\perp,R_\perp)\) of
Lemma~\ref{lem:uniform-leafwise-frames}.

There are invertible primal and dual quotient frames \(X_1,Y_2\)
on the ambient cylinder, chosen before any test functions, such that
the dual frame consists of two principal and one additional canonical physical
adjoint polarizations. In the normalizations of
Theorem~\ref{thm:graded-frame},
\[
 \det Y_2=2\mu_2^{-1/2}\ne0.
\]
For every fixed \(b\in C_c^\infty(I)\) and holomorphic polynomial
\(g\), these frames have exact realizations on the fixed CGO ball:
the primal reduced fields solve \(M_1\boldsymbol z_{1,h}=0\), and every dual
column is \(S_2^Tu_{2,h}\) with \(L_2u_{2,h}=0\). Their reduced
remainders are \(O(h^2)\) in \(H^5\), and the normalized physical adjoint
remainders are \(O(h)\) in \(H^3\). On a transverse neighborhood of
\(\supp b\), the leading primal and dual amplitudes are respectively
\(b(x_\perp)g(z)Z_1\) and \(\mathsf U_\theta Y_2\), with
\(\mathsf U_\theta^TZ_1=X_1\). The exact physical adjoint matrix has
rank three on compact subsets of this region for sufficiently small
\(h\).

The exact primal and three physical-adjoint CGO realizations give all entries of the matrix identity
\begin{equation}\label{eq:3Dmoment}
 \int b(x_\perp)g(z)Y_2^T(P_1-P_2)X_1\,dx=0.
\end{equation}
The integral is over \(B_R\), with zero extension outside the
transverse support of \(b\). On every enclosing disk
\(B=B_{\theta,x_\perp}\), \(x_\perp\in I\), they also give
\begin{equation}\label{eq:leaf-moment}
 \int_B g(z)Y_2^T(P_1-P_2)X_1\,dA=0.
\end{equation}
Both are full \(3\times3\) matrix identities for the same quotient
frames and all holomorphic polynomials. They give nine scalar identities
for each polynomial weight, rather than only finitely many weighted
moment conditions. The leafwise identity holds
at every transverse parameter. The only boundary map used is the
physical elasticity DN map.
\end{theorem}

\begin{proof}
Theorem~\ref{thm:graded-frame} supplies canonical physical adjoint
columns: the principal pair spans \(\mathcal F_\theta\), and
\eqref{eq:Yslow-mod} and \eqref{eq:slow-polarization-normalization} make
the additional column independent, with determinant
\eqref{eq:graded-determinant}. Together with
Lemma~\ref{lem:uniform-leafwise-frames}, it gives \(X_1,Y_2,Z_1\)
on the cylinder, fixed before \(b,g\). The principal columns give the
first two rows and the additional column the third row of
\(Y_2^T(P_1-P_2)X_1\); all nine entries use these same frames.
The rank of the exact fields follows from \eqref{eq:Yh-limit}, the bounded
inverse of \(Y_2\), and Sobolev embedding, as in
\eqref{eq:proof-graded-rank-estimate}.
\space{}For fixed \(\theta,b,g\), Lemma~\ref{lem:uniform-leafwise-frames}
gives the \(4\times3\) amplitude \(a=b(x_\perp)g(z)Z_1\), smoothly extended
by zero in \(x_\perp\) near the full fixed CGO domain, with
\((2D_\theta+A_{1,\theta})a=0\). Apply
Theorem~\ref{thm:autonomous-cgo} columnwise through the common
background conjugation with direction \(-\theta\). Its leading
transport is the negative of this homogeneous equation, and its
first correction in the original variables satisfies
\[
 (2D_\theta+A_{1,\theta})a_1=-M_1a
 \quad\hbox{near }\overline\Omega_{\mathrm{CGO}},
\]
The exact primal fields, with their two-order remainder controlled by
Proposition~\ref{prop:appendix-physical-realization}, are
\begin{equation}\label{eq:proof-moment-primal-realization}
 \begin{gathered}
 \boldsymbol z_{1,h}=e^{\theta\cdot x/h}(a+h a_1+r_h),
 \qquad M_1\boldsymbol z_{1,h}=0\quad\hbox{in }\Omega_{\mathrm{CGO}},\\
 \|r_h\|_{H^5(\Omega_{\mathrm{CGO}})}=O(h^2).
 \end{gathered}
\end{equation}
The fixed CGO ball \(\Omega_{\mathrm{CGO}}\) contains \(\overline{B_R}\).

Choose a transverse cutoff equal to one near \(\supp b\).
Lemma~\ref{lem:uniform-leafwise-frames} realizes the fixed dual frame
as a \(4\times3\) matrix of exact physical adjoint fields
\begin{equation}\label{eq:proof-moment-dual-realization}
 y_{2,h}=e^{-\theta\cdot x/h}(\widehat y_2+q_h),
 \qquad \|q_h\|_{H^3(\Omega_{\mathrm{CGO}})}=O(h),
\end{equation}
Here \(\widehat y_2=\mathsf U_\theta Y_2\) near \(\supp b\).
Every column is exactly \(S_2^Tu_{2,h}\), with \(L_2u_{2,h}=0\):
the two normalized physical fields are \(-hT_2\boldsymbol z_h^{F,j}\), and the unscaled field
is \(T_2\boldsymbol z_h^S\). The physical-adjoint column normalizations $(-h,-h,1)$ preserve
\eqref{eq:primal-adjoint}. The cutoff and corrections may depend
on \(b\), while their leading columns on its support stay fixed.
Lemma~\ref{lem:slow-physical-lift} already accounts for the additional
lift's \(h^{-1}\) factor and gives its unscaled \(O(h)\) conjugated
\(H^3\) error. Outside the open region where the cutoff is identically one \(a=0\), so the leading product
vanishes.

The phase-removed principal adjoint lifts
\(K_2\boldsymbol z_h^{F,j}\), \(j=1,2\), are of order
\(O(h^{-1})\), whereas the supplementary lift
\(K_2\boldsymbol z_h^S\) is of order \(O(1)\).
Since the conjugated first-order operator difference
contributes a factor \(h^{-1}\), the unnormalized
pairings have respective orders
\(O(h^{-2}),O(h^{-2}),O(h^{-1})\).
Combining the physical normalization factors
\((-h,-h,1)\) with the common factor \(h\) in
\eqref{eq:proof-matrix-moment-limit} gives
\begin{equation}\label{eq:physical-moment-row-normalization}
 h\,\operatorname{diag}(-h,-h,1)
 =\operatorname{diag}(-h^2,-h^2,h).
\end{equation}
All three rows thus admit finite limits, while the
rescaled Betti identities remain exactly zero for
every \(h>0\).

For each fixed direction, the primal and adjoint frames
are chosen independently of the test functions \(b\) and
\(g\). The limit \(h\to0\) is taken for each fixed test
pair, yielding the matrix identities for all such tests
without requiring estimates uniform in the polynomial
degree or direction derivatives of the exact fields.
Only the invertibility of the frames is needed; no
nondegeneracy assumption on the coefficient difference
is imposed.

In particular, the third row is obtained from the first
nonzero physical lift \(K_0\boldsymbol z_0^S-J_\theta \boldsymbol z_1^S\), not by prescribing an
independent row of auxiliary Cauchy data. The same invertible \(Y_2\)
is used for all three primal columns and for every fixed polynomial
test. Each entry therefore follows from the exact elasticity Betti
identity \eqref{eq:primal-adjoint}, together with the whole-domain
\(O(h)\) limit in \eqref{eq:proof-matrix-moment-limit}.

All fields are smooth near \(\overline{B_R}\). In the original
variables, the operator difference is
\[
 M_1-M_2=\sum_{k=1}^3(A_{1,k}-A_{2,k})\partial_k+(B_1-B_2),
\]
Its coefficients are compactly supported in \(B_R\). The remainder
estimates \eqref{eq:proof-moment-primal-realization} and
\eqref{eq:proof-moment-dual-realization}, restricted to \(B_R\),
and the fixed \(H^3\) amplitude bounds satisfy
Corollary~\ref{lem:DN-to-moment}. The error estimate \eqref{eq:CGO-product-limit}
controls corrections on the whole integration domain,
including where \(a=0\). Applied to all nine column pairs, the
corollary gives, in a fixed matrix norm,
\begin{equation}\label{eq:proof-matrix-moment-limit}
 \begin{aligned}
 \Big\|{}&h\int_{B_R}y_{2,h}^T(M_1-M_2)\boldsymbol z_{1,h}\,dx\\
 &-\int_{B_R}b(x_\perp)g(z)Y_2^T\mathsf U_\theta^T
       (A_{1,\theta}-A_{2,\theta})Z_1\,dx\Big\|\leq Ch.
 \end{aligned}
\end{equation}
The first integral is zero by \eqref{eq:primal-adjoint}, so the
second vanishes as \(h\to0\). Write
\(Z_1=\mathsf V_\theta X_1+e_\theta\alpha\). The common quotient
basis, \((A_{1,\theta}-A_{2,\theta})e_\theta=0\), and
\eqref{eq:primal-quotient-connection} give
\begin{equation}\label{eq:proof-moment-quotient-contraction}
 \mathsf U_\theta^T(A_{1,\theta}-A_{2,\theta})Z_1
 =(\mathcal A_{1,\theta}-\mathcal A_{2,\theta})X_1
 =2(P_1-P_2)X_1.
\end{equation}
Thus \eqref{eq:proof-matrix-moment-limit} and
\eqref{eq:proof-moment-quotient-contraction} give
\[
 0=\int_{B_R}b(x_\perp)g(z)Y_2^T
       (\mathcal A_{1,\theta}-\mathcal A_{2,\theta})X_1\,dx
   =2\int_{B_R}b(x_\perp)g(z)Y_2^T(P_1-P_2)X_1\,dx,
\]
proving \eqref{eq:3Dmoment}. 
The common background conjugation preserves the leading moment by
\eqref{eq:moment-covariance}.

Fix a polynomial \(g\) and the enclosing disks of
Lemma~\ref{lem:uniform-leaf-disks}.
Identify \(B_{x_\perp}=\{z:|z|<r(x_\perp)\}\), where
\(r(x_\perp)=\sqrt{R_1^2-x_\perp^2}\), with the unit disk \(\mathbb D\)
by \(\Phi_{x_\perp}(w)=r(x_\perp)w\).
Set
\[
\begin{aligned}
F_g(x_\perp,z)&=g(z)Y_2^T(P_1-P_2)X_1,\\
H(x_\perp)&=\int_{B_{x_\perp}}F_g(x_\perp,z)\,dA(z).
\end{aligned}
\]
The pullback \(\widehat F_{g,x_\perp}(w)=r(x_\perp)^2F_g(x_\perp,r(x_\perp)w)\)
is jointly smooth on \(J\times\overline{\mathbb D}\) for every
\(J\Subset I\); hence \(H\) is continuous on \(I\).
Each disk contains the leafwise difference support. Using
\(dx=dA\,d x_\perp\), \eqref{eq:3Dmoment} gives
\(\int_I b(x_\perp)H(x_\perp)\,d x_\perp=0\) for every \(b\in C_c^\infty(I)\).
Entrywise distributional vanishing and continuity imply
\(H(x_\perp)=0\) for every \(x_\perp\in I\).

This proves \eqref{eq:leaf-moment} at every \(x_\perp\in I\).
The local intervals cover all enclosing disks, and the integrand
vanishes on leaves outside the difference support.
\end{proof}

Remark~\ref{rem:audit-vii-low-regularity-moments} shows that
only the \(k=1\) two-level CGO bounds are needed for the preceding
nine moment identities. The stronger \(H^5\) construction in the proof
also establishes the separately stated pointwise rank of the exact
physical frame.

Theorem~\ref{thm:physical-nine-moments} gives all nine scalar
moment identities in \eqref{eq:leaf-moment} for every holomorphic
polynomial, with the same \(X_1,Y_2\) fixed before the tests.
The Cauchy--Green argument therefore supplies the contour conditions
\eqref{eq:boundary-moments} for one product \(\mathsf W=Y_2^TX_1\),
before its Cauchy--Plemelj factor is constructed.
Invertibility of that factor uses both quotient frames on the entire
closed disk. The construction in
\cite[Proposition~5.3]{OpenAIElasticity2026} instead transfers physical
solutions and takes weak limits to obtain second-medium columns from
an invertible first-medium frame, without an initially invertible
second-medium frame. The two constructions use the same physical
boundary data; the difference is their passage to the relative transport.

\begin{remark}
\label{rem:ER-leading-and-slow}
The identity \(H_2=0\) in
\cite[Section~3, equations (3.4)--(3.5) and the paragraph following (3.5)]{EskinRalstonSystems2004}
depends on the two contracted leading reduced amplitudes
\((\theta\cdot r_0,s_0)\), where \(r_0,s_0\) are the vector and
scalar amplitudes in that reference. The relation obtained after
its equation (3.5) is equivalent to five scalar equations and,
by itself, admits nonconstant candidate coefficients even when
the first pair is constant.

The same work also derives the full sequence of asymptotic
invariants and uses \(H_0=0\) in its small-gradient argument
\cite[Section~3]{EskinRalstonSystems2004}; see also
\cite[Section~5]{EskinRalston2002}.
Here higher-order terms in canonical physical adjoint lifts are
organized into the complete weighted quotient pairing.
The two principal columns \eqref{eq:Jfast} span \(\mathcal F_\theta\)
and produce the first two rows for the ordered frame of
Theorem~\ref{thm:physical-nine-moments}. When
\(J_\theta \boldsymbol z_0^S=0\), the unscaled lift retains
\(Y_S=K_0\boldsymbol z_0^S-J_\theta \boldsymbol z_1^S\); its class
\eqref{eq:Yslow-mod} lies outside \(\mathcal F_\theta\).

Write \(Y_2=(y_{F,1},y_{F,2},y_S)\) in quotient coordinates,
so that the ambient additional amplitude is \(\mathsf U_\theta y_S\).
For the standard basis vectors \(\mathsf e_k\) of \(\C^3\),
the additional polarization supplies the three additional identities

\begin{equation}\label{eq:slow-third-row-moments}
 \int_B g(z)y_S^T(P_1-P_2)X_1\mathsf e_k\,dA=0,
 \qquad k=1,2,3.
\end{equation}

Combining the third-row identities in
\eqref{eq:slow-third-row-moments} with the two principal
rows yields the full matrix identity \eqref{eq:leaf-moment},
with both transport weights retained.
The derivation relies on the exact primal--physical-adjoint
identity \eqref{eq:primal-adjoint} and the physical elastic
Dirichlet-to-Neumann maps.
The resulting complete matrix identities provide the contour
compatibility required for the subsequent Cauchy--Plemelj
construction.

The additional row completes the two-dimensional image of the leading physical adjoint lift
in this fixed-frame formulation. All nine scalar entries of the matrix moment identity \eqref{eq:leaf-moment} come from the same physical elasticity DN map. The comparison made here is
with the contracted \(H_2\) identity; the argument does not require
a term-by-term identification with the full asymptotic sequence
\(\{H_j\}\), or independence from that sequence.

\end{remark}

\subsection{Relative transport from complete moments}

Relative transport is now reconstructed from the full physical matrix
moments. The passage through boundary moments, a holomorphic Cauchy
factor, and an exterior-normalized relative field follows
\cite[Section~3.3, proof of Theorem~3.4,
equations (3.15)--(3.18)]{Cekic2025}. The closed-disk invertibility and parameter-regularity arguments
are given for the rank-three physical quotient, followed by verification
of its bundle transitions and fixed-\(x\) holomorphic dependence. The matrix transpose throughout
is bilinear, with the sign prescribed by \eqref{eq:leaf-dual}.

All disk boundaries in this subsection are oriented counterclockwise,
and \(dA\) is the positive Euclidean area measure in the leaf coordinate.
The square frame product is denoted by \(\mathsf W\); the symbol
\(Q\) is reserved for the \(3\times2\) spinor matrix in \eqref{eq:qspinor}.

\begin{lemma}
\label{lem:cauchy-green-moment}
Let \(B\subset\C\) be a disk and let
\(\mathsf W\in C^1(\overline B;M_m(\C))\). Let \(X_1,Y_2,P_1,P_2\) be
smooth \(m\times m\) matrix fields near \(\overline B\). If
\begin{equation}\label{eq:proof-cauchy-green-source}
 2\partial_{\bar z}\mathsf W=-Y_2^T(P_1-P_2)X_1
\end{equation}
and
\begin{equation}\label{eq:proof-cauchy-green-volume-moments}
 \int_B g(z)Y_2^T(P_1-P_2)X_1\,dA=0
\end{equation}
for every holomorphic polynomial \(g\), then
\[
 \oint_{\partial B}g(z)\mathsf W(z)\,dz=0
 \qquad\text{for every holomorphic polynomial }g.
\]
The assertion is entrywise and remains valid for jointly smooth
parameter families after complex-affine pullback to a fixed disk.
\end{lemma}

\begin{proof}
With the stated orientation, the Cauchy--Green formula and
\eqref{eq:proof-cauchy-green-source}--\eqref{eq:proof-cauchy-green-volume-moments}
give
\[
 \begin{aligned}
 \oint_{\partial B}g\mathsf W\,dz
 &=2\mathrm i\int_B\partial_{\bar z}(g\mathsf W)\,dA\\
 &=-\mathrm i\int_BgY_2^T(P_1-P_2)X_1\,dA=0.
 \end{aligned}
\]
Here \(\partial_{\bar z}g=0\), and all identities hold entrywise.
For \(\Phi_p(w)=a_pw+b_p\), the contour integral becomes
\[
 a_p\oint_{\partial\mathbb D}
       g(\Phi_p(w))\mathsf W_p(\Phi_p(w))\,dw.
\]
Thus the contour differential contributes \(a_p\), whereas the
area differential contributes \(|a_p|^2\). On compact parameter
subcharts the pulled-back integrand and all its parameter derivatives
are continuous on a fixed circle, so differentiation under that
integral is justified.
\end{proof}

Boundary moments determine a holomorphic Cauchy factor with boundary
trace \(\mathsf W\). To use its inverse, we must also exclude
interior determinant zeros. Invertibility of \(\mathsf W\) on the
whole closed disk makes its boundary determinant loop contractible
in \(\C^*\); the argument principle then gives the required
nonvanishing. For \(\mathsf W=Y_2^TX_1\), the closed-disk assumption
follows from the invertible quotient frames.

\begin{lemma}\label{lem:matrix-plemelj}
Let \(B\subset\C\) be a disk and let
\(\mathsf W\in C^\infty(\overline B;\mathrm{GL}(m,\C))\). If
\begin{equation}\label{eq:proof-plemelj-contour-moments}
 \oint_{\partial B}z^n\mathsf W(z)\,dz=0\qquad(n=0,1,2,\ldots),
\end{equation}
then the Cauchy factor
\begin{equation}\label{eq:proof-plemelj-cauchy-factor}
 \mathsf R(z)=\frac{1}{2\pi\mathrm i}\oint_{\partial B}
              \frac{\mathsf W(\zeta)}{\zeta-z}\,d\zeta,\qquad z\in B,
\end{equation}
is holomorphic in \(B\), extends smoothly and invertibly to
\(\overline B\), and has boundary trace \(\mathsf W|_{\partial B}\).

Moreover, in the Euclidean matrix operator norm,
\begin{equation}\label{eq:plemelj-maximum-inverse-bound}
 \sup_{z\in\overline B}\|\mathsf R(z)^{-1}\|_{\mathrm{op}}
 \leq \sup_{z\in\partial B}\|\mathsf W(z)^{-1}\|_{\mathrm{op}},
 \qquad
 \min_{z\in\overline B}|\det\mathsf R(z)|
 \geq\min_{z\in\partial B}|\det\mathsf W(z)|>0.
\end{equation}

For a jointly smooth local family under smooth complex-affine
pullbacks from a fixed disk, \(\mathsf R\) and \(\mathsf R^{-1}\) are jointly smooth
on the pulled-back closed disks. Their derivatives of each fixed
finite order are bounded on compact parameter subcharts.
\end{lemma}

\begin{proof}
For \(|z|\) larger than every \(|\zeta|\) on \(\partial B\), we expand
\((\zeta-z)^{-1}=-\sum_{n\geq0}\zeta^nz^{-n-1}\). The moments
\eqref{eq:proof-plemelj-contour-moments} make the exterior Cauchy
transform zero there, hence throughout the connected exterior.
The Plemelj jump formula therefore gives
\(\mathsf R|_{\partial B}=\mathsf W|_{\partial B}\).

For the parameter assertion, we denote the smooth matrix family by
\(\mathsf W_p\) and its complex-affine pullbacks from the unit disk by \(\Phi_p\).
Set \(\mathsf W_p^\sharp=\mathsf W_p\circ\Phi_p\) and denote by \(\mathsf R_p\)
its Cauchy factor on that disk. An affine change of contour
preserves polynomial moments, so the negative Fourier modes vanish.
The remaining coefficients and Cauchy factor are

\begin{equation}\label{eq:plemelj-fourier-series}
 \mathsf W_{n,p}=\frac{1}{2\pi}\int_0^{2\pi}
       \mathsf W_p^\sharp(e^{it})e^{-int}\,dt,\qquad
 \mathsf R_p(w)=\sum_{n\geq0}\mathsf W_{n,p}w^n.
\end{equation}

Repeated integration by parts gives, for every \(N\) and parameter
multi-index \(\alpha\), uniformly on compact parameter subcharts,
\begin{equation}\label{eq:proof-plemelj-fourier-decay}
 \|\partial_p^\alpha\mathsf W_{n,p}\|
       \le C_{\alpha,N}(1+n)^{-N},\qquad n\geq0.
\end{equation}
A spatial derivative of order \(r\) contributes at most
\(C_r(1+n)^r\); taking \(N>r+1\) makes every fixed mixed spatial
and parameter derivative uniformly summable on the closed disk.
Thus \(\mathsf R_p\) is jointly smooth up to the contour, also for a fixed
matrix viewed as a constant parameter family.

The nonvanishing continuous function \(\det\mathsf W\) on
\(\overline B\) contracts its boundary loop in \(\C^*\), so its
winding number is zero. The boundary trace gives the same loop for
\(\det \mathsf R\). Interior circles sufficiently close to \(\partial B\)
have that winding number because the boundary determinant is
nonzero. The argument principle therefore excludes interior zeros
of \(\det \mathsf R\), proving invertibility on the closed disk.

For fixed Euclidean unit vectors \(u,v\in\C^m\), the scalar function
\(z\mapsto v^*\mathsf R(z)^{-1}u\) is holomorphic in \(B\).
The maximum-modulus principle and the equality
\(\mathsf R=\mathsf W\) on \(\partial B\) imply
\[
 |v^*\mathsf R(z)^{-1}u|
 \leq \sup_{\partial B}\|\mathsf W^{-1}\|_{\mathrm{op}}.
\]
Taking the supremum over \(u,v\) proves the first estimate in
\eqref{eq:plemelj-maximum-inverse-bound}. The minimum determinant
bound follows by applying the same principle to the holomorphic
function \(1/\det\mathsf R\). Both inequalities apply pointwise
in every smooth parameter family, giving parameter-uniform inverse
bounds whenever the boundary inverse-frame bounds are uniform.

In a compact parameter subchart, continuity and nonvanishing of
\(\det \mathsf R_p\) on the product with the closed reference disk give a
uniform positive lower bound. The adjugate formula then gives joint
smoothness and every stated finite derivative bound for \(\mathsf R_p^{-1}\).
\end{proof}

The inverse Cauchy factor is now available up to the contour, with
mixed spatial and parameter derivative bounds on compact subcharts.
The mixed spatial and real-parameter derivative bounds established
above apply to the Cauchy factor $\mathsf R_p$ and its inverse
$\mathsf R_p^{-1}$. The smooth leading transport frames $X_1,Y_2$
supply the fixed data entering this construction. No parameter
derivative of an exact CGO field is required.

On each enclosing leaf disk \(B\), we fix the invertible quotient
frames \(X_1,Y_2\) constructed above and set
\begin{equation}\label{eq:Qdef}
 \mathsf W=Y_2^TX_1.
\end{equation}
This product is invertible on a neighborhood of \(\overline B\).
The product rule and \eqref{eq:leaf-primal}--\eqref{eq:leaf-dual} give
\begin{equation}\label{eq:dbarQ}
 2\partial_{\bar z}\mathsf W=-Y_2^T(P_1-P_2)X_1.
\end{equation}
Thus Lemma~\ref{lem:cauchy-green-moment} applied to
\eqref{eq:leaf-moment} gives
\begin{equation}\label{eq:boundary-moments}
 \oint_{\partial B}g(z)\mathsf W(z)\,dz=0
\end{equation}
for every holomorphic polynomial \(g\).

The factorization constructs a relative field for each choice of disk and
frames. Independence of those choices will follow from uniqueness with
an exterior open normalization set. This uniqueness is local and needs
no frame on an entire annulus.

\begin{lemma}\label{lem:Gunique}
Let \(P_1,P_2\) be smooth square matrix coefficients on a connected
open leaf domain. Every smooth matrix \(W\) satisfying
\[
 2\partial_{\bar z}W+P_1W-WP_2=0
\]
and vanishing on a nonempty open subset vanishes everywhere.
Consequently two relative transports with the same equation and
the same values on a nonempty open subset coincide. In particular,
identity-normalized transports coincide on each connected leaf
overlap meeting their common normalization set.
\end{lemma}

\begin{proof}
On a small disk we choose invertible frames \(X_j\) satisfying
\(2\partial_{\bar z}X_j+P_jX_j=0\), by
Lemma~\ref{lem:matrix-dbar-frame}. We differentiate their inverses to obtain
\[
 2\partial_{\bar z}(X_1^{-1}WX_2)
 =X_1^{-1}(2\partial_{\bar z}W+P_1W-WP_2)X_2=0.
\]
The transformed matrix is holomorphic. If it vanishes on a
nonempty open part of that disk, the identity theorem makes it
zero on the whole disk. A chain of overlapping disks along a path
in the connected leaf domain propagates this conclusion to every
point. We apply the result to the difference of two transports.
Only local frames are used, including on a multiply connected annulus.
\end{proof}

The Cauchy factor gives identity trace on the disk boundary. We next
continue the reconstructed field through a common two-sided annulus
and then combine continuation, parameter regularity, and uniqueness.
The annulus argument uses the cancellation of the interior and
exterior traces in the distributional transport equation.

\begin{lemma}
\label{lem:distributional-annulus-gluing}
Let \(B\subset\C\) be a disk and let \(A\) be a connected two-sided
annulus of \(\partial B\). Let \(P_1,P_2\) be smooth matrix
coefficients near \(\overline B\) and on \(A\), with \(P_1=P_2\)
on \(A\). Suppose that
\(G_{\mathrm{in}}\in C^\infty(B)\cap C^0(\overline B)\) solves
\begin{equation}\label{eq:proof-annulus-interior-transport}
 2\partial_{\bar z}G_{\mathrm{in}}+P_1G_{\mathrm{in}}-G_{\mathrm{in}}P_2=0
\end{equation}
in \(B\), and that \(G_{\mathrm{in}}|_{\partial B}=I\). Define
\[
 G(z)=
 \begin{cases}
 G_{\mathrm{in}}(z),&z\in B,\\
 I,&z\in A\setminus\overline B\text{ or }z\in\partial B.
 \end{cases}
\]
Then \(G\) is a distributional solution of the relative equation
on \(B\cup A\), is smooth across the contour, and equals \(I\)
throughout \(A\), including its inner part.
\end{lemma}

\begin{proof}
Continuity on \(\overline B\) and \eqref{eq:proof-annulus-interior-transport}
make
\(\partial_{\bar z}G_{\mathrm{in}}\) bounded near the circle.
Apply integration by parts with a compactly supported smooth test
function on concentric interior disks and let their radii approach
the boundary radius. The traces converge uniformly, and the bounded
source permits passing to the limit in the volume integral. On
the exterior part of \(A\), the inner boundary is oriented oppositely.
The two boundary contributions therefore cancel because the inside
and outside traces both equal \(I\). Equivalently, the contour
measure in the distributional derivative is proportional to their
zero trace jump. Since \(P_1=P_2\) on \(A\), the outside field \(I\)
also satisfies the relative equation. This proves the distributional
equation on \(B\cup A\).

We choose a small disk contained in \(A\) and crossing the circle.
An invertible local frame \(X\) for the common coefficient \(P\)
satisfies \(2\partial_{\bar z}X+PX=0\). The product rule for
distributions gives
\begin{equation}\label{eq:proof-annulus-holomorphic-conjugation}
 2\partial_{\bar z}(X^{-1}GX)
 =X^{-1}(2\partial_{\bar z}G+PG-GP)X=0.
\end{equation}
The Weyl lemma for \(\bar\partial\), applied to
\eqref{eq:proof-annulus-holomorphic-conjugation}, makes \(X^{-1}GX\)
holomorphic.
On the exterior part of that disk it equals \(I\), so the identity
theorem gives \(X^{-1}GX=I\) everywhere on the disk. Hence \(G=I\)
on a neighborhood of every point of the circle. These neighborhoods
prove smoothness across it. Now \(G-I\) is a smooth solution of
the common relative equation on the connected annulus \(A\).
Lemma~\ref{lem:Gunique} propagates its vanishing throughout \(A\).
No global annular frame is required.
\end{proof}

The following proposition isolates the leafwise statement used in the
reconstruction. Its sufficient direction is the Cauchy--Plemelj passage
of \cite[Section~3.3, proof of Theorem~3.4,
equations (3.15)--(3.18)]{Cekic2025}, under the closed-disk and
common-annulus assumptions stated here. The proof below includes the converse and the uniqueness assertion.

\begin{proposition}
\label{prop:complete-moment-relative-transport}
Let \(B\subset\C\) be a disk, with Euclidean area measure \(dA\),
and let \(A\) be a connected two-sided annulus of \(\partial B\).
Let \(P_1,P_2\) be smooth
\(m\times m\) matrix coefficients on a neighborhood of
\(\overline B\cup A\), with \(P_1=P_2\) on \(A\).
Suppose that smooth frames \(X_1,Y_2\) are invertible throughout
that neighborhood and satisfy
\[
 2\partial_{\bar z}X_1+P_1X_1=0,\qquad
 2\partial_{\bar z}Y_2-P_2^TY_2=0.
\]
Then the following conditions are equivalent:
\begin{enumerate}[label=\textup{(\roman*)}]
\item For every holomorphic polynomial \(g\),
\[
 \int_B g(z)Y_2^T(P_1-P_2)X_1\,dA=0.
\]
\item There exists a smooth \(\mathrm{GL}(m,\C)\)-valued field
\(G\) on \(B\cup A\) satisfying
\[
 2\partial_{\bar z}G+P_1G-GP_2=0,\qquad G=I\quad\hbox{on }A.
\]
\end{enumerate}
When these conditions hold, \(G\) is unique. The moment identity
also holds for every \(g\in\mathcal O(B)\), without a boundary
regularity assumption on \(g\).
\end{proposition}

\begin{proof}
Put \(\delta=2\partial_{\bar z}\) and
\(\mathsf W=Y_2^TX_1\). The frame equations give
\[
 \delta\mathsf W=-Y_2^T(P_1-P_2)X_1.
\]
Assume \textup{(i)}. Cauchy--Green gives
\(\oint_{\partial B}z^n\mathsf W\,dz=0\) for every \(n\ge0\).
Consequently the exterior Cauchy transform is zero, by its expansion
at infinity and the identity theorem. The interior Cauchy factor
\[
 \mathsf R(z)=\frac{1}{2\pi\mathrm i}\oint_{\partial B}
       \frac{\mathsf W(\zeta)}{\zeta-z}\,d\zeta
\]
is holomorphic and has boundary trace \(\mathsf W\), by the
Cauchy--Plemelj formula. The Fourier-decay argument of
Lemma~\ref{lem:matrix-plemelj} gives a smooth extension to
\(\overline B\).
Since \(\mathsf W\) is invertible throughout \(\overline B\), its
boundary determinant loop has winding number zero. The same boundary
loop is \(\det \mathsf R\), so the argument principle excludes zeros of
\(\det \mathsf R\) in \(B\). Thus \(\mathsf R\) and \(\mathsf R^{-1}\) are smooth
on the closed disk.

Both hypotheses used above are substantive: all polynomial
moments are needed to make the exterior Cauchy transform vanish,
whereas the invertibility of \(\mathsf W\) on the \emph{whole closed disk}
forces zero winding of \(\det\mathsf W\) on its boundary. Boundary
nonvanishing alone would not exclude interior zeros of \(\det \mathsf R\).

Set \(G=X_1\mathsf R^{-1}Y_2^T\) in \(B\). The product rule gives the
relative equation, and its boundary value is
\(X_1(Y_2^TX_1)^{-1}Y_2^T=I\).
Extend it by \(I\) on the exterior part of \(A\).
The equal traces cancel the contour contribution to its
distributional derivative. Because \(P_1=P_2\) on \(A\), the
extended field solves the relative equation there as a distribution.
On a small disk crossing the contour, choose, by Lemma~\ref{lem:matrix-dbar-frame}, an invertible local
frame \(X\) with \(\delta X=-P_1X\). Then
\(\delta(X^{-1}GX)=0\), and the Weyl lemma makes this matrix
holomorphic. It equals \(I\) on the exterior open part, so the
identity theorem gives \(G=I\) across the contour. Local
holomorphic uniqueness propagates that equality throughout the
connected annulus \(A\). The resulting field is smooth and
invertible on \(B\cup A\), proving \textup{(ii)}.

Conversely, assume \textup{(ii)} and define
\(\widetilde{\mathsf R}=Y_2^TG^{-1}X_1\).
Since
\(\delta G^{-1}=G^{-1}P_1-P_2G^{-1}\), the product rule in this
matrix order gives \(\delta\widetilde{\mathsf R}=0\).
The difference \(Q_*=\mathsf W-\widetilde{\mathsf R}\) vanishes on the
annulus and is therefore compactly supported inside \(B\).
For every \(g\in\mathcal O(B)\),
\[
 \int_BgY_2^T(P_1-P_2)X_1\,dA
 =-\int_B\delta(gQ_*)\,dA=0.
\]
Compact support justifies this integration without a boundary trace
of \(g\). This proves \textup{(i)} and the assertion for all
holomorphic weights.

For uniqueness, the difference of two relative fields solves the
homogeneous relative equation and vanishes on \(A\).
Local invertible frames for the two coefficients transform that
difference into a holomorphic matrix. The identity theorem on a
chain of overlapping disks then gives equality throughout the
connected set \(B\cup A\).
\end{proof}

The criterion holds for square matrices of any size. In elasticity,
Theorem~\ref{thm:physical-nine-moments} provides its hypothesis at
rank three from physical boundary data. Proposition~
\ref{prop:complete-moment-relative-transport} concerns holomorphic
weights on one leaf. Smooth dependence on parameters, bundle covariance,
and holomorphic dependence in the projective direction require the
additional hypotheses of Theorem~\ref{thm:global-plemelj-gluing}.
Thus the criterion can be reused once closed-disk frames and complete
weighted moments have been proved for a specified system.

The continuation lemma turns identity trace into identity on an open
annulus, where relative uniqueness applies. The following theorem
uses this normalization first to compare disks and frames, then to
compare bundle charts. Its final assertion has additional hypotheses:
a fixed physical product, holomorphic connection matrices, and one
normalization set meeting every connected leaf section.

\begin{theorem}
\label{thm:global-plemelj-gluing}
\label{lem:parameter-plemelj}
\label{rem:plemelj-global-convention}
Let \(p\) range in a smooth finite-dimensional real parameter chart.
In complex leaf coordinates suppose that the disks \(B_p\) and
two-sided boundary annuli are images of fixed reference sets under
\(\Phi_p(w)=a_pw+b_p\), where \(a_p\ne0\) and \(a_p,b_p\) are smooth.
Let \(P_{j,p}\) be jointly smooth connection matrices on neighborhoods
of these sets, with \(P_{1,p}=P_{2,p}\) throughout the annuli.
Assume that jointly smooth invertible frames extend across the closed disks
and annuli, satisfy
\begin{equation}\label{eq:proof-global-plemelj-frame-equations}
 2\partial_{\bar z}X_{1,p}+P_{1,p}X_{1,p}=0,\qquad
 2\partial_{\bar z}Y_{2,p}-P_{2,p}^TY_{2,p}=0,
\end{equation}
and have uniformly bounded derivatives and inverses on compact subcharts.
Set \(\mathsf W_p=Y_{2,p}^TX_{1,p}\), and suppose that
\begin{equation}\label{eq:proof-global-plemelj-monomial-moments}
 \oint_{\partial B_p} z^n\mathsf W_p(z)\,dz=0
 \qquad(n\ge0).
\end{equation}
Then the Cauchy--Plemelj construction gives a unique invertible relative
transport \(G_p\), jointly smooth on the disks and annuli and equal to \(I\)
on the annuli, satisfying
\[
 2\partial_{\bar z}G_p+P_{1,p}G_p-G_pP_{2,p}=0.
\]
Different frames and enclosing disks give the same transport on each connected
overlap where their equations agree and both transports equal \(I\)
on a common nonempty open subset.

For bundle gluing, suppose additionally that the invariant relative
operators are covariant under smooth invertible bundle transition matrices
independent of \(x\), together with nonzero scalar changes of null
representative. Assume that each connected leaf component of a chart
overlap meets a common nonempty open set where both transports equal
\(I\). Then the transports obey the corresponding conjugation law
and define a global smooth bundle automorphism.

Suppose that this automorphism is defined on a common physical product
\(\mathcal U\times U\) in each spinor chart, with invariant equation
\begin{equation}\label{eq:proof-parameter-invariant-transport}
 2D_{\theta(\zeta)}G+\mathcal A_1(x,\zeta)G-G\mathcal A_2(x,\zeta)=0.
\end{equation}
Assume that the bundle transitions, the nonzero complex null representative
\(\theta\), and the matrices \(\mathcal A_j\) are holomorphic in \(\zeta\)
at fixed physical \(x\). Assume also that \(G=I\) on one fixed physical
open set \(E\subset\mathcal U\), and that each connected component of
\(\Pi\cap\mathcal U\), for every affine complex leaf \(\Pi\) and every
parameter, meets \(E\cap\Pi\) in a nonempty relatively open set.
Then \(G\) is holomorphic in the spinor parameter at fixed physical \(x\).
\end{theorem}

\begin{proof}
Take the reference disk to be the unit disk and set
\(\mathsf W_p^\sharp=\mathsf W_p\circ\Phi_p\). The affine contour change gives
\begin{equation}\label{eq:proof-global-reference-moments}
 \oint_{|w|=1}w^n\mathsf W_p^\sharp(w)\,dw
 =a_p^{-1}\oint_{\partial B_p}
       \left(\frac{z-b_p}{a_p}\right)^n
       \mathsf W_p(z)\,dz=0,\qquad n\geq0.
\end{equation}
Since \(((z-b_p)/a_p)^n\) is a polynomial in \(z\), we use
\eqref{eq:proof-global-plemelj-monomial-moments} and linearity to prove
\eqref{eq:proof-global-reference-moments}. Applying
Lemma~\ref{lem:matrix-plemelj} to the jointly smooth invertible family
\(\mathsf W_p^\sharp\), we obtain
\begin{equation}\label{eq:Rplemelj}
 \mathsf R_p(w)=\frac{1}{2\pi\mathrm i}\oint_{|\zeta|=1}
             \frac{\mathsf W_p^\sharp(\zeta)}{\zeta-w}\,\dd\zeta,
 \qquad |w|<1,
\end{equation}
with boundary trace \(\mathsf W_p^\sharp\). Its Fourier estimate
\eqref{eq:proof-plemelj-fourier-decay} and determinant argument give
joint smoothness of \(\mathsf R_p,\mathsf R_p^{-1}\) on the closed reference disks,
with all fixed finite derivative bounds on compact subcharts.

In physical leaf coordinates, we set
\begin{equation}\label{eq:proof-plemelj-transport-product}
 F_p(z)=\mathsf R_p(\Phi_p^{-1}(z))^{-1},\qquad
 G_p(z)=X_{1,p}(z)F_p(z)Y_{2,p}(z)^T.
\end{equation}
The affine inverse is holomorphic in \(z\), so \(F_p\) is leafwise
holomorphic. Using \eqref{eq:proof-global-plemelj-frame-equations} and
\eqref{eq:proof-plemelj-transport-product}, we derive the relative
equation in \eqref{eq:Gtransport}. On the
contour \(F_p=\mathsf W_p^{-1}\), and
\(X_{1,p}(Y_{2,p}^TX_{1,p})^{-1}Y_{2,p}^T=I\). Thus
\begin{equation}\label{eq:Gtransport}
 2\partial_{\bar z}G_p+P_{1,p}G_p-G_pP_{2,p}=0,\qquad
 G_p=I\quad\hbox{on }\partial B_p.
\end{equation}
Lemma~\ref{lem:distributional-annulus-gluing} extends \(G_p\) by
\(I\) and gives \(G_p=I\) on the entire two-sided annulus. After
pullback this annulus is fixed locally in the parameters. The
jointly smooth closed-disk formula is therefore \(I\) on a fixed
inner collar and glues smoothly to \(I\) across the moving contour.
Invertibility inside follows from the three invertible factors in
\eqref{eq:proof-plemelj-transport-product}, and outside from \(G_p=I\).

Lemma~\ref{lem:Gunique} identifies two choices of frames or enclosing
disks on each connected overlap with the same equation and a common
normalization open set. On a bundle-chart overlap we conjugate one
transport by the prescribed transition matrix. Covariance makes it
solve the other chart's equation, after a nonzero scalar change of
null representative. Both transports equal \(I\) on an open part of
each connected leaf component, so the same uniqueness lemma proves
the conjugation law. The fundamental frames themselves need not obey
that transition.

For the final assertion, we set \(V=\partial_{\bar\zeta}G\).
Joint smoothness permits commuting this derivative with spatial
differentiation. Differentiating
\eqref{eq:proof-parameter-invariant-transport} and using holomorphicity
of \(\theta\) and \(\mathcal A_j\) gives
\begin{equation}\label{eq:proof-global-parameter-derivative}
 2D_{\theta(\zeta)}V+\mathcal A_1V-V\mathcal A_2=0.
\end{equation}

The derivative \(V=\partial_{\bar\zeta}G\) is taken in a
holomorphic quotient frame at fixed physical \(x\). The assumed
joint smoothness on \(\mathcal U\times U\) accounts for the
parameter dependence of the moving leaf coordinates. Since
\(G(x,\zeta)=I\) for every \(x\in E\) and every parameter,
\(V=0\) on that same fixed set. Only the reconstructed transport
is differentiated here.

At a fixed parameter, we choose \(b\ne0\) such that
\(b\theta=e_1+\mathrm i e_2\) is normalized. In its leaf coordinate
\(D_\theta=2b^{-1}\partial_{\bar z}\). Multiplication of
\eqref{eq:proof-global-parameter-derivative} by \(b/2\) puts the
equation for \(V\) in the form of Lemma~\ref{lem:Gunique}, with
coefficients \(b\mathcal A_j/2\). Each connected leaf component meets
\(E\) in a nonempty relatively open set, so that lemma gives \(V=0\)
on every component. Thus \(G\) is holomorphic at fixed physical \(x\).
\end{proof}

Theorem~\ref{thm:global-plemelj-gluing} separates smooth dependence
of the leafwise Cauchy construction from holomorphic dependence in
the direction. Holomorphic dependence on the projective null direction follows by differentiating the invariant relative equation at fixed $x$ and using the common exterior normalization. We now check the required transition law for the
paired physical quotient bundles.

\begin{lemma}
\label{lem:spinor-overlap-plemelj}
Express the primal Plemelj transports \(G_0,G_\infty\) in the quotient
frames associated with the holomorphic spinors \((1,\zeta)\) and
\((\zeta^{-1},1)\), respectively. Let
\[
 D(\zeta)=\begin{pmatrix}\zeta I_2&0\\0&1\end{pmatrix}
\]
be the dual quotient transition matrix. On each connected leaf overlap
meeting their common exterior normalization open set, one has
\[
 G_0=D(\zeta)^TG_\infty D(\zeta)^{-T}.
\]
Thus these local transports glue as primal bundle automorphisms on
the overlaps used in the enclosing-ball construction.
\end{lemma}

\begin{proof}
The spinor identities \eqref{eq:theta-spinor}--\eqref{eq:qspinor} and
the connection covariance \eqref{eq:connection-chart-scaling} give
\(\theta_0=\zeta^2\theta_\infty\) and
\(\mathcal A_{j,0}=\zeta^2D^T\mathcal A_{j,\infty}D^{-T}\).
Since \(D\) is independent of \(x\), the infinity-chart
relative equation is conjugated by \(D^T\) and then multiplied by
\(\zeta^2\) to give the equation for \(D^TG_\infty D^{-T}\) in the
zero chart. It is the same equation as for \(G_0\). Both transports
equal \(I\) on their common normalization open set.
Lemma~\ref{lem:Gunique} therefore proves the identity on the connected
leaf overlap. The primal coordinate change is \(D^T\); the matrix
\(D\) itself is the dual frame transition.
\end{proof}

The CGO construction uses normalized directions, while the projective
charts use holomorphic spinor representatives. Multiplication by a
nonzero scalar preserves the intrinsic leaf and relative equation;
rescaling the spinor also changes the quotient frame. We keep rescaling of the null representative and the associated quotient-frame change explicit in the following comparison.

\begin{lemma}\label{lem:representative-invariance}
Let \(b\in\C^*\) be independent of \(x\). The characteristic lines,
paired physical quotient bundles, and affine complex leaves of
\(\theta\) and \(b\theta\) are canonically identical. Under this
identification, a relative quotient transport solves the equation for
\(\theta\) if and only if it solves that for \(b\theta\). Its exterior
normalization and leafwise uniqueness are therefore independent of
the nonzero representative of the projective null direction.
\end{lemma}

\begin{proof}
Multiplication by \(b\) gives an invertible real rotation and dilation
on the pair \((\operatorname{Re}\theta,\operatorname{Im}\theta)\).
Their real span, and hence the affine leaf planes, are unchanged.
The symbol and derivative satisfy
\[
 A_{b\theta}=bA_\theta,\qquad D_{b\theta}=bD_\theta,
 \qquad \C e_{b\theta}=\C e_\theta.
\]
Use the same quotient basis under this canonical identification.
The induced relative operator is multiplied by \(b\), so its
solution set and normalization are unchanged. If coordinates are
instead supplied by rescaled spinors, the corresponding quotient
frame change must also be applied, as in
\eqref{eq:normalization-frame-change} below.
\end{proof}

Representative invariance permits the same relative transport to be
expressed in the two direction conventions. Its independence of the
auxiliary invertible frames and enclosing radii follows from the
holomorphic frame changes and exterior uniqueness already proved.

\begin{corollary}\label{cor:frame-independence}
The relative transport is independent of the invertible primal and
dual \(\bar\partial\)-frames on a neighborhood of the closed leaf
disk used to form \(\mathsf W=Y_2^TX_1\). It is also independent of the
radius of a concentric enclosing disk whose boundary annulus lies
in the common exterior and which contains the entire support of
\(P_1-P_2\) on that leaf.
\end{corollary}

\begin{proof}
For alternative frames, we define
\begin{equation}\label{eq:proof-holomorphic-frame-changes}
 C_1=X_1^{-1}X_1',\qquad C_2=Y_2^{-1}Y_2'.
\end{equation}
Using \eqref{eq:leaf-primal} and \eqref{eq:leaf-dual}, we prove that
the frame changes in \eqref{eq:proof-holomorphic-frame-changes} are
holomorphic. Their invertibility follows from that of the quotient frames.
On the contour the new product satisfies
\begin{equation}\label{eq:proof-alternative-plemelj-trace}
 \mathsf W'=C_2^T\mathsf W C_1=C_2^T\mathsf RC_1.
\end{equation}
Thus \(\mathsf R'=C_2^T\mathsf RC_1\) is its holomorphic Cauchy factor: the new
contour moments vanish by Cauchy's theorem, and the Cauchy formula
recovers this product. Consequently
\[
 X_1'(\mathsf R')^{-1}(Y_2')^T=X_1\mathsf R^{-1}Y_2^T=G.
\]
For two enclosing radii, we extend the smaller-disk transport by \(I\)
through the layer where \(P_1=P_2\). Lemma~\ref{lem:distributional-annulus-gluing}
makes this a smooth relative solution on the larger disk. It agrees
with the larger-disk transport on an exterior annulus, so the
uniqueness assertion of Theorem~\ref{thm:global-plemelj-gluing} identifies them.
\end{proof}
\subsection{Projective covariance and holomorphic transport}

The local leafwise transports are assembled into a transport on the fixed product
\(B_R\times\mathbb{CP}^1\). This requires the spinor frame change,
a common exterior normalization on chart overlaps, and smooth
continuation across leaves that miss the difference support. After
these properties are established, the final assertion of
Theorem~\ref{thm:global-plemelj-gluing} gives holomorphic dependence
in the projective direction at fixed physical \(x\).

By Lemma~\ref{lem:representative-invariance}, replacing a normalized
null representative by a holomorphic spinor representative preserves the
intrinsic relative transport. Its coordinate matrices also undergo the
corresponding frame change. The projective null cone is covered by the
two spinor charts
\[
 U_0=\{s\neq0\},\qquad U_\infty=\{t\neq0\}.
\]
On a local chart, we denote by \(\theta_{\mathrm h}\) its nonzero holomorphic
spinor representative. Its real and imaginary parts are orthogonal,
have the same positive length, and vary smoothly in the real chart
coordinates. Thus a smooth normalized representative may be chosen as
\(\theta_{\mathrm n}=a\theta_{\mathrm h}\), with \(a\ne0\) independent
of \(x\); for example,
\[
 a=\left(\frac{\theta_{\mathrm h}\cdot
              \overline{\theta_{\mathrm h}}}{2}\right)^{-1/2}>0.
\]
Choose locally the smooth positive spinor factor
\(\varsigma=\sqrt a\). The normalized spinor is
\(\varsigma\) times the holomorphic spinor, so
\[
 \mathsf U_{\mathrm n}=\mathsf U_{\mathrm h}S_\varsigma,
 \qquad S_\varsigma=\operatorname{diag}(\varsigma I_2,1).
\]
Consequently the frames, connection, and primal relative transport obey
\begin{equation}\label{eq:normalization-frame-change}
 \begin{gathered}
 Y_{\mathrm h}=S_\varsigma Y_{\mathrm n},\qquad
 X_{\mathrm h}=S_\varsigma^{-T}X_{\mathrm n},\\
 \mathcal A_{j,\mathrm n}
    =aS_\varsigma^T\mathcal A_{j,\mathrm h}S_\varsigma^{-T},\qquad
 G_{\mathrm h}=S_\varsigma^{-T}G_{\mathrm n}S_\varsigma^T.
 \end{gathered}
\end{equation}
At fixed physical \(x\), \eqref{eq:normalization-frame-change} gives
the matrix-density law
\begin{equation}\label{eq:proof-normalized-moment-density}
 Y_{\mathrm n}^T(\mathcal A_{1,\mathrm n}
                   -\mathcal A_{2,\mathrm n})X_{\mathrm n}
 =aY_{\mathrm h}^T(\mathcal A_{1,\mathrm h}
                   -\mathcal A_{2,\mathrm h})X_{\mathrm h}.
\end{equation}
The factor \(a\) is independent of \(x\). Holding the physical
measure and the leafwise test functions fixed,
\eqref{eq:proof-normalized-moment-density} therefore preserves vanishing
of the moment \eqref{eq:leaf-moment}. The normalization matrices
and their inverses are bounded on compact parameter subcharts.
In particular, \(\det Y_{\mathrm h}=a\det Y_{\mathrm n}\), so the
quantitative nonvanishing of Theorem~\ref{thm:graded-frame} persists.
Choices of normalized representative and local square root give the
same intrinsic transport by Lemma~\ref{lem:representative-invariance},
\eqref{eq:normalization-frame-change}, and Lemma~\ref{lem:Gunique}
with the common exterior normalization.
The normalization factor is smooth in the real chart variables;
the parameter equation will be differentiated in the holomorphic
quotient frame.

Recall from Subsection~\ref{subsec:quotient and transport} that the
dual quotient transition matrix on \(U_0\cap U_\infty\) is
\begin{equation}\label{eq:spinor-transition}
 D(\zeta)=
 \begin{pmatrix}
 \zeta\Id_2&0\\
 0&1
 \end{pmatrix},
 \qquad \zeta=t/s.
\end{equation}
Here \(Q_0=Q(1,\zeta)\) and \(Q_\infty=Q(\zeta^{-1},1)\)
are the \(3\times2\) spinor matrices of \eqref{eq:qspinor}.
Use the corresponding representatives \(\theta_0,\theta_\infty\).
Directly from
\eqref{eq:theta-spinor}--\eqref{eq:qspinor},
\[
 Q_0=\zeta Q_\infty,\qquad \theta_0=\zeta^2\theta_\infty.
\]
Thus dual coordinates \(v=(\xi,\eta)^T\) satisfy \(v_0=D^{-1}v_\infty\).
Coordinates \(u\) in the paired primal quotient satisfy \(u_0=D^Tu_\infty\),
so that \(v_0^Tu_0=v_\infty^Tu_\infty\). By
Lemma~\ref{lem:holomorphic-lifts}, we choose the local right inverses in
\eqref{eq:primal-quotient-connection} compatibly by
\begin{equation}\label{eq:proof-compatible-quotient-right-inverses}
 \mathsf V_0=\mathsf V_\infty D^{-T}\pmod{\C e_{\theta_0}}.
\end{equation}
Equation~\eqref{eq:proof-compatible-quotient-right-inverses} gives a right
inverse for \(\mathsf U_0^T\), where \(\mathsf U_0=\mathsf U_\infty D\),
and any other such choice gives the same
quotient connection. The connection matrices therefore obey
\begin{equation}\label{eq:connection-chart-scaling}
 \Gamma_{j,0}=\zeta^2D^{-1}\Gamma_{j,\infty}D,
 \qquad
 \mathcal A_{j,0}=\zeta^2D^T\mathcal A_{j,\infty}D^{-T},
 \qquad \mathcal A_j=-\Gamma_j^T.
\end{equation}
The differential operators obey \(D_{\theta_0}=\zeta^2D_{\theta_\infty}\).

Since the transition matrix \(D\) is independent of \(x\), its
conjugation introduces no spatial derivative term. Together with
\(D_{\theta_0}=\zeta^2D_{\theta_\infty}\), the connection law
\eqref{eq:connection-chart-scaling} therefore identifies the two
relative equations on the overlap.

The common leaf normalization on chart overlaps can be checked in
physical coordinates. For any nonzero complex null representative, we set
\begin{equation}\label{eq:proof-projective-leaf-normal}
 \nu_\theta=\frac{\operatorname{Re}\theta\times
                         \operatorname{Im}\theta}
                        {|\operatorname{Re}\theta|^2}.
\end{equation}
Equation~\eqref{eq:proof-projective-leaf-normal} gives the unit normal
of the oriented real leaf plane. Multiplication
of \(\theta\) by \(b\in\C^*\) multiplies both the numerator and
denominator by \(|b|^2\), so it preserves \(\nu_\theta\), the real transverse
coordinate \(x_\perp=\nu_\theta\cdot x\), and the physical affine leaf. In particular,
the representatives \(\theta_0=\zeta^2\theta_\infty\) use the same
enclosing disk \(B_{R_1}\cap\Pi_{\theta,x_\perp}\). Their common exterior
annulus may be chosen from the same fixed radial shell supplied by
Lemma~\ref{lem:uniform-leaf-disks}. Thus each connected leaf overlap
used here meets an open set where both transports equal \(I\).

Let \(G_0,G_\infty\) denote the primal transports in the holomorphic
quotient frames on the common leaf disks. The common annulus just
verified supplies the normalization hypothesis of
Lemma~\ref{lem:spinor-overlap-plemelj}, which gives
\begin{equation}\label{eq:G-chart-covariance}
 G_0=D^TG_\infty D^{-T}=: \widehat G_\infty.
\end{equation}
In particular, inverse transpose gives the \emph{dual} covariance
\begin{equation}\label{eq:H-chart-covariance}
 H_0=D^{-1}H_\infty D,\qquad H=G^{-T}.
\end{equation}
Equations~\eqref{eq:G-chart-covariance}--\eqref{eq:H-chart-covariance}
are the paired covariance laws on the common leaf disks.
The continuation below extends them to a common physical product.

We first verify the smooth-family hypotheses of
Theorem~\ref{thm:global-plemelj-gluing} for the normalized leaf disks.
The spinor frame change will then express the resulting transports
in holomorphic quotient coordinates.

Lemma~\ref{lem:uniform-leaf-disks} supplies the common exterior annuli,
where the two connections agree. The disk parametrization is
\(w\mapsto (R_1^2-x_\perp^2)^{1/2}w\) in the normalized leaf coordinate.
On a compact local parameter subchart, we choose a fixed reference annulus
\(1-\varepsilon<|w|<1+\varepsilon\) whose image lies in that common
exterior for every parameter, as in Lemma~\ref{lem:uniform-leaf-disks}.
Restricting the annulus to this image gives precisely the fixed-set
hypothesis of the gluing theorem.
Lemma~\ref{lem:uniform-leafwise-frames} supplies the smooth invertible
frames on neighborhoods of the closed disks; their equations are
\eqref{eq:leaf-primal} and \eqref{eq:leaf-dual}.
Equation~\eqref{eq:boundary-moments} supplies every required contour moment.
Using Theorem~\ref{thm:global-plemelj-gluing}, the Fourier estimate
\eqref{eq:proof-plemelj-fourier-decay}, and the determinant argument of
Lemma~\ref{lem:matrix-plemelj}, we obtain joint smoothness on the closed
reference disks and smooth inverses on compact parameter subcharts.
In leaf coordinates, the reconstructed transport depends smoothly on
\((z,x_\perp,\operatorname{Re}\zeta,\operatorname{Im}\zeta)\).
The coordinate changes and the frame changes
\eqref{eq:normalization-frame-change} and \eqref{eq:spinor-transition}
are smooth in these real variables. The covariance data
\eqref{eq:connection-chart-scaling} and the common leaf annulus verify
both overlap hypotheses of the bundle assertion. Uniqueness identifies
constructions from different local parameter neighborhoods, by
Lemma~\ref{lem:Gunique}. For \(|x_\perp|<R_\perp\), we extend the Plemelj transport by \(I\) from
\(B_{\theta,x_\perp}\) to the entire affine leaf \(\Pi_{\theta,x_\perp}\).
The extension remains a smooth solution of \eqref{eq:Gtransport}:
Lemma~\ref{lem:distributional-annulus-gluing} gives \(I\) on a two-sided
annulus of its boundary, and the
two connections coincide outside \(\overline{B_{R_0}}\).
The fixed reference annuli give joint smoothness across the moving
contour locally in the parameters. The extensions agree on overlapping
parameter neighborhoods by leafwise uniqueness.

For fixed real \(x\), the functions
\(x_\perp=\nu_{\theta(\zeta)}\cdot x\) and
\(z=e_1(\zeta)\cdot x+\mathrm i e_2(\zeta)\cdot x\)
are smooth functions of the real variables
\((x,\operatorname{Re}\zeta,\operatorname{Im}\zeta)\).
Composing with these functions gives smooth physical-coordinate
matrices on the open subset
\begin{equation}\label{eq:proof-physical-product-inside}
 \mathcal U_< =\{(x,\zeta)\in B_R\times U:
                        |\nu_{\theta(\zeta)}\cdot x|<R_\perp\}.
\end{equation}
On the second open subset
\begin{equation}\label{eq:proof-physical-product-outside}
 \mathcal U_> =\{(x,\zeta)\in B_R\times U:
                        |\nu_{\theta(\zeta)}\cdot x|>R_0\}
\end{equation}
we define the transport to be \(I\). Each such leaf misses the entire
coefficient-difference support, so \(G=I\) solves its relative
equation. On \(\mathcal U_<\cap\mathcal U_>\), the extended Plemelj
transport also solves that equation and equals \(I\) on an exterior
annulus. Lemma~\ref{lem:Gunique} on the connected full leaf makes it
identically \(I\). Since \(R_0<R_\perp\), the open subsets
\eqref{eq:proof-physical-product-inside} and
\eqref{eq:proof-physical-product-outside} cover \(B_R\times U\).
Their equal smooth matrices on the overlap therefore define a jointly
smooth matrix on this fixed product, including at \(|x_\perp|=R_\perp\).
Since the local inverses agree on the overlap of
\eqref{eq:proof-physical-product-inside} and
\eqref{eq:proof-physical-product-outside}, we glue them on the same
physical product. The extended covariance
laws \eqref{eq:G-chart-covariance}--\eqref{eq:H-chart-covariance}
hold throughout the physical products on chart overlaps.

The exterior portion of a full leaf is connected: it is the
complement of a closed disk for \(|x_\perp|<R_0\), the complement of one
point for \(|x_\perp|=R_0\), and the whole leaf for \(|x_\perp|>R_0\).
The two connection matrices agree there. Thus \(G-I\) solves the
common homogeneous relative equation and vanishes on the enclosing
annulus, or on the whole leaf in the last case. Lemma~\ref{lem:Gunique}
gives \(G=I\) throughout that exterior portion. Thus its identity normalization
holds on the fixed physical shell
\begin{equation}\label{eq:proof-fixed-exterior-shell}
 E=\{x\in B_R:R_0<|x|<R\},
\end{equation}
independently of the direction. Every nonempty section
\(B_R\cap\Pi_{\theta,x_\perp}\) is a connected disk and meets \(E\) in a
nonempty relatively open set. Indeed, for \(|x_\perp|<R_0\) its exterior part
contains the annulus
\[
 \sqrt{R_0^2-x_\perp^2}<|z|<\sqrt{R^2-x_\perp^2}.
\]
For \(|x_\perp|=R_0\), only the center of the section is excluded from
\(E\); for \(R_0<|x_\perp|<R\), the entire section belongs to \(E\).
In every case its intersection with \(E\) is nonempty and relatively
open. The cover \eqref{eq:proof-physical-product-inside}--
\eqref{eq:proof-physical-product-outside} and the shell
\eqref{eq:proof-fixed-exterior-shell} verify the fixed-product,
connected-leaf, and exterior normalization hypotheses of the parameter
assertion of
Theorem~\ref{thm:global-plemelj-gluing}. In particular, every chart
overlap on the fixed ball has this same normalization set on each
nonempty leaf section, as required for the global bundle gluing.

It remains to check the holomorphic-parameter hypotheses of
Theorem~\ref{thm:global-plemelj-gluing}. We transform to the holomorphic
quotient frame \(\mathsf U_{\mathrm h}\) using
\eqref{eq:normalization-frame-change}; by
Lemma~\ref{lem:representative-invariance}, the zero chart may use
\eqref{eq:theta-spinor} with spinor \((1,\zeta)\):
\[
 \theta(\zeta)=(1-\zeta^2,\,\mathrm i(1+\zeta^2),\,2\zeta)^T.
\]
This representative is nonzero throughout the chart. In this frame,
\(\mathcal A_j=-\Gamma_j^T\), and \eqref{eq:Gamma} expresses
\(\Gamma_j\) in terms of the polynomial spinor functions
\(\theta,Q,n\), and \(\mathfrak r_{m_j}=-2m_j^TQ\).
The coefficient fields depend only on \(x\), so both the direction
and the connection matrices are holomorphic in \(\zeta\) at fixed
physical \(x\). Lemma~\ref{lem:holomorphic-lifts} gives the same
induced connection independently of the local right inverse.

The holomorphic connection matrices just checked, together with the
fixed product and normalization verified in
\eqref{eq:proof-physical-product-inside}--\eqref{eq:proof-fixed-exterior-shell},
satisfy the final hypotheses of
Theorem~\ref{thm:global-plemelj-gluing}. It follows that \(G_0\) is
holomorphic in \(\zeta\) at fixed physical \(x\).

The infinity chart uses \(\upsilon=s/t\) and the nonzero holomorphic
representative \(\theta(\upsilon,1)\); the same polynomial argument
verifies its connection hypothesis and gives holomorphicity of
\(G_\infty\). The transition matrix \(D(\zeta)\) and its inverse
are holomorphic on \(\zeta\ne0\), so
\eqref{eq:G-chart-covariance} glues these matrices as one global
holomorphic primal bundle automorphism for each fixed \(x\).

The local Cauchy fields have a common physical domain, the fixed-shell normalization \eqref{eq:proof-fixed-exterior-shell}, and the bundle transition law \eqref{eq:G-chart-covariance}. Combined with holomorphic direction dependence, these properties yield the relative transport in Theorem~\ref{thm:quotient-reconstruction} from equality of the physical DN maps.

The regularity assertions here are logically separate: the
Cauchy factor and its inverse are first controlled on closed leaf
disks with smooth \emph{real} parameters; the local fields are then
identified on one fixed physical product using the common exterior
shell \eqref{eq:proof-fixed-exterior-shell}; only afterwards is the
parameter transport equation differentiated at fixed \(x\) to prove
holomorphicity. No parameter derivatives of the exact CGO solutions
are required.

\begin{theorem}\label{thm:quotient-reconstruction}
For the common exterior extensions and fixed physical ball \(B_R\)
chosen in this section, assume \(\DN_1=\DN_2\). Then there is a unique
relative automorphism \(G\) of the rank-three primal quotient bundle
over \(B_R\times\mathbb{CP}^1\) which is jointly smooth in local
spatial and real chart coordinates, holomorphic in the projective
null direction at fixed \(x\), and equal to \(I\) on the fixed shell
\(E=\{R_0<|x|<R\}\). It intertwines the two quotient transports:
\begin{equation}\label{eq:proof-reconstructed-primal-transport}
 2D_\theta G+\mathcal A_{1,\theta}G-G\mathcal A_{2,\theta}=0.
\end{equation}
Its local matrices obey \eqref{eq:G-chart-covariance}.
\end{theorem}

\begin{proof}
The construction from Theorem~\ref{thm:global-plemelj-gluing}, extended
using \eqref{eq:proof-physical-product-inside}--
\eqref{eq:proof-physical-product-outside}, verifies existence, invertibility,
and joint smoothness. Equation~\eqref{eq:G-chart-covariance} gives the
transition law, and \eqref{eq:proof-fixed-exterior-shell} specifies the fixed
exterior normalization. 
In each holomorphic spinor chart, \eqref{eq:theta-spinor} and
\eqref{eq:Gamma} give holomorphic direction and connection matrices.
The fixed product and the shell meeting every connected leaf section
were verified in \eqref{eq:proof-physical-product-inside}--
\eqref{eq:proof-fixed-exterior-shell}. The final assertion of
Theorem~\ref{thm:global-plemelj-gluing} therefore gives holomorphicity
of the relative transport at fixed \(x\).

For uniqueness, we subtract two such local matrices in the same quotient
frame. Their difference solves the homogeneous relative equation and
vanishes on \(E\). Each connected disk \(B_R\cap\Pi_{\theta,x_\perp}\)
meets that set in a relatively open subset, so Lemma~\ref{lem:Gunique}
gives equality on every section and hence on the whole ball.
\end{proof}

The reconstruction theorem produces a primal bundle automorphism;
the rigidity argument is written for the dual connection
\(\Gamma_j\). Taking the inverse transpose respects their paired
transition laws and preserves the order of the two media in the
relative equation.

\begin{corollary}\label{cor:dual-relative}
Let \(G\) be the primal relative transport of Theorem~\ref{thm:quotient-reconstruction} and set
\[
 H=G^{-T}.
\]
Then \(H\) is smooth in \(x\), is a global holomorphic automorphism in the
projective null direction of the dual physical quotient bundle, equals the
identity on the fixed exterior shell \(E\), and satisfies
\begin{equation}\label{eq:dual-relative-transport}
 2D_\theta H+\Gamma_1H-H\Gamma_2=0,
\end{equation}
where \(\Gamma_j\) is the dual Lam\'e quotient connection \eqref{eq:Gamma}.
\end{corollary}

\begin{proof}
In paired quotient coordinates, \(\Gamma_j=-\mathcal A_j^T\).
Differentiate the inverse transpose and use
\eqref{eq:proof-reconstructed-primal-transport} to obtain
\[
 \begin{aligned}
 2D_\theta H
 &=-H(2D_\theta G)^TH\\
 &=\mathcal A_1^TH-H\mathcal A_2^T
  =-\Gamma_1H+H\Gamma_2.
 \end{aligned}
\]
Rearranging gives \eqref{eq:dual-relative-transport}, with the order of the
two media unchanged. Inversion of an invertible smooth holomorphic
matrix preserves its spatial smoothness and parameter holomorphicity.
Also \(G=I\) on the fixed shell implies \(H=I\) there.
Taking inverse transposes in \eqref{eq:G-chart-covariance} gives
\eqref{eq:H-chart-covariance}. Thus the local matrices of \(H\)
define the dual bundle automorphism with all the asserted properties.
\end{proof}

\section{Rigidity of the quotient transport}\label{sec:rigidity}

The recovery of the Lam\'e coefficients uses the dual intertwiner
\(H=G^{-T}\) constructed in Section~\ref{sec:transport}. The bundle
transition restricts its holomorphic block form, and its transport
equation then yields a trace-free Hessian equation for
\(\log(\mu_1/\mu_2)\).

For the boundary inverse problem, we set \(\mathcal D=B_R\) and
\(E=\{x\in B_R:R_0<|x|<R\}\), using the common extensions and
radii of Section~\ref{sec:transport}. Lemma~\ref{lem:common-extension}
and \eqref{eq:kappa-range} give smooth real coefficient pairs with
\(\mu_j>0\) and \(1/4<\kappa_j<1\), coinciding on \(E\).
After DN equality is transferred by Lemma~\ref{lem:enlargement},
Theorem~\ref{thm:quotient-reconstruction} and
Corollary~\ref{cor:dual-relative} supply \(H=G^{-T}\), jointly smooth
in spatial and real chart coordinates and holomorphic in the projective
direction. Its matrices obey \eqref{eq:H-chart-covariance} and
\eqref{eq:dual-relative-transport}, and \(H=I\) on \(E\). Thus all hypotheses of
Theorem~\ref{thm:rigidity} hold on the connected domain
\(\mathcal D\), with the same fixed nonempty open set \(E\).

\subsection{The quotient bundle and holomorphic automorphisms}

The spinor transition determines the dual quotient bundle.
The growth allowed by that transition in the two affine charts determines
the blocks of a global holomorphic automorphism.

Use the dual quotient frame matrices \(\mathsf U_0\) and
\(\mathsf U_\infty\) associated with the spinors \((1,\zeta)\) and
\((\upsilon,1)\), where \(\upsilon=\zeta^{-1}\) on the overlap.
Equation~\eqref{eq:spinor-transition} gives
\(\mathsf U_0=\mathsf U_\infty D(\zeta)\), with
\begin{equation}\label{eq:transition}
 D(\zeta)=\begin{pmatrix}\zeta\Id_2&0\\0&1\end{pmatrix}.
\end{equation}
The spinor matrix \(Q\) has rank two in each chart, so its columns
are local holomorphic frames for the vector annihilator
\(\theta^\perp\). Each of these two frame columns transforms by
\(\zeta\), while the scalar frame column is the same in both charts.
Consequently the diagonal transition gives the explicit splitting
\begin{equation}\label{eq:bundle-splitting}
 E^*\simeq\cO(-1)\oplus\cO(-1)\oplus\cO.
\end{equation}
To fix the sign, the tautological line over \(\mathbb{CP}^1\) has
local frames \(e_0=(1,\zeta)^T\) and \(e_\infty=(\upsilon,1)^T\),
which satisfy \(e_0=\zeta e_\infty\). This is our convention for
\(\mathcal O(-1)\). The corresponding coordinates transform by the
inverse multiplier, \(v_0=\zeta^{-1}v_\infty\).
The splitting \eqref{eq:bundle-splitting} therefore follows directly
from the spinor frames in \eqref{eq:qspinor} and the transition
\eqref{eq:transition}, without a bundle classification theorem.

For an endomorphism of the split bundle, the same transition gives
separate growth bounds for the four blocks. Holomorphicity on both charts
therefore determines the following matrix form.

\begin{lemma}\label{lem:Hform}
For each fixed \(x\), let \(H(x,\cdot):E^*\to E^*\) be a global
holomorphic automorphism. In a local spinor frame it has the form
\begin{equation}\label{eq:Hform}
 H=\begin{pmatrix}A(x)&0\\C(x;s,t)&\delta_H(x)\end{pmatrix},
\end{equation}
where \(A\in GL(2,\C)\), \(\delta_H\in\C^\times\), and \(C\) is a
\(1\times2\) row homogeneous of degree one in \((s,t)\).
If the local matrices of \(H\) are jointly smooth in \(x\) and the
real chart coordinates, then \(A\), \(\delta_H\), and the two
coefficient rows of \(C\) are smooth in \(x\).
\end{lemma}

\begin{proof}
Fix \(x\) and write the chart matrices in the \(2+1\) block
decomposition as
\[
 H_0(\zeta)=\begin{pmatrix}A_0(\zeta)&B_0(\zeta)\\
                          C_0(\zeta)&\delta_{H,0}(\zeta)\end{pmatrix},
 \qquad
 H_\infty(\upsilon)=\begin{pmatrix}A_\infty(\upsilon)&B_\infty(\upsilon)\\
                          C_\infty(\upsilon)&\delta_{H,\infty}(\upsilon)\end{pmatrix}.
\]
Every block is entire in its own affine chart coordinate. The dual
transition law \eqref{eq:H-chart-covariance}, written as
\[
 H_0(\zeta)=D(\zeta)^{-1}H_\infty(\zeta^{-1})D(\zeta)
\]
is equivalent on \(\zeta\ne0\) to
\begin{equation}\label{eq:proof-h-block-transitions}
 \begin{aligned}
 A_0(\zeta)&=A_\infty(\zeta^{-1}),&
 \delta_{H,0}(\zeta)&=\delta_{H,\infty}(\zeta^{-1}),\\
 B_0(\zeta)&=\zeta^{-1}B_\infty(\zeta^{-1}),&
 C_0(\zeta)&=\zeta C_\infty(\zeta^{-1}).
 \end{aligned}
\end{equation}
By \eqref{eq:proof-h-block-transitions}, holomorphicity of the
infinity-chart blocks near \(\upsilon=0\)
makes \(A_0\) and \(\delta_{H,0}\) bounded for large \(|\zeta|\).
They are bounded on compact sets as well, so Liouville's theorem
applied entrywise makes them constant in \(\zeta\). We denote these
blocks by \(A(x)\) and \(\delta_H(x)\). The entire column \(B_0\)
is \(O(|\zeta|^{-1})\) at infinity, so the same theorem gives
\(B_0=0\), and \eqref{eq:proof-h-block-transitions} gives
\(B_\infty=0\).

Equation~\eqref{eq:proof-h-block-transitions} makes the entire row
\(C_0\) \(O(1+|\zeta|)\). Cauchy's estimates
on expanding circles show that every derivative of order at least
two vanishes. Hence there are \(1\times2\) rows \(C_s(x),C_t(x)\)
such that
\begin{equation}\label{eq:proof-h-affine-row}
 C_0(x,\zeta)=C_s(x)+\zeta C_t(x),\qquad
 C_\infty(x,\upsilon)=\upsilon C_s(x)+C_t(x).
\end{equation}
The expressions in \eqref{eq:proof-h-affine-row} are the restrictions
to \((1,\zeta)\) and \((\upsilon,1)\) of
the same homogeneous row
\[
 C(x;s,t)=sC_s(x)+tC_t(x).
\]
Thus the coefficient rows are common to both charts, even though
their coordinate expressions differ. Block triangularity gives
\(\det H=\det A\,\delta_H\); invertibility of \(H\) therefore forces
\(A\in GL(2,\C)\) and \(\delta_H\ne0\).

Finally, the coefficients can be recovered by fixed chart evaluations:
\begin{equation}\label{eq:proof-h-coefficient-evaluations}
 \begin{gathered}
 A(x)=A_0(x,0),\qquad \delta_H(x)=\delta_{H,0}(x,0),\\
 C_s(x)=C_0(x,0),\qquad
 C_t(x)=C_0(x,1)-C_0(x,0).
 \end{gathered}
\end{equation}
Joint smoothness of the chart matrices makes each function in
\eqref{eq:proof-h-coefficient-evaluations} smooth in \(x\). On compact spatial subsets, their derivatives of
every fixed finite order, \(A^{-1}\), and \(\delta_H^{-1}\) are bounded
by smoothness, nonvanishing, and compactness.
\end{proof}

The zero upper-right block and the degree-one lower-left row are
consequences of holomorphicity and the transition law; no spatial
transport equation has yet been used. The coefficient evaluations
\eqref{eq:proof-h-coefficient-evaluations} also fix the same smooth
functions across the two charts.

For the transport \(H\) of Corollary~\ref{cor:dual-relative}, we use
the identity normalization on \(E\) and \eqref{eq:Hform} to obtain
\[
 A=I_2,\qquad \delta_H=1,\qquad C_s=C_t=0\quad\text{on }E.
\]
The block transport equations will use the smooth functions
\(A,\delta_H,C_s,C_t\) identified in
\eqref{eq:proof-h-coefficient-evaluations}, with their chart expressions
fixed by \eqref{eq:proof-h-affine-row}.

\subsection{Transport identities and coefficient recovery}

The block equations satisfied by the dual relative transport
\(H\) provide the final step in the recovery of the Lam\'e
parameters. They first determine the auxiliary parameter
\(\kappa\) and subsequently yield a differential relation
leading to the determination of \(\mu\). The analysis of
the lower-left block involves a direction-independent
matrix that preserves every complex null line. The
following lemma characterizes matrices with this property.

\begin{lemma}\label{lem:null-line-rigidity}
Let \(M_{\mathrm{null}}\in M_3(\C)\).  If
\[
 M_{\mathrm{null}}\theta\in\C\theta
 \qquad\text{for every }\theta\in\C^3\setminus\{0\}
 \text{ with }\theta\cdot\theta=0,
\]
then \(M_{\mathrm{null}}=q\Id_3\) for some \(q\in\C\).
\end{lemma}

\begin{proof}
Fix pairwise distinct indices \(i,j,k\). Apply the hypothesis to \(e_i\pm\mathrm i e_j\).
Their \(k\)-th coordinates vanish, so the \(k\)-th coordinates of their images give
\[
 (M_{\mathrm{null}})_{ki}+\mathrm i(M_{\mathrm{null}})_{kj}=0,\qquad (M_{\mathrm{null}})_{ki}-\mathrm i(M_{\mathrm{null}})_{kj}=0.
\]
Hence \((M_{\mathrm{null}})_{ki}=(M_{\mathrm{null}})_{kj}=0\). We use the triples
\((i,j,k)=(1,2,3),(1,3,2),(2,3,1)\) to eliminate every
off-diagonal entry. The matrix is now diagonal. For each \(i\ne j\),
the components with indices \(i\) and \(j\) of
\(M_{\mathrm{null}}(e_i+\mathrm i e_j)\in\C(e_i+\mathrm i e_j)\)
give \((M_{\mathrm{null}})_{ii}=(M_{\mathrm{null}})_{jj}\). All three diagonal entries therefore
coincide, so \(M_{\mathrm{null}}=qI_3\), with \(q=\tr M_{\mathrm{null}}/3\).
\end{proof}

The conclusion requires one fixed matrix for all null directions.
In the rigidity proof, this condition is verified before applying the
lemma to the matrix obtained from the projected connection equation.

A trace-free equation for \(\nabla^2\rho\) does not directly determine
\(\Delta\rho\). Adjoining the Laplacian to \(\nabla\rho\) and commuting
third derivatives gives a homogeneous first-order system for these four
components, which can be propagated along spatial paths.

\begin{lemma}\label{lem:tracefree-closure}
Let \(\mathcal D\subset\R^3\) be open, let
\(\rho\in C^3(\mathcal D)\), and let
\(\beta\in C^1(\mathcal D;\R^3)\). For a \(3\times3\) matrix
\(M_{\mathrm{test}}\), write \(\tf M_{\mathrm{test}}=M_{\mathrm{test}}-(\tr M_{\mathrm{test}})I_3/3\).
Define \(\mathscr H_\beta(x)v=\tf(\beta(x)\otimes v+v\otimes\beta(x))\).
Set \(r=\nabla\rho\), \(w=\Delta\rho\), and
\begin{equation}\label{eq:proof-closure-tensor}
 T_{ij}^{\ k}=\beta_i\delta_{jk}+\beta_j\delta_{ik}
                   -\frac23\delta_{ij}\beta_k.
\end{equation}
If
\begin{equation}\label{eq:proof-closure-tracefree}
 \tf(\nabla^2\rho)=\tf(\beta\otimes r+r\otimes\beta)=\mathscr H_\beta(x)r,
\end{equation}
then \(U=(r_1,r_2,r_3,w)^T\) satisfies a closed first-order system
\begin{equation}\label{eq:proof-closure-system}
 \partial_jU=\mathcal C_jU\qquad (j=1,2,3)
\end{equation}
with continuous \(4\times4\) coefficient matrices \(\mathcal C_j\).
Here all indices range from one to three, and repeated indices are
summed. More precisely,
\[
 \partial_jr_i=\frac13\delta_{ij}w+T_{ij}^{\ k}r_k,\qquad
 \partial_jw=A_jw+B_j^{\ k}r_k,
\]
where
\begin{equation}\label{eq:proof-closure-a-coefficient}
 A_j=\frac12\sum_iT_{ij}^{\ i}=\frac53\beta_j
\end{equation}
and
\begin{equation}\label{eq:proof-closure-b-coefficient}
 B_j^{\ k}=\frac32\left(\sum_i\partial_iT_{ij}^{\ k}
                  +\sum_{i,\ell}T_{ij}^{\ \ell}T_{\ell i}^{\ k}\right).
\end{equation}
\end{lemma}

\begin{proof}
Since \(\rho\in C^3\), its Hessian is symmetric, its third
derivatives commute, and \(U=(r,w)^T\) belongs to \(C^1\).
Restore the trace in \eqref{eq:proof-closure-tracefree}, using
\eqref{eq:proof-closure-tensor}, to obtain
\begin{equation}\label{eq:proof-closure-restored-hessian}
 \partial_jr_i=\frac13\delta_{ij}w+T_{ij}^{\ k}r_k.
\end{equation}
Differentiate \eqref{eq:proof-closure-restored-hessian} in
\(x_i\) and sum over \(i\). The left side becomes
\(\partial_j\sum_i\partial_i r_i=\partial_jw\), while the trace
term contributes \(\partial_jw/3\). Substitute
\eqref{eq:proof-closure-restored-hessian} for \(\partial_i r_k\) to obtain
\begin{equation}\label{eq:proof-closure-laplacian-derivative}
 \frac23\partial_jw=(\partial_iT_{ij}^{\ k})r_k
       +T_{ij}^{\ k}\left(\frac13\delta_{ki}w+T_{ki}^{\ \ell}r_\ell\right).
\end{equation}
In \eqref{eq:proof-closure-laplacian-derivative}, we collect the
coefficient of \(w\) and rename the summed indices in the coefficient
of \(r\) to prove \eqref{eq:proof-closure-a-coefficient} and
\eqref{eq:proof-closure-b-coefficient}. By direct contraction of
\eqref{eq:proof-closure-tensor}, we obtain
\begin{equation}\label{eq:proof-closure-linear-contractions}
 \sum_iT_{ij}^{\ i}=\frac{10}{3}\beta_j,\qquad
 \sum_i\partial_iT_{ij}^{\ k}
  =(\Div\beta)\delta_{jk}+\partial_k\beta_j-\frac23\partial_j\beta_k.
\end{equation}
For the quadratic contraction of \eqref{eq:proof-closure-tensor},
symmetry of the two lower indices
and \(\sum_iT_{ii}^{\ k}=0\) give
\[
 \begin{aligned}
 \sum_{i,\ell}T_{ij}^{\ \ell}T_{\ell i}^{\ k}
 &=\sum_i\beta_iT_{ji}^{\ k}
   +\beta_j\sum_iT_{ii}^{\ k}
   -\frac23\sum_\ell\beta_\ell T_{\ell j}^{\ k}\\
 &=\frac13\left(|\beta|^2\delta_{jk}
                         +\frac13\beta_j\beta_k\right).
 \end{aligned}
\]
Thus
\begin{equation}\label{eq:proof-closure-quadratic-contraction}
 \sum_{i,\ell}T_{ij}^{\ \ell}T_{\ell i}^{\ k}
 =\frac13|\beta|^2\delta_{jk}+\frac19\beta_j\beta_k.
\end{equation}
Substitute \eqref{eq:proof-closure-linear-contractions} and
\eqref{eq:proof-closure-quadratic-contraction} into
\eqref{eq:proof-closure-a-coefficient} and
\eqref{eq:proof-closure-b-coefficient} to obtain the closed-system matrices
\[
 \mathcal C_j=
 \begin{pmatrix}
   (T_{ij}^{\ k})_{i,k=1}^3 & (\delta_{ij}/3)_{i=1}^3\\
   (B_j^{\ k})_{k=1}^3 & 5\beta_j/3
 \end{pmatrix},
\]
where
\[
 B_j^{\ k}=\frac32\left[
  (\Div\beta)\delta_{jk}+\partial_k\beta_j-\frac23\partial_j\beta_k
  +\frac13|\beta|^2\delta_{jk}+\frac19\beta_j\beta_k
 \right].
\]
The entries involve only \(\beta\), its first derivatives, and
their products, so \(\beta\in C^1\) makes them continuous.
The resulting system has no independent source term: every spatial
derivative of \(U\) is a linear combination of its four components.
The field \(U\) is already known to solve all three equations in
\eqref{eq:proof-closure-system};
propagation of its zero value along paths therefore requires no
flatness assumption on the matrices \(\mathcal C_j\).
\end{proof}

The fourth component \(w\) restores the trace omitted from
\eqref{eq:proof-closure-tracefree}. The formulas
\eqref{eq:proof-closure-a-coefficient}--\eqref{eq:proof-closure-b-coefficient}
use only the first derivatives of \(\beta\), so their continuity follows
from the stated \(C^1\) hypothesis.

We now apply the holomorphic block form to the explicit connection
\eqref{eq:Gamma}. The rigidity statement assumes a global intertwiner;
its construction from DN equality is supplied separately by
Theorem~\ref{thm:quotient-reconstruction} and
Corollary~\ref{cor:dual-relative}. In the boundary inverse problem,
Lemma~\ref{lem:common-extension} gives the coefficient hypotheses and
agreement on \(E\), while Theorem~\ref{thm:physical-nine-moments} gives
the complete matrix moments used in that reconstruction.

\begin{theorem}\label{thm:rigidity}
Let \(\mathcal D\subset\R^3\) be connected and open, and let
\((\mu_j,\kappa_j)\in C^\infty(\mathcal D;\R^2)\), \(j=1,2\),
satisfy \(\mu_j>0\) and \(1/4<\kappa_j<1\). Assume that the
two pairs coincide on a nonempty exterior open set
\(E\subset\mathcal D\).

Suppose that their dual quotient connections \eqref{eq:Gamma} are
intertwined by a bundle automorphism \(H\) whose local matrices are
jointly smooth in \(x\) and the real spinor-chart coordinates,
holomorphic in the projective null direction for each fixed \(x\),
and satisfy the dual transition law \eqref{eq:H-chart-covariance}.
Assume also that \(H=I\) on \(E\) for every direction and that its
local matrices satisfy \eqref{eq:Hintertwine}. Then
\[
 \mu_1=\mu_2,\qquad \kappa_1=\kappa_2.
\]
\end{theorem}

\begin{proof}
By Lemma~\ref{lem:Hform}, \(A\) and \(\delta_H\) are smooth and
direction-independent, and \(C=sC_s+tC_t\) has smooth coefficient rows.
The spinor representative is held fixed under spatial differentiation,
so \(D_\theta Q=D_\theta n=D_\theta\theta=0\).
Write
\[
 \begin{gathered}
 \Gamma_j=\begin{pmatrix}M_j&N_j\\\mathsf P_j&0\end{pmatrix},\qquad
 m_j=\nabla\log\mu_j,\qquad d_j=\kappa_j\sqrt{\mu_j},\\
 M_j=n\mathfrak r_{m_j},\qquad N_j=2d_jn,\qquad
 \mathsf P_j=-\theta^TF_{\mu_j}Q.
 \end{gathered}
\]
The row \(\mathsf P_j\) is the \(1\times2\) lower-left block of
\(\Gamma_j\); the leaf connection matrices \(P_j\) of Section~\ref{sec:transport}
have size \(3\times3\). Positivity gives \(d_j>0\)
everywhere, so division by \(d_j\) is legitimate.
Using \eqref{eq:Hform}, we split the intertwining equation
\begin{equation}\label{eq:Hintertwine}
 2D_\theta H+\Gamma_1H-H\Gamma_2=0
\end{equation}
into the four equations
\begin{align}
 2D_\theta A+M_1A+N_1C-AM_2&=0,\label{eq:H-block-ul}\\
 N_1\delta_H-AN_2&=0,\label{eq:H-block-ur}\\
 2D_\theta C+\mathsf P_1A-CM_2-\delta_H \mathsf P_2&=0,\label{eq:H-block-ll}\\
 2D_\theta\delta_H-CN_2&=0.\label{eq:H-block-lr}
\end{align}

Equations~\eqref{eq:H-block-ul}--\eqref{eq:H-block-lr} hold in
either affine spinor chart. They also
hold for every nonzero homogeneous spinor: by
\eqref{eq:theta-spinor} and \eqref{eq:qspinor}, under
\((s,t)\mapsto b(s,t)\), the quantities \(Q,n,C\) scale by
\(b\), \(\theta\) scales by \(b^2\), and \(A,\delta_H\) are
unchanged. The upper-left and lower-right equations have degree
two, the upper-right equation degree one, and the lower-left
equation degree three. This homogeneity carries the chart equations
to all representatives of the projective null cone.

The upper-right equation \eqref{eq:H-block-ur}, with the blocks
in \eqref{eq:Gamma}, gives
\[
 d_1\delta_H\,n=d_2An.
\]
At fixed \(x\), we apply \eqref{eq:H-block-ur} to the spinors
\((1,0)\) and \((0,1)\),
which give \(n=(0,1)^T\) and \(n=(-1,0)^T\), respectively.
The same scalar \(d_1\delta_H/d_2\) is therefore the multiplier
of \(A\) on both coordinate vectors of \(\C^2\). Hence
\begin{equation}\label{eq:Aalpha}
 A=\alpha\Id_2,
 \qquad
 \alpha=\frac{d_1\delta_H}{d_2}.
\end{equation}
Substituting \eqref{eq:Aalpha} and the blocks of \eqref{eq:Gamma}
into \eqref{eq:H-block-ul} gives
\begin{equation}\label{eq:proof-rigidity-alpha-equation}
 2(D_\theta\alpha)\Id_2+n\left[\alpha(\mathfrak r_{m_1}-\mathfrak r_{m_2})+2d_1C\right]=0.
\end{equation}
The outer product in \eqref{eq:proof-rigidity-alpha-equation}
has rank at most one, whereas a nonzero scalar
multiple of \(I_2\) has rank two. Hence \(D_\theta\alpha=0\)
for every null direction. Since
\(2e_i=(e_i+\mathrm i e_j)+(e_i-\mathrm i e_j)\) for \(i\ne j\),
these null directions span \(\C^3\). Hence \(\nabla\alpha=0\),
and connectedness makes \(\alpha\) constant on \(\mathcal D\).
On \(E\), the equal coefficient pairs give
\(d_1=d_2\), and \(H=I\) gives \(\delta_H=1\). Thus \(\alpha=1\)
throughout \(\mathcal D\), and \eqref{eq:Aalpha} gives \(A=I_2\).
Set \(\alpha=1\) in \eqref{eq:proof-rigidity-alpha-equation}
to obtain the remaining outer product in the form \(n\mathsf r=0\) with \(n\ne0\); its row
\(\mathsf r\) must vanish. Consequently
\begin{equation}\label{eq:Cblock}
 C=-\frac{\mathfrak r_{m_1-m_2}}{2d_1}.
\end{equation}
Set
\[
 \rho=\log(\mu_1/\mu_2).
\]
Positivity of the shear moduli makes \(\rho\) a smooth real function.
Equation~\eqref{eq:H-block-lr}, with \eqref{eq:Cblock}, and the
identity \(\mathfrak r_{\nabla\rho}n=D_\theta\rho\) from
\eqref{eq:Qn} and \eqref{eq:rmathfrak} give
\begin{equation}\label{eq:proof-rigidity-delta-equation}
 2D_\theta\delta_H+\frac{d_2}{d_1}D_\theta\rho=0.
\end{equation}
Since \(\alpha=1\), \eqref{eq:Aalpha} gives
\(\delta_H=d_2/d_1>0\), so \eqref{eq:proof-rigidity-delta-equation}
is \(2D_\theta\delta_H+\delta_H D_\theta\rho=0\).
Applying the integrating factor to \eqref{eq:proof-rigidity-delta-equation}, we obtain
\begin{equation}\label{eq:proof-rigidity-integrating-factor}
 D_\theta(\delta_H e^{\rho/2})
 =e^{\rho/2}\left(D_\theta\delta_H+\tfrac12\delta_H D_\theta\rho\right)=0
\end{equation}
for every null direction. The null vectors \(e_i\pm\mathrm i e_j\),
\(i\ne j\), span \(\C^3\), since
\(2e_i=(e_i+\mathrm i e_j)+(e_i-\mathrm i e_j)\).
Using \eqref{eq:proof-rigidity-integrating-factor}, we therefore obtain
\(\nabla(\delta_H e^{\rho/2})=0\). Connectedness makes
\(\delta_H e^{\rho/2}\) constant on \(\mathcal D\). On \(E\),
\(\rho=0\) and \(\delta_H=1\), so the constant equals one. Hence
\begin{equation}\label{eq:proof-rigidity-delta}
 \delta_H=e^{-\rho/2}.
\end{equation}
The definition
\eqref{eq:dmu} of \(d_j\) also gives
\[
 \frac{d_2}{d_1}=e^{-\rho/2}\frac{\kappa_2}{\kappa_1}.
\]
Combining \eqref{eq:dmu}, \eqref{eq:Aalpha}, and
\eqref{eq:proof-rigidity-delta}, with \(\alpha=1\), we obtain
\begin{equation}\label{eq:kappaequal}
 \kappa_1=\kappa_2=: \kappa.
\end{equation}

It remains to prove \(\rho=0\) from the lower-left equation
\eqref{eq:H-block-ll}; the other blocks have already given
\eqref{eq:kappaequal}. We write
\begin{equation}\label{eq:proof-rigidity-gamma}
 \gamma=\frac{\nabla\rho}{\kappa\sqrt{\mu_1}}.
\end{equation}
With \(r=\nabla\rho\), \eqref{eq:rmathfrak} gives
\(\mathfrak r_{\nabla\rho}=-2r^TQ\), so \eqref{eq:Cblock} and
\eqref{eq:proof-rigidity-gamma} give
\(C=\gamma^TQ\). Equation~\eqref{eq:H-block-ll}, with \(A=I_2\), is
\[
 2D_\theta C-\theta^TF_{\mu_1}Q-Cn\mathfrak r_{m_2}
 +\delta_H\theta^TF_{\mu_2}Q=0.
\]
The identities \eqref{eq:Qn} and \eqref{eq:rmathfrak}, namely
\(Qn=-\theta/2\) and \(\mathfrak r_{m_2}=-2m_2^TQ\), give
\[
 Cn=-\frac{\gamma\cdot\theta}{2},\qquad
 -Cn\mathfrak r_{m_2}=-(\gamma\cdot\theta)m_2^TQ.
\]
With \((\nabla\gamma)_{ij}=\partial_j\gamma_i\), we also have
\(D_\theta C=\theta^T(\nabla\gamma)^TQ\).
Substituting \eqref{eq:Qn} and \eqref{eq:rmathfrak} into
\eqref{eq:H-block-ll} therefore gives
\begin{equation}\label{eq:proof-rigidity-projected-equation}
 \theta^T\bigl[2(\nabla\gamma)^T-F_{\mu_1}
       +\delta_H F_{\mu_2}-\gamma m_2^T\bigr]Q=0.
\end{equation}
Both \(F_{\mu_j}\) are symmetric by \eqref{eq:Fmu}. We take the
bilinear transpose of \eqref{eq:proof-rigidity-projected-equation},
using \eqref{eq:proof-rigidity-delta}, to obtain
\(Q^T\mathcal K\theta=0\), where
\begin{equation}\label{eq:Kcal}
 \mathcal K=2\nabla\gamma-F_{\mu_1}+e^{-\rho/2}F_{\mu_2}-m_2\gamma^T.
\end{equation}
For every nonzero spinor, \eqref{eq:qspinor} gives
\(\theta^TQ=0\) and \(Q\) has rank two. For clarity,
\(Q(a,b)^T=0\) gives \(ta=0\), \(sb=0\), and \(sa-tb=0\), which imply
\(a=b=0\) even at \(s=0\) or \(t=0\). Thus \(\operatorname{im}Q=\theta^\perp\).
Nondegeneracy of the complex-bilinear dot product gives
\((\theta^\perp)^\perp=\C\theta\), including when \(\theta\) is null.
Equivalently, \(\ker Q^T=\C\theta\), so
\(\mathcal K\theta\in\C\theta\) for every null direction.
At each fixed \(x\), the matrix \(\mathcal K(x)\) of
\eqref{eq:Kcal} is independent of the spinor representative and appears in
the projected identities obtained from
\eqref{eq:proof-rigidity-projected-equation}: it has no direction dependence. Thus
Lemma~\ref{lem:null-line-rigidity} applies pointwise and gives
\(\mathcal K=q(x)I_3\), with the smooth scalar \(q=\tr\mathcal K/3\).

The scalar-matrix conclusion makes the symmetric trace-free part of
\eqref{eq:Kcal} vanish. To express that part in coefficient gradients
and the Hessian of \(\rho\), we put
\[
 r=\nabla\rho,
 \qquad
 k=\nabla\log\kappa,
 \qquad
 c_1=\mu_1^{-1/2}.
\]
By \eqref{eq:proof-rigidity-delta},
\[
 \delta_H\mu_2^{-1/2}=\mu_1^{-1/2}=c_1.
\]
Moreover, \eqref{eq:proof-rigidity-gamma} gives
\begin{equation}\label{eq:proof-rigidity-gamma-weight}
 \gamma=\frac{c_1}{\kappa}r,
 \qquad
 \nabla\log\left(\frac{c_1}{\kappa}\right)
 =-\frac12m_1-k.
\end{equation}
With the convention \((\nabla\gamma)_{ij}=\partial_j\gamma_i\),
differentiating \eqref{eq:proof-rigidity-gamma-weight} gives
\begin{equation}\label{eq:grad-gamma-expanded}
 2\nabla\gamma
 =\frac{2c_1}{\kappa}\nabla^2\rho
 -\frac{c_1}{\kappa}r\otimes m_1
 -\frac{2c_1}{\kappa}r\otimes k.
\end{equation}
On the other hand, we use \eqref{eq:Fmu},
\eqref{eq:proof-rigidity-delta}, and \(m_1-m_2=r\) to obtain
\begin{align*}
 -F_{\mu_1}+\delta F_{\mu_2}
 =2c_1\Bigl[
 -\nabla^2\rho
 +(m_1\otimes m_1-m_2\otimes m_2)
 +\Xi_\rho\Id
 \Bigr],
\end{align*}
where
\[
 \Xi_\rho=\Delta\rho-(m_1+m_2)\cdot r.
\]
The identity-matrix contribution from \eqref{eq:Fmu} disappears
under the trace-free projection. We use the identities
\[
 m_1\otimes m_1-m_2\otimes m_2
 =\frac12\Bigl[(m_1+m_2)\otimes r+r\otimes(m_1+m_2)\Bigr],
\]
\[
 -m_2\gamma^T=-\frac{c_1}{\kappa}m_2\otimes r,
\]
and the remaining product identity
\begin{equation}\label{eq:proof-rigidity-mixed-product}
 r\otimes m_1+m_2\otimes r
 =\frac12\bigl[(m_1+m_2)\otimes r+r\otimes(m_1+m_2)\bigr]
\end{equation}
in \eqref{eq:Kcal}.
Combining \eqref{eq:proof-rigidity-mixed-product}, \eqref{eq:Kcal},
\eqref{eq:grad-gamma-expanded}, and \eqref{eq:Fmu}, we obtain the full expansion
\begin{equation}\label{eq:proof-rigidity-k-expansion}
 \begin{aligned}
 \frac{\mathcal K}{c_1}={}&
 2\left(\frac1\kappa-1\right)\nabla^2\rho\\
 &+\left(1-\frac1{2\kappa}\right)
   \bigl[(m_1+m_2)\otimes r+r\otimes(m_1+m_2)\bigr]\\
 &-\frac2\kappa r\otimes k+2\Xi_\rho I_3.
 \end{aligned}
\end{equation}
Take the trace-free part of the symmetric part in
\eqref{eq:proof-rigidity-k-expansion},
\((\mathcal K+\mathcal K^T)/(2c_1)\), which vanishes because
Lemma~\ref{lem:null-line-rigidity} gives \(\mathcal K=qI_3\), and obtain
\begin{align}\label{eq:tf-hessian}
0={}&2\left(\frac1\kappa-1\right)\tf(\nabla^2\rho)\\
&+\tf\Bigg[
\left(1-\frac1{2\kappa}\right)
\left((m_1+m_2)\otimes r+r\otimes(m_1+m_2)\right)\notag\\
&\hspace{18mm}-\frac1\kappa\left(r\otimes k+k\otimes r\right)
\Bigg].\notag
\end{align}
Equation~\eqref{eq:proof-rigidity-k-expansion} also gives
\[
 \frac{\mathcal K-\mathcal K^T}{2}
 =-\frac{c_1}{\kappa}(r\otimes k-k\otimes r).
\]
Thus the skew-symmetric part of \eqref{eq:proof-rigidity-k-expansion},
with \(\mathcal K=q\Id_3\), additionally gives
\begin{equation}\label{eq:skew-extra}
 r\otimes k-k\otimes r=0,
\end{equation}
although this identity is not needed below.

Since \eqref{eq:kappa-range} and \eqref{eq:kappaequal} give
\(1/4<\kappa<1\), the Hessian coefficient
\(2(1/\kappa-1)=2(1-\kappa)/\kappa\) is strictly positive.
Divide \eqref{eq:tf-hessian} by the Hessian coefficient and collect
the terms linear in \(r\) to obtain the smooth real vector field
\begin{equation}\label{eq:beta-explicit}
 \beta=\frac{2k-(2\kappa-1)(m_1+m_2)}{4(1-\kappa)}.
\end{equation}
Then \eqref{eq:tf-hessian} is precisely
\begin{equation}\label{eq:tf-linear}
 \tf(\nabla^2\rho)=\tf(\beta\otimes r+r\otimes\beta)
                    =\mathscr H_\beta(x)r,
\end{equation}
where, with indices ranging from 1 to 3 and repeated indices summed,
\begin{equation}\label{eq:T-explicit}
 T_{ij}^{\ k}=\beta_i\delta_{jk}+\beta_j\delta_{ik}
                   -\frac23\delta_{ij}\beta_k.
\end{equation}
Here \(T_{ij}^{\ k}=T_{ji}^{\ k}\) and \(\sum_iT_{ii}^{\ k}=0\).
For the fixed candidate pair, \(\beta\) and \(T\) are smooth coefficients
of a homogeneous linear system for \(U=(\nabla\rho,\Delta\rho)^T\).

Set \(w=\Delta\rho\) and \(U=(r_1,r_2,r_3,w)^T\).  Since
\(\mu_j,\kappa_j\in C^\infty\), we have \(\rho\in C^\infty\),
\(\beta\in C^\infty\), and hence \(T_{ij}^{\ k}\in C^1\).  Thus
Lemma~\ref{lem:tracefree-closure} applies to \eqref{eq:tf-linear} and
\eqref{eq:T-explicit}; hence \(U\) satisfies the closed system
\eqref{eq:proof-closure-system}, \(\partial_jU=\mathcal C_jU\),
with continuous \(4\times4\)
coefficient matrices. Since \(\rho=0\) on the open set \(E\), all
its derivatives vanish there; in particular \(r=0\), \(w=0\),
and hence \(U=0\) on \(E\).
Choose \(x_0\in E\) and an arbitrary \(x_1\in\mathcal D\). A piecewise
\(C^1\) path \(x(t)\subset\mathcal D\) joins them because
\(\mathcal D\) is connected and open. On each smooth segment, the chain rule and
\eqref{eq:proof-closure-system} give, for
\(V(t)=U(x(t))\),
\begin{equation}\label{eq:proof-propagation-path-ode}
 V'(t)=\left(\sum_jx_j'(t)\mathcal C_j(x(t))\right)V(t).
\end{equation}
The coefficient in \eqref{eq:proof-propagation-path-ode} is continuous.
Using uniqueness for the linear ODE
\eqref{eq:proof-propagation-path-ode} with zero initial value, we obtain
\(V=0\) on successive segments.
Thus \(U(x_1)=0\), and \(U=0\) throughout \(\mathcal D\).
ODE uniqueness applies to the existing solution \(U\) of
\eqref{eq:proof-closure-system}; no flatness assumption on
the matrices \(\mathcal C_j\) is required.
Thus \(\nabla\rho=0\); connectedness and \(\rho|_E=0\)
give \(\rho=0\) on \(\mathcal D\).
The definition of \(\rho\) yields \(\mu_1=\mu_2\), and
\eqref{eq:kappaequal} gives \(\kappa_1=\kappa_2\). \qedhere
\end{proof}

The open-set agreement supplies both the gradient and the Laplacian
of \(\rho\) as zero initial values for
\eqref{eq:proof-closure-system}. Equality of the coefficient values at
a single point does not supply this full initial state.

\section{Proof of the main theorem}\label{sec:mainproof}

The preceding results can now be combined to prove
Theorem~\ref{thm:main}. The common-extension argument
reduces the problem to two coefficient pairs on an
enclosing ball that agree in the exterior.
The complete physical matrix-moment identities then
yield, through the Cauchy--Plemelj construction,
an exterior-normalized relative transport.
The corresponding dual transport satisfies the
hypotheses of Theorem~\ref{thm:rigidity}, which implies
equality of the extended Lam\'e parameters and hence
of the original coefficients in \(\Omega\).

\begin{proof}[Proof of Theorem~\ref{thm:main}]
Lemmas~\ref{lem:common-extension} and \ref{lem:enlargement} give
smooth real strongly convex extensions to an origin-centered
enclosing ball \(B_R\), with equal DN maps and a common constant
exterior pair. We choose \(R_0<R\) with
\(\overline\Omega\subset B_{R_0}\), so the extensions coincide
outside \(\overline{B_{R_0}}\), and we set
\[
 E=\{x\in B_R:R_0<|x|<R\}.
\]
Fix the CGO balls and source ball as in Section~\ref{sec:transport}:
\[
 \overline{B_R}\Subset\Omega_{\mathrm{CGO}}\Subset\Omega_1\Subset\Omega_0,
 \qquad \overline\Omega_{\mathrm{CGO}}\subset\operatorname{int}K_{\mathrm{src}},
 \qquad K_{\mathrm{src}}\Subset\Omega_1,
\]
with both conjugated coefficient supports in \(\operatorname{int}K_{\mathrm{src}}\).

Theorem~\ref{thm:physical-nine-moments} supplies the fixed primal and
physical dual frames and their exact localized realizations. In the
original reduced variables these satisfy, columnwise,
\begin{equation}\label{eq:proof-main-exact-fields}
 M_1\boldsymbol z_{1,h}=0,\qquad u_{1,h}=T_1\boldsymbol z_{1,h},\qquad
 y_{2,h}=S_2^Tu_{2,h},\qquad L_2u_{2,h}=0
 \quad\hbox{in }\Omega_{\mathrm{CGO}}.
\end{equation}
The \(H^5\) reduced-remainder and \(H^3\) physical-adjoint bounds in
Theorem~\ref{thm:physical-nine-moments}, applied in
\eqref{eq:primal-adjoint}, give every entry of
\eqref{eq:leaf-moment} on each enclosing leaf disk. The limits are taken
at fixed direction and fixed tests, using the smooth frames selected
before the tests.

Applying Theorem~\ref{thm:quotient-reconstruction} to
\eqref{eq:leaf-moment} and the common exterior annuli, we obtain
a jointly smooth primal relative
automorphism \(G\) on \(B_R\times\mathbb{CP}^1\), holomorphic in the
projective direction at fixed \(x\), normalized on \(E\), and satisfying
\eqref{eq:proof-reconstructed-primal-transport} and
\eqref{eq:G-chart-covariance}.
Corollary~\ref{cor:dual-relative} gives \(H=G^{-T}\).
Using Lemma~\ref{lem:common-extension}, \eqref{eq:kappa-range},
Theorem~\ref{thm:quotient-reconstruction}, and
Corollary~\ref{cor:dual-relative}, we verify all hypotheses of
Theorem~\ref{thm:rigidity} on the connected domain
\(\mathcal D=B_R\), with the same \(E\). Applying
Theorem~\ref{thm:rigidity}, we obtain
\begin{equation}\label{eq:proof-main-coefficient-equality}
 \mu_1=\mu_2,\qquad \kappa_1=\kappa_2.
\end{equation}
Using \eqref{eq:proof-main-coefficient-equality} in
\eqref{eq:lambda-kappa}, we obtain \(\lambda_1=\lambda_2\).
Restrict the extended coefficients to \(\Omega\) to complete the proof.
\end{proof}

The quotient moments and the intertwiner used in this proof are
consequences of the physical DN equality. The recovery therefore uses
the boundary map stated in Theorem~\ref{thm:main}, without an additional
boundary measurement for the reduced system.

\appendix

\section{Physical realization and Sobolev remainder estimates}\label{app:cgo-input}

The localized right inverse constructed in
Proposition~\ref{prop:appendix-interior-right-inverse},
together with the transport construction in
Theorem~\ref{thm:autonomous-cgo},
yields the exact CGO realizations used in
Theorems~\ref{thm:graded-frame} and
\ref{thm:physical-nine-moments}. 
The following proposition records both ordinary Sobolev orders
for one correction; the \(k=1\) case controls the physical moment
limit, while \(k=5\) also provides higher regularity on the entire
working ball for the pointwise rank estimate of the exact physical
adjoint frame.
The fixed balls satisfy
\[
 \overline{B_R}\Subset\Omega_{\mathrm{CGO}}\Subset\Omega_1\Subset\Omega_0,
 \qquad \overline\Omega_{\mathrm{CGO}}\subset\operatorname{int}K_{\mathrm{src}},
 \qquad K_{\mathrm{src}}\Subset\Omega_1,
\]
Both conjugated coefficient supports lie in
\(\operatorname{int}K_{\mathrm{src}}\). With
\(T_\theta=2D_\theta+\widetilde A_\theta\),
\begin{equation}\label{eq:proof-app-transport-recursion}
 \begin{aligned}
 T_\theta\widetilde a_0&=0
       &&\text{near }\overline\Omega_1,\\
 T_\theta\widetilde a_1&=\widetilde M\widetilde a_0
       &&\text{near }\overline\Omega_{\mathrm{CGO}}.
 \end{aligned}
\end{equation}
Choose \(\chi_0\in C_c^\infty(\operatorname{int}K_{\mathrm{src}})\), equal to one
near \(\overline\Omega_{\mathrm{CGO}}\), with support where the
first-correction equation holds.

\begin{proposition}
\label{prop:appendix-physical-realization}
Fix a normalized complex null direction \(\theta\) and the nested
balls above. Let smooth amplitudes \(\widetilde a_0,\widetilde a_1\)
satisfy \eqref{eq:proof-app-transport-recursion} on the specified
neighborhoods, and suppose that
\(F=\chi_0\widetilde M\widetilde a_1\in H^k(\Omega_0)\)
for a fixed integer \(k\geq1\). Then there exists a correction
\(\widetilde r_h\), obtained from the \emph{same} right inverse at
Sobolev order \(k\), such that
\[
 \widetilde{\boldsymbol z}_h=e^{-\theta\cdot x/h}
    (\widetilde a_0+h\widetilde a_1+\widetilde r_h),
 \qquad \widetilde M\widetilde{\boldsymbol z}_h=0
 \quad\text{in }\Omega_{\mathrm{CGO}},
\]
and
\begin{equation}\label{eq:appendix-cgo-two-sobolev}
 \|\widetilde r_h\|_{H^k(\Omega_0)}\leq Ch^2\|F\|_{H^k(\Omega_0)},
 \qquad
 \|\widetilde r_h\|_{H^{k+1}(\Omega_0)}\leq Ch\|F\|_{H^k(\Omega_0)}.
\end{equation}
The constants are uniform in compact normalized-direction families
when the prescribed source norms are uniformly bounded. No support
condition is asserted for \(\widetilde r_h\) itself.
\end{proposition}
\begin{proof}
Extend \(F\) by zero from \(\Omega_0\) and apply the
localized right inverse of Theorem~\ref{lem:ER-solvability}
at the selected order \(k\):
\begin{equation}\label{eq:proof-app-residual-correction}
 \widetilde r_h=-R_{h,\theta}^{(k)}(hF)
 =-R_{h,\theta}^{(k)}
                 (h\chi_0\widetilde M\widetilde a_1).
\end{equation}
By \eqref{eq:system-cgo-right-inverse-equation}, the conjugated
reduced equation annihilates the leading residual on the region
where \(\chi_0=1\). The two bounds in
\eqref{eq:appendix-cgo-two-sobolev} follow from
\eqref{eq:system-cgo-right-inverse}. In particular, for \(k=5\),
we have
\begin{equation}\label{eq:proof-app-cgo-estimate}
 \begin{gathered}
 \widetilde{\boldsymbol z}_h=e^{-\theta\cdot x/h}
       (\widetilde a_0+h\widetilde a_1+\widetilde r_h),\\
 \widetilde M\widetilde{\boldsymbol z}_h=0\quad\text{in }\Omega_{\mathrm{CGO}},\\
 \|\widetilde r_h\|_{H^5(\Omega_{\mathrm{CGO}})}
 \leq C h^2\|\chi_0\widetilde M\widetilde a_1\|_{H^5(\Omega_0)}.
 \end{gathered}
\end{equation}

Applying the same localized right inverse, with $k=5$ in
\eqref{eq:system-cgo-right-inverse}, to the source
$h\chi_0\widetilde M\widetilde a_1$ yields both ordinary-Sobolev bounds

\begin{equation}\label{eq:app-H5-H6-remainder}
\begin{aligned}
 \|\widetilde r_h\|_{H^5(\Omega_0)}
       &\leq Ch^2\|\chi_0\widetilde M\widetilde a_1\|_{H^5(\Omega_0)},\\
 \|\widetilde r_h\|_{H^6(\Omega_0)}
       &\leq Ch\|\chi_0\widetilde M\widetilde a_1\|_{H^5(\Omega_0)}.
\end{aligned}
\end{equation}

For a prescribed first correction, an \(H^7\) bound on the cutoff
support controls the source norm. For a constructed correction, the
fixed cutoffs and the \(C_{x_\perp}^9C_z^{8,\alpha}\) coefficient, amplitude,
frame, inverse, and partition-weight bounds in
Theorem~\ref{thm:autonomous-cgo} suffice. Multiplication by the common
\(\mathcal R\) preserves Sobolev orders. The reduced \(H^5\) error is
therefore \(O(h^2)\) on the entire working ball, and the additional
\(h^{-1}J_\theta r_h\) term is \(O(h)\) in \(H^3\), retaining
\(K_0a_0-J_\theta a_1\).
\end{proof}

\begin{remark}
\label{rem:appendix-minimal-moment-regularity}
For a supplementary physical adjoint polarization with
\(J_\theta a_0=0\), the normalized lift has an error consisting of
\(hK_0a_1-h^{-1}J_\theta r_h+K_0r_h\).
Since \(J_\theta\) has order at most one and \(K_0\) order at most
two, the estimates from Proposition~\ref{prop:appendix-physical-realization}
at \(k=1\) imply, for fixed smooth amplitudes,
\[
 \|hK_0a_1-h^{-1}J_\theta r_h+K_0r_h\|_{L^2}
 \leq C\bigl(h\|a_1\|_{H^2}
       +h^{-1}\|r_h\|_{H^1}+\|r_h\|_{H^2}\bigr)
 =O(h).
\]
This is enough for the \(L^2\)-based Betti moment limit.
The \(k=5\) specialization is retained for the stronger
\(H^3\) physical lift estimate and the pointwise rank control
of Theorem~\ref{thm:graded-frame}. The exact correction is not
assumed to vanish on the region where the source cutoff is one.
\end{remark}

The residual cutoff \(\chi_0\) multiplies the source, so its derivatives
do not enter the residual equation where \(\chi_0=1\). The exact
elliptic correction need not share its support or vanish where \(\chi=1\). Its estimate holds on the entire integration domain and
controls the product error through Lemma~\ref{lem:CGO-product-limit}.
By contrast, the additional commutator correction
\eqref{eq:proof-slow-localization-correction} preserves transverse
support and vanishes with its source on transverse neighborhoods where \(\chi\) is identically one.
Lemma~\ref{lem:cauchy-leaf-estimate} gains one leaf derivative and
no transverse derivative. The cylinder construction of
Lemma~\ref{lem:uniform-leafwise-frames} supplies smooth transverse
zero extensions and the required ambient transport domains.
Since \(D_\theta b=D_\theta g=0\), transverse tests and holomorphic
polynomials preserve leading transport. Each moment uses fixed
\(\theta,b,g\); neither bounds over all polynomial degrees nor
parameter derivatives of exact CGO fields are required.

\section{Checks of the factorization and phase signs}\label{app:factorization-check}

The appendix checks the scalar and vector coefficients of
Proposition~\ref{prop:factorization} in components and then verifies the
phase signs in the physical adjoint lift. Set
\[
 a=\mu^{-1}=c^2,\qquad c=\mu^{-1/2},\qquad
 \mathsf K=\lambda+\mu,\qquad \Sigma=\lambda+2\mu.
\]
Indices range from one to three and repeated indices are summed.
Write \(f_i=\partial_i f\), \(a_{ij}=\partial_i\partial_j a\),
and similarly for the other scalar fields.

For \(u^{(f)}=a\nabla f-f\nabla a\) and
\(p_f=\Delta f-(\Delta a/a)f\), the cancellation of mixed
first derivatives gives
\(\eps(u^{(f)})_{ij}=af_{ij}-fa_{ij}\) and
\(\Div u^{(f)}=ap_f\). Thus

\begin{equation*}
 \sigma(u^{(f)})_{ij}
   =a\lambda p_f\delta_{ij}+2f_{ij}-2\mu f a_{ij}.
\end{equation*}

Taking the divergence of \(\sigma(u^{(f)})\) gives
\eqref{eq:L-scalar-part}. The scalar coefficient identities are
\begin{equation}\label{eq:proof-app-scalar-coefficients}
 \begin{gathered}
 \nabla a=-am,\qquad
 \nabla^2a=a(m\otimes m-\nabla^2\ell),\qquad
 -\Delta a/a=q_\mu,\\
 \Delta a\,I_3-\nabla^2a=\frac c2F_\mu,\qquad
 a\Sigma=(1-\kappa)^{-1}.
 \end{gathered}
\end{equation}
Using \eqref{eq:proof-app-scalar-coefficients} in
\eqref{eq:L-scalar-part}, we identify the coefficients of
\(\nabla f\), \(mf\), and the scalar gradient in \eqref{eq:M-explicit}.

For the vector contribution, set \(\Psi_v=c\mathsf K\,\Div v+(\lambda-\mu)c_jv_j\).
In \eqref{eq:L-general-expanded}, the mixed first derivatives cancel
by \(2\mu c_j+c\mu_j=0\) and
\(\mathsf K-(\lambda-\mu)=2\mu\). After subtracting the scalar gradient \(\nabla\Psi_v\), the
remaining vector operator is

\begin{equation*}
 (L_{\lambda,\mu}(cv))_i-\partial_i\Psi_v
 =c^{-1}(\Delta v_i-m_i\Div v)+Z_{ij}v_j,
\end{equation*}

with
\begin{equation}\label{eq:proof-app-vector-zero-order}
 Z_{ij}=2\mu c_{ij}+\mu_i c_j+\mu_j c_i
                  +(\mu\Delta c+\mu_kc_k)\delta_{ij}.
\end{equation}
Using \(\mu=c^{-2}\) in \eqref{eq:proof-app-vector-zero-order},
we obtain \(Z=c^{-1}V_\mu\), where

\begin{equation*}
 \begin{aligned}
 V_\mu={}&2c^{-1}\nabla^2c+c^{-1}(\Delta c)I_3\\
         &-4c^{-2}\nabla c\otimes\nabla c
              -2c^{-2}|\nabla c|^2I_3.
 \end{aligned}
\end{equation*}

Substitution of

\begin{equation*}
 \begin{gathered}
 \nabla c=-\tfrac12cm,\qquad
 \nabla^2c=c\left(-\tfrac12\nabla^2\ell+\tfrac14m\otimes m\right),\\
 \Delta c=c\left(-\tfrac12\Delta\ell+\tfrac14|m|^2\right)
 \end{gathered}
\end{equation*}

recovers \eqref{eq:Vmu}. The gradient coefficients are those of
\eqref{eq:vector-gradient-convert1}--\eqref{eq:vector-gradient-convert2}.

For the adjoint and phase conventions, the scalar gradient in \eqref{eq:Smu} has bilinear
transpose \(-\Div u/(1-\kappa)\); see \eqref{eq:Stranspose}.
The component identities

\begin{equation*}
 \begin{aligned}
 c^{-1}T_\mu(v,f)&=v+c\nabla f+cfm,\\
 \Div T_\mu(v,f)&=c\Div v-\tfrac12c\,m\cdot v
                         +a(\Delta f+q_\mu f)
 \end{aligned}
\end{equation*}

give \eqref{eq:Kexact}. For \(K=S^TT\), \(K_0=K\), and
\(\varphi=\theta\cdot x\) with constant \(\theta\), we put
\(\sigma\in\{1,-1\}\). The shifts

\begin{equation*}
 \begin{aligned}
 e^{-\sigma\varphi/h}\partial_j e^{\sigma\varphi/h}
   &=\partial_j+\sigma\theta_j/h,\\
 e^{-\sigma\varphi/h}\Delta e^{\sigma\varphi/h}
   &=\Delta+2\sigma D_\theta/h+(\theta\cdot\theta)/h^2
 \end{aligned}
\end{equation*}

give, before imposing nullness,
\begin{equation}\label{eq:proof-app-full-conjugation}
 \begin{aligned}
 e^{-\sigma\varphi/h}Ke^{\sigma\varphi/h}\binom vf
 ={}&K_0\binom vf+\frac{\sigma}{h}J_\theta\binom vf\\
    &-\frac{\theta\cdot\theta}{h^2}
       \begin{pmatrix}0\\ \dfrac{af}{1-\kappa}\end{pmatrix}.
 \end{aligned}
\end{equation}
Equation~\eqref{eq:proof-app-full-conjugation} specifies the complete
\(h^{-2}\) term and its sign. We impose \(\theta\cdot\theta=0\)
in \eqref{eq:proof-app-full-conjugation} and take \(\sigma=\pm1\)
to obtain \eqref{eq:Kconjugated} and \eqref{eq:Kconjugated-minus}.

For a null phase, the first-order term \(J_\theta/h\) changes sign
between \eqref{eq:Kconjugated} and \eqref{eq:Kconjugated-minus}, while
\(K_0\) is unchanged. These are the signs used in the principal normalization
and the unscaled additional lift.

\section{Localized right inverses with unweighted Sobolev estimates}
\label{audrev-ordinary-parametrix}

This appendix constructs the localized right inverse used in
Theorem~\ref{lem:ER-solvability}, with both estimates in
\eqref{eq:system-cgo-right-inverse} provided by the same operator.
The construction uses matrix transport frames and pseudodifferential
estimates following \cite{Eskin2001,EskinRalston2002,EskinRalstonSystems2004}.
The free inverse is first estimated in weighted Sobolev spaces, whereas
the localized right inverse satisfies estimates in unweighted Sobolev
spaces. Here ``localized'' refers to the region where the right-inverse
identity holds, not to preservation of the support of the source.
In both cases the derivative norms are non-semiclassical: spatial
derivatives are not multiplied by powers of the large parameter.

\subsection{The free weighted right inverse}
Put \(D=-\mathrm i\nabla\), and consider
\begin{equation}\label{audrev-system-operator}
 P_\delta=(D+\delta)^2+
             2\sum_{j=1}^3 A_j(x)(D_j+\delta_j)+q(x),\qquad
 \delta=\tau(\omega_0+\mathrm i\nu),
\end{equation}
where \(\tau\geq1\), \(\omega_0,\nu\) are real orthonormal vectors, and
the matrix coefficients are fixed members of \(C_c^\infty\).
No smallness condition on these coefficients is imposed.
The free symbol is
\[
 p_\delta(\xi)=|\xi|^2+2\delta\cdot\xi.
\]
For an integer \(j\geq0\) and a real weight exponent \(\sigma\), use
the weighted norm

\begin{equation}\label{audrev-weighted-Sobolev-norm}
 \|u\|_{H^j_\sigma}^2
   =\sum_{|\alpha|\leq j}
       \|\langle x\rangle^\sigma\partial^\alpha u\|_{L^2(\mathbb R^3)}^2,
 \qquad \langle x\rangle=(1+|x|^2)^{1/2}.
\end{equation}

\begin{lemma}
\label{lem:appendix-free-inverse}
Let \(P_\delta^0=(D+\delta)^2\), with
\(\delta=\tau(\omega_0+i\nu)\), where \(\tau\geq1\) and
\(\omega_0,\nu\) form a real orthonormal pair. The Fourier multiplier
\(E_\delta\) with symbol \(1/p_\delta\), interpreted by its locally
integrable representative, satisfies
\(P_\delta^0 E_\delta f=f\) and, for every integer \(j\geq0\),
\begin{equation}\label{audrev-free-estimates}
 \|E_\delta f\|_{H^j_{-1}}\leq C_j\tau^{-1}\|f\|_{H^j_1},
 \qquad
 \|E_\delta f\|_{H^{j+1}_{-1}}\leq C_j\|f\|_{H^j_1}.
\end{equation}
The constants are uniform under rotations of the orthonormal pair.
\end{lemma}
\begin{proof}
The estimates \eqref{audrev-free-estimates} follow from a one-dimensional
factorization of the free operator.  Rotate so that \(\omega_0=e_1,\nu=e_2\), write \(t=x_2\), and
set \(u=e^{-\mathrm i\tau x_1}v\).
After Fourier transformation in \((x_1,x_3)\), the free operator on
\(v\) is
\[
 L_b=-\partial_t^2+2\tau\partial_t+b^2-\tau^2
     =-(\partial_t-a_+)(\partial_t-a_-),
 \quad b=(\eta_1^2+\eta_3^2)^{1/2},\quad a_\pm=\tau\pm b.
\]
For a nonzero real \(a\), let
\[
 G_a(t)=
 \begin{cases}
  -{\bf1}_{(-\infty,0)}(t)e^{at},&a>0,\\
  {\bf1}_{(0,\infty)}(t)e^{at},&a<0.
 \end{cases}
\]
Then \((\partial_t-a)G_a=\delta_0\).
The convolution kernel \(K_b=-G_{a_+}*G_{a_-}\) of \(L_b^{-1}\)
satisfies
\begin{equation}\label{audrev-green-bounds}
 |K_b|\leq(\tau+b)^{-1},\qquad
 b|K_b|\leq1,\qquad |K_b'|\leq2.
\end{equation}
Indeed \(\|G_{a_+}\|_{L^1}=(\tau+b)^{-1}\),
\(\|G_{a_-}\|_{L^\infty}\leq1\), and
\(K_b'=a_-K_b-G_{a_+}\).
At \(b=\tau\), one may take \(G_0={\bf1}_{(0,\infty)}\);
the set of tangential frequencies $\{b=\tau\}$ has measure zero.
For $\rho_{\mathrm{wt}}>1/2$ put $w_{\rho_{\mathrm{wt}}}(t)=\langle t\rangle^{-\rho_{\mathrm{wt}}}$
and $C_{\rho_{\mathrm{wt}}}=\int_{\mathbb R}w_{\rho_{\mathrm{wt}}}(t)^2\,dt$. The pointwise kernel
bounds \eqref{audrev-green-bounds} yield the following weighted
Hilbert--Schmidt estimates (for almost every tangential frequency $b$):

\begin{equation}\label{audrev-green-HS-estimates}
\begin{aligned}
 \|w_{\rho_{\mathrm{wt}}}(t)K_b(t-s)w_{\rho_{\mathrm{wt}}}(s)\|_{L^2(\mathbb R^2_{t,s})}
    &\leq C_{\rho_{\mathrm{wt}}}(\tau+b)^{-1},\\
 \|w_{\rho_{\mathrm{wt}}}(t)bK_b(t-s)w_{\rho_{\mathrm{wt}}}(s)\|_{L^2(\mathbb R^2_{t,s})}
    &\leq C_{\rho_{\mathrm{wt}}},\\
 \|w_{\rho_{\mathrm{wt}}}(t)K_b'(t-s)w_{\rho_{\mathrm{wt}}}(s)\|_{L^2(\mathbb R^2_{t,s})}
    &\leq 2C_{\rho_{\mathrm{wt}}}.
\end{aligned}
\end{equation}

Partial Plancherel, followed by
\(\langle x\rangle^{-\rho_{\mathrm{wt}}}\leq\langle t\rangle^{-\rho_{\mathrm{wt}}}\) for the resulting function and \(\langle t\rangle^{\rho_{\mathrm{wt}}}\leq\langle x\rangle^{\rho_{\mathrm{wt}}}\)
for the source function, proves the weighted \(L^2\) estimate and its first
derivative version.  The modulation in \(x_1\) costs the additional
term \(\tau K_b\), which is bounded by one.

The direct integral of the kernels \(K_b\) gives the Fourier multiplier
\(1/p_\delta\), with no additional distribution on its characteristic set.
For \(b\ne\tau\), its Fourier transform in \(t\) is
\[
 -\big[(\mathrm i\eta_2-a_+)(\mathrm i\eta_2-a_-)\big]^{-1}
  =\big[\eta_2^2+2\mathrm i\tau\eta_2+b^2-\tau^2\big]^{-1}.
\]
Near \(b=\tau,\eta_2=0\), the real and imaginary parts on the
right have a nonsingular two-dimensional Jacobian.  Its reciprocal
is thus locally integrable.  Fubini on compact frequency sets and
then on Schwartz test functions identifies the direct-integral
kernel with this locally integrable reciprocal.  The value selected
at \(b=\tau\) does not change the direct integral.
Undoing the modulation gives \(1/p_\delta\) and
\(p_\delta E_\delta=I\).
The free-inverse estimates \eqref{audrev-free-estimates} extend from Schwartz functions by density.
Since the original \(E_\delta\) commutes with constant-coefficient
derivatives, the \(L^2\) and first derivative estimates with
\(\rho_{\mathrm{wt}}=1\) imply \eqref{audrev-free-estimates} for every integer \(j\).
\end{proof}

\subsection{Matrix transport frames on fixed disks}
The matrix symbols require transport frames on fixed leaf disks,
with smooth dependence on parameters in local charts. We first construct
a plane frame: if \(B\in C_c^\infty(\mathbb C;M_m(\mathbb C))\), there is
a smooth everywhere invertible matrix \(C\) on \(\mathbb C\) with
\begin{equation}\label{audrev-plane-frame}
 \partial_{\bar z}C=BC.
\end{equation}
The following proof also specifies the Fredholm spaces used in this
assertion.

Let
\[
 \mathcal F f(z)=\frac1\pi\int_{\mathbb C}\frac{f(w)}{z-w}\,dA(w),
 \qquad H=L^2_{-\rho_{\mathrm{wt}}}(\mathbb C;\mathbb C^m),\quad \tfrac12<\rho_{\mathrm{wt}}<1.
\]
The weighted Cauchy-operator bound, with
$H=L^2_{-\rho_{\mathrm{wt}}}(\mathbb C;\mathbb C^m)$, is

\begin{equation}\label{audrev-plane-Cauchy-weighted-estimate}
 \|\mathcal F f\|_{L^2_{-\rho_{\mathrm{wt}}}}
 \leq C_{\rho_{\mathrm{wt}}}\|f\|_{L^2_{\rho_{\mathrm{wt}}}},\qquad\tfrac12<\rho_{\mathrm{wt}}<1.
\end{equation}

To verify this bound, conjugate its absolute kernel by the weights and apply
Schur's test with \(h(z)=\langle z\rangle^{-1}\).
Splitting the integral into
\(|w|\leq|z|/2\), \(|z-w|\leq|z|/2\), and their complement gives
\[
 \int_{\mathbb C}\frac{\langle w\rangle^{-1-\rho_{\mathrm{wt}}}}{|z-w|}\,dA(w)
       \leq C_{\rho_{\mathrm{wt}}}\langle z\rangle^{-\rho_{\mathrm{wt}}}.
\]
The Schur quotient is bounded by
\(C_{\rho_{\mathrm{wt}}}\langle z\rangle^{1-2\rho_{\mathrm{wt}}}\leq C_{\rho_{\mathrm{wt}}}\);
the second Schur quotient obeys the same bound after interchanging $z$ and $w$.
The operator \(\mathcal F B:H\to H\) is compact.
On each fixed disk it maps bounded subsets of \(H\) to bounded
subsets of \(H^1\): use
\(\partial_{\bar z}\mathcal F f=f\) and the \(L^2\) boundedness of
the derivative \(\partial_z\mathcal F\), whose Fourier multiplier
has modulus one.  Rellich compactness applies locally.
Outside a disk containing \(\supp B\),
\[
 |\mathcal F Bu(z)|\leq C|z|^{-1}\|u\|_H.
\]
Consequently its squared \(H\)-norm outside radius \(R\) is at most
\(CR^{-2\rho_{\mathrm{wt}}}\|u\|_H^2\).  This proves compactness including the
tail.  Thus \(I-\mathcal F B\) is Fredholm of index zero.

One must use its weighted Hilbert adjoint for the cokernel.
Write \(\varpi(z)=\langle z\rangle^{-2\rho_{\mathrm{wt}}}\).
If \(h\in\ker(I-\mathcal F B)^*_H\), put
\(\phi=\mathcal F^*(\varpi h)\), where \(\mathcal F^*\) denotes
the unweighted integral adjoint.  The weighted boundedness just
proved makes this definition valid, and the adjoint equation gives
\[
 \varpi h=B^*\phi,\qquad
 \phi=\mathcal F^*(B^*\phi),\qquad
 -\partial_z\phi=B^*\phi.
\]
For a compactly supported source \(f\), the relevant pairing is
\((\mathcal F f,h)_H=(f,\phi)_{L^2}\).
These equations distinguish the weighted adjoint from the
unweighted formal differential adjoint.

The adjoint equation \(-\partial_z\phi=B^*\phi\) has unique
continuation. We prove this locally by invertible frames, using the
same construction for \eqref{audrev-plane-frame}.
On a sufficiently small disk, cut off \(B\) in a disk of twice the
radius.  Its Cauchy operator has norm at most
\(Cr\|B\|_\infty\) on bounded continuous vector functions.
The equation \(d=I+\mathcal F(B_{\rm cut}d)\) has a solution uniformly
close to \(I\), hence invertible, when the radius \(r\) is small.
Local Cauchy regularity makes \(d\) smooth.  For any local solution
\(u\) of \(\partial_{\bar z}u=Bu\), the vector \(d^{-1}u\) is
holomorphic.  The identity theorem, continued through overlapping
small disks, proves unique continuation.  Complex conjugation
reduces the adjoint equation to the same argument.

Take a basis \(h_1,\ldots,h_r\) of the cokernel, with corresponding
\(\phi_1,\ldots,\phi_r\).
The functionals on vector polynomials
\[
 \ell_j(P)=(BP,\phi_j)_{L^2}
\]
are linearly independent.  Otherwise a nonzero combination
\(\phi\) has every exterior Laurent coefficient of
\(\mathcal F^*(B^*\phi)\) equal to zero, since those coefficients
are the moments of the compact source \(B^*\phi\).
It vanishes outside a large disk, hence everywhere by unique
continuation; then \(\varpi h=B^*\phi\) contradicts independence
of the \(h_j\).
Select finitely many vector polynomials of degrees at most \(N\)
on which these functionals have full row rank.
For each standard vector \(e_i\), choose a polynomial
\[
 P_i(z)=z^{N+1}e_i+\text{a vector polynomial of degree at most }N
\]
such that all \(\ell_j(P_i)\) vanish.
The Fredholm alternative gives \(v_i\in H\) with
\((I-\mathcal F B)v_i=\mathcal F(BP_i)\).
It follows that \(P_i+v_i\) solves \eqref{audrev-plane-frame},
is smooth, and equals \(P_i+O(|z|^{-1})\) outside a large disk.
The matrix with these columns has determinant not identically zero.
Indeed its leading term at infinity is \(z^{m(N+1)}\).
Jacobi's formula shows that
\[
 \det C=\exp(\mathcal F\,\operatorname{tr}B)\,F(z),
\]
where \(F\) is an entire polynomial, so there are only finitely many
zeros.  Near any zero use the local invertible frame \(d\) above to
write \(C=dH\) with \(H\) holomorphic.  A constant invertible change
of columns makes one column of \(H\) vanish at that point.
Dividing that column by \(z-z_0\) gives another global smooth
solution, removes one determinant zero counted with multiplicity,
and creates no other zero.  After finitely many divisions the
matrix is invertible everywhere.  This proves
\eqref{audrev-plane-frame}.

Next fix \(L<M\) and a compact set of parameters \(y\) for smooth
coefficients \(B_y\) on a neighborhood of \(\overline{D_M}\).
Extend the coefficient at a fixed \(y_0\) smoothly to a compactly
supported coefficient on the plane and choose the frame \(C_0\)
just constructed.  For \(p>2\), define on \(D_M\)
\[
 S_0 f=C_0\,\mathcal F\big({\bf1}_{D_M}C_0^{-1}f\big).
\]
It maps \(L^p(D_M)\) to \(W^{1,p}(D_M)\) and satisfies
\((\partial_{\bar z}-B_{y_0})S_0 f=f\).
For \(y\) sufficiently close to \(y_0\), the equation
\begin{equation}\label{audrev-parameter-frame}
 C_y=C_0+S_0\big((B_y-B_{y_0})C_y\big)
\end{equation}
is a contraction in \(W^{1,p}(D_M)\).
Sobolev embedding and the positive minimum singular value of
\(C_0\) on \(\overline{D_L}\) make \(C_y\) invertible there.
The bounded inverse of $I-S_0(B_y-B_{y_0})$ on $W^{1,p}(D_M)$ depends smoothly on $y$;
interior elliptic regularity gives joint smoothness and bounds
for every fixed finite collection of derivatives of \(C_y\) and
\(C_y^{-1}\) on \(\overline{D_L}\).
Only the difference of two nearby parameter coefficients has
been made small.  The original coefficient can be arbitrarily
large.
This supplies an independent analytic proof of the
parameter-local assertion of Lemma~\ref{lem:matrix-dbar-frame}.
There is no need to introduce a second theorem statement with the same
hypotheses; the frames below are the local choices furnished by
that lemma and the present construction.
For trace-free \(B_y\), multiplication of the first column by
\((\det C_y)^{-1}\) normalizes the determinant to one.
The determinant is nonvanishing and holomorphic in \(z\), so this
operation preserves the equation and all the stated bounds on a
compact parameter subchart.

\subsection{Paired symbols and the localized parametrix}
Choose fixed cutoffs
\(\psi_2,\psi_3,\chi_{\mathrm{sp}}\in C_c^\infty(\Omega_0)\), equal to one
successively on neighborhoods of
\(\overline\Omega_1,\supp\psi_2,\supp\psi_3\).
They may be chosen radial.  Fix a radius \(R_*\) larger than that
of \(\Omega_0\), and \(R_*<L<M\).
Apply the parameter construction \eqref{audrev-parameter-frame} to
\[
 x=z_1\omega+z_2\nu+x_\perp\nu_\perp,\quad
 \nu_\perp=\omega\times\nu,\quad
 B_{(\omega,\nu,x_\perp)}(z)=-\frac{\mathrm i}{2}A_{\omega+\mathrm i\nu}(x),
 \quad |\!x_\perp\!|\leq R_*.
\]
The parameter set of ordered orthonormal pairs and
\(|x_\perp|\leq R_*\) is compact.
It has a finite cover \(U_a\) with smooth frames \(C_a\) on the
whole fixed disk \(D_L\) satisfying
\begin{equation}\label{audrev-frame-transport}
 \mathrm i(\omega+\mathrm i\nu)\cdot\nabla C_a
          =A_{\omega+\mathrm i\nu}C_a.
\end{equation}
Choose real smooth weights \(w_a(\omega,\nu,x_\perp)\), supported in their
respective parameter charts, with \(\sum_a w_a^2=1\) on this
compact parameter set.  Shrinking the charts used for the weights
ensures uniform bounds for all fixed derivatives of the frames
and their inverses there.  For a fixed pair, these weights are
constant along its complex leaves:
\((\omega+\mathrm i\nu)\cdot\nabla_xw_a=0\).
Each frame is paired with its own inverse; there is no need for the
matrix average \(\sum_a w_a C_a\) to be invertible.

For the frequency-dependent pair write
\[
 \xi_\nu=\xi\cdot\nu,\quad
 \xi'=\xi-\xi_\nu\nu,\quad
 r=|\xi'+\tau\omega_0|,\quad
 \omega=\frac{\xi'+\tau\omega_0}{r}.
\]
Choose a real smooth function \(\varrho\), equal to one on
\([-1,1]\), zero outside \([-2,2]\), and between zero and one.
Set

\begin{equation}\label{audrev-frequency-cutoffs}
\begin{aligned}
 t_\tau(\xi)&=
   \varrho(\xi_\nu/\sqrt\tau)\,
   \varrho((r^2-\tau^2)/\tau^{3/2}),\\
 \chi_N(\xi)&=\sin(\pi t_\tau(\xi)/2),&
 \chi_F(\xi)&=\cos(\pi t_\tau(\xi)/2).
\end{aligned}
\end{equation}

For sufficiently large \(\tau\), \(r\asymp\tau\) on the support
of \(\chi_N\), and \(\omega\) is defined there.
Where \(r=0\), that cutoff vanishes on a neighborhood; products
containing it are defined by zero.
For \(a\geq1\), define the symbols
\begin{equation}\label{audrev-paired-symbols}
 \begin{aligned}
 b_a(x,\xi)&=\chi_{\mathrm{sp}}(x)\chi_N(\xi)w_a(\omega,\nu,x_\perp)C_a(z;\omega,\nu,x_\perp),\\
 d_a(x,\xi)&=\chi_{\mathrm{sp}}(x)\chi_N(\xi)w_a(\omega,\nu,x_\perp)C_a^{-1}(z;\omega,\nu,x_\perp).
 \end{aligned}
\end{equation}
Here \(x_\perp=x\cdot\nu_\perp\), \(z_1=x\cdot\omega,z_2=x\cdot\nu\);
put \(b_0=d_0=\chi_F I_m\).
Terms supported in individual charts extend smoothly by zero.
All disks and parameter subcharts used here are fixed independently
of \(\tau\).
On \(\supp\psi_3\),
\begin{equation}\label{audrev-symbol-product}
 \sum_{a\geq0}b_a d_a=I_m.
\end{equation}
For every fixed \(\alpha,\beta\), these symbols satisfy
\begin{equation}\label{audrev-symbol-seminorms}
 |\partial_x^\alpha\partial_\xi^\beta b_a|
 +|\partial_x^\alpha\partial_\xi^\beta d_a|
       \leq C_{\alpha\beta}\tau^{-|\beta|/2}.
\end{equation}
Indeed derivatives of the cutoffs cost at most \(\tau^{-1/2}\);
derivatives of the direction and leaf coordinates cost
\(\tau^{-1}\) on their fixed supports.  In the region
\(\chi_{\mathrm{sp}}=1\), \eqref{audrev-frame-transport} and the constancy
of \(w_a\) along the leaf give

\begin{equation}\label{audrev-symbol-transport-cancellation}
 (\omega+\mathrm i\nu)\cdot(A b_a-\mathrm i\nabla b_a)=0,\qquad a\geq1.
\end{equation}

The cancellation identity \eqref{audrev-symbol-transport-cancellation} eliminates the component of the first-order
symbol along \(\omega+\mathrm i\nu\). The weights must therefore be
constant along each leaf; a partition varying within the leaf would
add a leading derivative term.

\begin{lemma}
\label{lem:appendix-local-error}
Fix the nested domains, compact coefficient supports, spatial cutoffs,
and a compact family of normalized directions as above. For every
fixed integer \(j\geq0\), the operator \(Q_\tau\) constructed from
the paired symbols in \eqref{audrev-parametrix} satisfies
\[
 P_\delta Q_\tau=\psi_3I_m+S_\tau+\mathcal E_\tau,
 \qquad
 \|\psi_3(S_\tau+\mathcal E_\tau)\|_{H^j\to H^j}
 \leq C_j\tau^{-1/2},
\]
and
\[
 \|\psi_2Q_\tau g\|_{H^j}\leq C_j\tau^{-1}\|g\|_{H^j},\qquad
 \|\psi_2Q_\tau g\|_{H^{j+1}}\leq C_j\|g\|_{H^j}.
\]
All bounds are uniform on the indicated compact family of directions.
The smallness assertion concerns the \emph{localized} error; no
smallness of \(P_\delta Q_\tau-I\) on the whole space is claimed.
\end{lemma}
\begin{proof}
Use left, or Kohn--Nirenberg, quantization
\[
 \operatorname{Op}(a)u(x)
  =(2\pi)^{-3}\int e^{\mathrm ix\cdot\xi}a(x,\xi)\widehat u(\xi)\,d\xi.
\]
The operators and their composition remainders are estimated on
\(H^j\) using their full Fourier kernels, including spatial tails.
For a symbol with fixed compact \(x\)-support and bounded
derivatives in \(x\), its partial Fourier transform satisfies
\[
 |\widehat a(\eta-\xi,\xi)|\leq C_N\langle\eta-\xi\rangle^{-N}.
\]
Schur's test with
\(\langle\eta\rangle^j/\langle\xi\rangle^j
 \leq C_j\langle\eta-\xi\rangle^j\)
therefore proves boundedness on \(H^j\).
More precisely, for compactly supported $x$-symbols with
uniform bounds on the spatial derivatives through sufficiently high order,

\begin{equation}\label{audrev-unweighted-pdo-bound}
 \|\operatorname{Op}(a)u\|_{H^j}
 \leq C_j\|u\|_{H^j}.
\end{equation}

Here $C_j$ depends on the fixed spatial support and finitely
many uniform symbol seminorms.
If a compactly supported symbol \(c\) is composed on the right
with such an \(a\), the difference from the product symbol has
Fourier kernel
\[
 \int
  \big[\widehat a(\eta-\zeta,\zeta)
       -\widehat a(\eta-\zeta,\xi)\big]
       \widehat c(\zeta-\xi,\xi)\,d\zeta.
\]
If the compact-support symbol seminorms of $\partial_\xi a$ used in \eqref{audrev-unweighted-pdo-bound} are
\(O(\varepsilon)\), the mean value theorem bounds this kernel by
\[
 C_N\varepsilon\int
 |\zeta-\xi|\langle\eta-\zeta\rangle^{-N}
              \langle\zeta-\xi\rangle^{-N}\,d\zeta.
\]
Choosing \(N\) large and applying Schur again gives an
\(O(\varepsilon)\) norm on \(H^j\).
Under the stated uniform symbol assumptions on $a,c$ and
the $O(\varepsilon)$ compact-support seminorm bounds for $\partial_\xi a$,
this reads

\begin{equation}\label{audrev-unweighted-composition-bound}
 \|\operatorname{Op}(a)\operatorname{Op}(c)
       -\operatorname{Op}(ac)\|_{H^j\to H^j}
 \leq C_j\varepsilon.
\end{equation}

When \(a\) is a Fourier multiplier the kernel is instead
\([a(\eta)-a(\xi)]\widehat c(\eta-\xi,\xi)\), with the same bound.
The composition estimate applies to
\(\operatorname{Op}(b_a)\psi_3\operatorname{Op}(d_a)\), since its right
symbol \(\psi_3d_a\) has compact \(x\)-support. The Fourier-kernel
bounds include the operators' spatial tails.

An additional estimate is needed from a weighted Sobolev space for a symbol with compact spatial support.
For \(a\) with fixed compact \(x\)-support and bounded \(x\)
derivatives of \(a,\partial_\xi a,\partial_\xi^2a\), the exact
identity on Schwartz functions is
\begin{equation}\label{audrev-weighted-pdo-identity}
 \operatorname{Op}(a)\langle x\rangle^2
 =\langle x\rangle^2\operatorname{Op}(a)
   -2\mathrm i\sum_j x_j\operatorname{Op}(\partial_{\xi_j}a)
   -\operatorname{Op}(\Delta_\xi a).
\end{equation}
The unweighted \(H^j\) bound from Schur's test, applied to each
symbol with compact spatial support in \eqref{audrev-weighted-pdo-identity}, proves
\begin{equation}\label{audrev-weighted-pdo}
 \|\operatorname{Op}(a)u\|_{H^j}
       \leq C_j\|u\|_{H^j_{-1}}.
\end{equation}
Here set \(v=\langle x\rangle^{-2}u\), for which
\(\|v\|_{H^j}\leq C_j\|u\|_{H^j_{-1}}\), and apply
\eqref{audrev-weighted-pdo-identity} to \(v\).
Smooth compactly supported functions are dense in the weighted Sobolev space defined in \eqref{audrev-weighted-Sobolev-norm}; the estimate defines the unique
bounded extension and preserves the distributional identity.
Thus two frequency derivatives, together with finitely many
spatial derivatives, suffice for the weighted passage.
This argument does not discard a nonlocal tail.

Taking \(\varepsilon=\tau^{-1/2}\) in the composition estimate and
using \eqref{audrev-symbol-product} gives
\begin{equation}\label{audrev-composition-error}
 S_\tau=\sum_{a\geq0}\operatorname{Op}(b_a)\,
              \psi_3\operatorname{Op}(d_a)-\psi_3 I_m,
 \qquad \|S_\tau\|_{H^j\to H^j}\leq C_j\tau^{-1/2}.
\end{equation}
The subtracted multiplication operator is \(\psi_3 I_m\).
In particular no identity assertion is made in the transition
region of the source cutoff.

Define
\begin{equation}\label{audrev-parametrix}
 Q_\tau g=\sum_{a\geq0}
       \operatorname{Op}(b_a)E_\delta
                 \psi_3\operatorname{Op}(d_a)g.
\end{equation}
For Schwartz functions the exact differential calculation is

\begin{equation}\label{audrev-pdo-differential-error-decomposition}
\begin{aligned}
 P_\delta\operatorname{Op}(b_a)E_\delta
   &=\operatorname{Op}(b_a)
        +\operatorname{Op}(r_{0,a})E_\delta
        +\operatorname{Op}(r_{1,a})E_\delta,\\
 r_{0,a}&=-\Delta_x b_a-2\mathrm i A\cdot\nabla_xb_a+q b_a,\\
 r_{1,a}&=2(\xi+\delta)\cdot(A b_a-\mathrm i\nabla_xb_a).
\end{aligned}
\end{equation}

All identities extend to the functions in the Sobolev spaces used below by the
bounds that follow.
Only the error localized by \(\psi_3\) is needed.  In this region
\(\chi_{\mathrm{sp}}=1\), and \(\psi_3r_{0,a}\) has bounded seminorms and compact spatial support.  Equations
\eqref{audrev-free-estimates} and \eqref{audrev-weighted-pdo}
give an \(O(\tau^{-1})\) bound on \(H^j\) for its contribution to
\(\psi_3P_\delta Q_\tau\).

For $a\geq1$, the symbol cancellation
\eqref{audrev-symbol-transport-cancellation}, derived from
\eqref{audrev-frame-transport}, removes the component
$\tau(\omega+\mathrm i\nu)$ from the first-order error. The remaining vector is
\[
 \xi+\delta-\tau(\omega+\mathrm i\nu)
       =(r-\tau)\omega+\xi_\nu\nu,
\]
whose length is \(O(\sqrt\tau)\) on the support of the near
cutoff.  More precisely,

\begin{equation}\label{audrev-near-error-symbol-bound}
 |\partial_x^\alpha\partial_\xi^\beta(\psi_3r_{1,a})|
       \leq C_{\alpha\beta}\tau^{(1-|\beta|)/2}.
\end{equation}

By \eqref{audrev-near-error-symbol-bound},
\(\tau^{-1/2}\psi_3r_{1,a}\) has the seminorms needed in
\eqref{audrev-weighted-pdo}.
Apply \(E_\delta\) first to the compactly supported source
\(\psi_3\operatorname{Op}(d_a)g\), use its \(H^j_1\)-norm and
\eqref{audrev-free-estimates}, and then apply \(\operatorname{Op}(\psi_3r_{1,a})\).
The resulting contribution is \(O(\tau^{-1/2})\) on \(H^j\).
No derivative of \(1/p_\delta\) has been taken near its
characteristic set.

For \(a=0\), the first-order error is
\(2A\cdot(\xi+\delta)\chi_F\).
On the support of $\chi_F$, the free symbol satisfies

\begin{equation}\label{audrev-far-symbol-ellipticity}
 |p_\delta(\xi)|\geq
 \begin{cases}
 c\tau^{3/2},&|\xi|\leq4\tau,\\
 c|\xi|^2,&|\xi|\geq4\tau.
 \end{cases}
\end{equation}

Indeed either \(|\xi_\nu|\geq\sqrt\tau\), giving the first bound from
the imaginary part, or
\(|r^2-\tau^2|\geq\tau^{3/2}\) and \(\xi_\nu^2\leq\tau\),
giving it from the real part for large \(\tau\).
The second bound follows from $|\xi|\geq4\tau$ and
$|p_\delta(\xi)-|\xi|^2|\leq 2|\delta||\xi|$.
The ratio
\[
 a_F(x,\xi)=\frac{2A(x)\cdot(\xi+\delta)\chi_F(\xi)}
                         {p_\delta(\xi)}
\]
is defined by zero near the characteristic set.  In the first
region, differentiating the reciprocal costs at most
\(\tau^{-1/2}\) per derivative; in the second it gives the usual
\(|\xi|^{-1}\) decay per derivative.  Consequently

\begin{equation}\label{audrev-far-error-symbol-bound}
 |\partial_x^\alpha\partial_\xi^\beta a_F|
        \leq C_{\alpha\beta}\tau^{-1/2-|\beta|/2}.
\end{equation}

The unweighted estimate for symbols with compact spatial support
in \eqref{audrev-unweighted-pdo-bound}, applied to the symbol
estimate \eqref{audrev-far-error-symbol-bound}, shows that
$\operatorname{Op}(a_F)$ has operator norm $O(\tau^{-1/2})$ on $H^j$.
Here the identity between the differential-error operator and
\(\operatorname{Op}(a_F)\) is legitimate because \(E_\delta\)
is a Fourier multiplier and the numerator vanishes on a
neighborhood of its characteristic set.
Only the far-frequency symbol $a_F$, bounded in \eqref{audrev-far-error-symbol-bound}, requires derivatives of $1/p_\delta$.

It follows that
\begin{equation}\label{audrev-local-error}
 P_\delta Q_\tau=\psi_3 I_m+S_\tau+\mathcal E_\tau,\qquad
 \|\psi_3\mathcal E_\tau\|_{H^j\to H^j}
           \leq C_j\tau^{-1/2}.
\end{equation}
The equality is distributional; the estimate concerns the
localized error displayed, not an error at the remote spatial
cutoff.
For the solution itself, apply \eqref{audrev-weighted-pdo} to
\(\psi_2b_a\), including the multiplier term \(a=0\).
The compactly supported sources
\(\psi_3\operatorname{Op}(d_a)g\) have
\(H^j_1\)-norm bounded by \(C_j\|g\|_{H^j}\).
The two bounds in \eqref{audrev-free-estimates} therefore yield
\begin{equation}\label{audrev-Q-estimates}
 \|\psi_2Q_\tau g\|_{H^j}\leq C_j\tau^{-1}\|g\|_{H^j},
 \qquad
 \|\psi_2Q_\tau g\|_{H^{j+1}}\leq C_j\|g\|_{H^j}.
\end{equation}
In the second estimate the first operator is bounded from
\(H^{j+1}_{-1}\) to \(H^{j+1}\); no \(H^{j+1}\) norm of the
source or of \(g\) is required.
\end{proof}

\begin{proposition}
\label{prop:appendix-interior-right-inverse}
Let \(m\geq1\), \(A_j,q\in C_c^\infty(\mathbb R^3;M_m(\mathbb C))\),
and fix concentric balls
\(\overline\Omega_{\mathrm{CGO}}\Subset\Omega_1\Subset\Omega_0\)
and a closed source ball \(K_{\mathrm{src}}\Subset\Omega_1\)
whose interior contains the coefficient supports and
\(\overline\Omega_{\mathrm{CGO}}\). Consider \(P_\delta\) from
\eqref{audrev-system-operator} with
\(\delta=\tau(\omega_0+i\nu)\) and a compact family of real
orthonormal pairs \((\omega_0,\nu)\). For each fixed integer
\(k\geq0\), all sufficiently large \(\tau\), and every
\(f\in H^k(\Omega_0;\mathbb C^m)\) supported in \(K_{\mathrm{src}}\),
there is a \emph{single linear} operator \(\mathcal R_\delta\) with
\begin{equation}\label{eq:appendix-two-order-interior}
 P_\delta\mathcal R_\delta f=f\quad\text{in }\Omega_1,\qquad
 \tau\|\mathcal R_\delta f\|_{H^k(\Omega_0)}
 +\|\mathcal R_\delta f\|_{H^{k+1}(\Omega_0)}
 \leq C_k\|f\|_{H^k(\Omega_0)}.
\end{equation}
The constants are uniform for a compact set of normalized null
directions. No boundary condition on \(\partial\Omega_0\) is imposed.
\end{proposition}
\begin{proof}
By Lemma~\ref{lem:appendix-local-error}, for the prescribed
fixed integer \(k\), set
\[
 T_\tau=\psi_3(S_\tau+\mathcal E_\tau):H^k\longrightarrow H^k.
\]
Its values have compact support, and its operator norm is at most
\(C_k\tau^{-1/2}\).
For large \(\tau\), set
\[
 g=(I+T_\tau)^{-1}f,\qquad
 \mathcal R_\delta f=\psi_2Q_\tau g.
\]
A source supported in \(K_{\mathrm{src}}\Subset\Omega_1\) is extended by zero
from \(\Omega_0\); this is a bounded operation at every integer
Sobolev order under consideration, since it vanishes in a
neighborhood of the boundary.  On \(\Omega_1\) both \(\psi_2\)
and \(\psi_3\) equal one.  Equation
\eqref{audrev-local-error} and the equation for \(g\) give

\begin{equation}\label{audrev-final-interior-right-inverse}
\begin{gathered}
 P_\delta\mathcal R_\delta f=f\quad\hbox{in }\Omega_1,\\
 \tau\|\mathcal R_\delta f\|_{H^k(\Omega_0)}
       +\|\mathcal R_\delta f\|_{H^{k+1}(\Omega_0)}
          \leq C_k\|f\|_{H^k(\Omega_0)}.
\end{gathered}
\end{equation}

Both bounds apply to \(\mathcal R_\delta\), constructed with the single
Neumann inverse \((I+T_\tau)^{-1}\) on \(H^k\).
The commutator with the final spatial cutoff vanishes on
\(\Omega_1\); it need not be small on the transition region.
The operator $\mathcal R_\delta$ is a localized right inverse
by \eqref{audrev-final-interior-right-inverse}; no boundary
condition is imposed on \(\partial\Omega_0\).
The finite parameter cover used above gives uniform constants
for a compact family of directions.  It does not assert
parameter derivatives of the exact right inverse.

Finally take
\[
 A_j=\frac{\widetilde A_j}{2\mathrm i},\qquad
 q=-\widetilde B,\qquad
 \delta=\mathrm i\theta/h,\qquad \tau=h^{-1}.
\]
Then \(P_\delta=-e^{\theta\cdot x/h}\widetilde M
e^{-\theta\cdot x/h}\).
Applying the localized right-inverse identity and the two
unweighted Sobolev estimates \eqref{audrev-final-interior-right-inverse}
to the source $-f$, with $\tau=h^{-1}$, gives
\eqref{eq:system-cgo-right-inverse-equation} and
\eqref{eq:system-cgo-right-inverse}.
For \(k=5\) this is an ordinary \(H^5\) estimate and an
ordinary \(H^6\) estimate, rather than a semiclassical estimate.
\end{proof}

\subsection{Transport corrections from local frames}
\label{audrev-finite-transport}
The finite family of transport frames satisfying
\eqref{audrev-frame-transport} also gives the leafwise inverse used in
Theorem~\ref{thm:autonomous-cgo}.
Fix \(\theta\), put
\(T_\theta=2D_\theta+\widetilde A_\theta\), and take a fixed
large leaf disk containing \(\overline\Omega_0\).
The parameter-local frames constructed from
\eqref{audrev-plane-frame} and \eqref{audrev-parameter-frame}, or
alternatively Lemma~\ref{lem:matrix-dbar-frame}, give finitely many
matrices \(C_a(z,x_\perp)\) with \(T_\theta C_a=0\) on this whole
disk and transverse intervals covering the required bounded
\(x_\perp\)-range. Choose weights \(w_a(x_\perp)\) subordinate to
these intervals with \(\sum_a w_a(x_\perp)^2=1\).
The transverse intervals cover the entire bounded \(x_\perp\)-range
of the target region, not just one selected leaf.
Choose a leaf cutoff \(\chi_{\rm leaf}\) supported inside the
large disk and equal to one on a neighborhood of the entire
planar projection of the target region.
For a smooth source \(F\) defined on a neighborhood of this
cutoff cylinder, set
\begin{equation}\label{audrev-transport-right-inverse}
 \mathcal V_\theta F
   =\frac12\sum_a w_a(x_\perp)C_a(z,x_\perp)\,
       \mathcal C_\theta
       \big(\chi_{\rm leaf}C_a^{-1}w_a(x_\perp)F\big)(z,x_\perp).
\end{equation}
Here \(\mathcal C_\theta\) is the scalar Cauchy operator already
defined in the text, with \(D_\theta\mathcal C_\theta=I\);
it integrates only in the leaf variable.
Each summand extends by zero outside its transverse parameter
chart.  As \(D_\theta w_a=0\),

\begin{equation}\label{audrev-transport-inverse-identity}
 T_\theta\mathcal V_\theta F=\chi_{\rm leaf}F.
\end{equation}

Because $\chi_{\rm leaf}=1$ on the target region, \eqref{audrev-transport-inverse-identity} makes $\mathcal V_\theta$ a transport right inverse there.  The source need not have compact support
in the leaf variable: the product \(\chi_{\rm leaf}F\) is the
compact leaf source used in the Cauchy integral.
Every fixed finite spatial and local
direction derivative is bounded by the corresponding source
seminorms, the fixed disk sizes, and the finitely many frame
seminorms.  No global smooth choice of an invertible matrix over
the entire parameter set is necessary.
For the first recursion one uses
\(F=\chi P_0a_0\), where \(\chi\) is compactly supported in
the domain of \(a_0\) and is one on the region where that
recursion is needed; this product is extended by zero.
For the transverse localization of \(\boldsymbol z_0^S\), the source is
\([M,\chi]\boldsymbol z_0^S\) from \eqref{eq:proof-slow-localization-source},
and the inverse is \eqref{audrev-transport-right-inverse}.
If the source vanishes on an open transverse interval,
\eqref{audrev-transport-right-inverse} vanishes on that interval,
together with all its spatial derivatives.
For a prescribed first amplitude defined only near the working
CGO ball, the construction of \(\mathcal V_\theta\) is bypassed:
the residual cutoff is supported in that neighborhood, as in
the proof of Theorem~\ref{thm:autonomous-cgo}.

The vanishing on that interval concerns the leafwise correction
\(\mathcal V_\theta F\) in \eqref{audrev-transport-right-inverse}.
The exact elliptic correction has the separate Sobolev bounds
\[
 r_h=-R_{h,\theta}^{(5)}(h\chi_0P_0a_1),\qquad
 \|r_h\|_{H^5}=O(h^2),\quad \|r_h\|_{H^6}=O(h),
\]
when \(a_1\) has the required \(H^7\) bound on the fixed source
region.  A compact source cutoff does not imply that this exact
correction vanishes in the inner ball.  The leading and first
transport amplitudes must satisfy their equations throughout
the region where the exact CGO solution is asserted, as arranged
by the large-disk constructions in the text.


\section{Boundary determination to arbitrary order}
\label{audrev-app-boundary-jets}

The full DN map determines the boundary-normal derivatives of both
Lam\'e parameters at every order, with the outward traction convention
\eqref{eq:DN}. Theorem~\ref{thm:appendix-boundary-jets} provides the
boundary determination used in Lemma~\ref{lem:common-extension}. The argument uses the classical
smooth elliptic Dirichlet--Poisson calculus, Borel summation, and
Dirichlet regularity; the elasticity-specific stable graph and
matrices governing the recovery of normal derivatives are computed below. Related boundary
determination results are discussed in
\cite{NakamuraUhlmann1995,LinNakamura2017}.

\begin{theorem}
\label{thm:appendix-boundary-jets}
Let \(\Omega\subset\mathbb R^3\) be a bounded smooth domain. Suppose that
\((\lambda_j,\mu_j)\in C^\infty(\overline\Omega;\mathbb R^2)\),
\(j=1,2\), satisfy the strong convexity assumptions
\eqref{eq:convexity}, and that their full elastic Dirichlet-to-Neumann
maps, with the outward traction convention \eqref{eq:DN}, coincide.
Then in every inward boundary-normal chart \((y,r_{\mathrm{in}})\)
and for every fixed integer \(j\geq0\),
\begin{equation}\label{audrev-boundary-all-jets}
 \partial_{r_{\mathrm{in}}}^j\lambda_1(y,0)
   =\partial_{r_{\mathrm{in}}}^j\lambda_2(y,0),\qquad
 \partial_{r_{\mathrm{in}}}^j\mu_1(y,0)
   =\partial_{r_{\mathrm{in}}}^j\mu_2(y,0).
\end{equation}
These identities hold on every boundary component, and tangential
differentiation gives equality of all mixed boundary derivatives.
\end{theorem}

Here \(r_{\mathrm{in}}\geq0\) is the known distance into the domain,
so that the outward normal is \(-\nabla r_{\mathrm{in}}\) at the
boundary. The arguments apply at each fixed derivative order; no
uniform estimate as \(j\to\infty\) is required.

Use tangential Fourier modes \(e^{iky_1}\), \(k>0\), so that
\(\partial_{y_1}\) acts by \(ik\). Displacements are expressed in
fixed Cartesian components. A local orthonormal rotation, used only
to evaluate the frozen symbol at a point, sends the inward normal
to \(e_3\) and the unit tangential covector to \(e_1\).

\subsection{The physical flux and stable normal modes}
First take a flat boundary with coefficients depending on \(r_{\mathrm{in}}\).
In this layered model a prime denotes \(\partial_{r_{\mathrm{in}}}\).
Set
\[
 q=\lambda+2\mu,\qquad \ell_*=\lambda+3\mu,
 \qquad a=\frac{\lambda+\mu}{\ell_*},
\]
\[
 A=\operatorname{diag}(\mu,\mu,q),\qquad
 R=\begin{pmatrix}0&0&\mu\\0&0&0\\\lambda&0&0\end{pmatrix},
 \qquad H_{\mathrm{tan}}=\operatorname{diag}(q,\mu,\mu).
\]
For \(u=e^{iky_1}U(r_{\mathrm{in}})\), its inward traction is
\(v=AU'+ikRU\). The tangential stress column is
\(\sigma(u)e_1=e^{iky_1}(ikH_{\mathrm{tan}}U+R^TU')\); equilibrium therefore gives
\begin{equation}\label{audrev-eq-layered-flux}
 U'=A^{-1}(v-ikRU),\qquad
 v'=k^2H_{\mathrm{tan}}U-ikR^TU'.
\end{equation}
Equation~\eqref{audrev-eq-layered-flux} retains the variable physical
flux. The derivatives of \(\lambda\) and \(\mu\) that arise in
\(\Div\sigma(u)\) are already included in \(v'\); no separate
coefficient-derivative terms should be added to the flux system \eqref{audrev-eq-layered-flux}.

Write the stable traction graph as \(v=-\mathfrak Z(r_{\mathrm{in}},k)U\). Since the
outward traction is \(-v\), \(\mathfrak Z\) has the physical DN sign.
Equation~\eqref{audrev-eq-layered-flux} gives the exact Riccati equation
\begin{equation}\label{audrev-eq-layered-riccati}
 \mathfrak Z'=(\mathfrak Z-ikR^T)A^{-1}(\mathfrak Z+ikR)-k^2H_{\mathrm{tan}}.
\end{equation}
Its leading solution \(\mathfrak Z\sim k\mathsf Z_0\) is
\begin{equation}\label{audrev-eq-boundary-principal}
 \mathsf Z_0=\begin{pmatrix}d&0&-ib_*\\0&\mu&0\\ib_*&0&d\end{pmatrix},
 \qquad d=\frac{2\mu q}{\ell_*},\qquad
 b_*=\frac{2\mu^2}{\ell_*},
 \qquad (\mathsf Z_0-iR^T)A^{-1}(\mathsf Z_0+iR)=H_{\mathrm{tan}}.
\end{equation}
Indeed, direct multiplication verifies the last identity. Strong
convexity implies \(q>0\), \(\ell_*>0\), and \(\lambda+\mu>0\).
The eigenvalues of \(\mathsf Z_0\) are
\(\mu\), \(d+b_*=2\mu\), and
\(d-b_*=2\mu(\lambda+\mu)/\ell_*\), all positive. In particular a
transverse mode \(e^{iky_1-k r_{\mathrm{in}}}e_2\) has outward traction
\(+\mu k e^{iky_1}e_2\), fixing the overall sign.

The stable displacement generator is
\begin{equation}\label{audrev-eq-boundary-jordan}
 K_{\mathrm{st}}=A^{-1}(\mathsf Z_0+iR)=I_3+N,\qquad
 N=\begin{pmatrix}a&0&ia\\0&0&0\\ia&0&-a\end{pmatrix},
 \qquad N^2=0.
\end{equation}
\begin{lemma}
\label{lem:appendix-stable-graph}
At a boundary point, freeze the strongly convex Lam\'e coefficients
and fix a nonzero tangential covector of length \(k\). With the
matrices \(\mathsf Z_0\) and \(N\) in
\eqref{audrev-eq-boundary-principal}--\eqref{audrev-eq-boundary-jordan},
the three-dimensional decaying displacement space is represented by
\[
 U_f(r_{\mathrm{in}})=e^{-kr_{\mathrm{in}}}
       (I_3-kr_{\mathrm{in}}N)f,\qquad f\in\mathbb C^3.
\]
Its Dirichlet boundary projection is an isomorphism. In particular,
the frozen Dirichlet problem satisfies the complementing condition,
even though its stable characteristic root is repeated.
\end{lemma}
\begin{proof}
The flux system \eqref{audrev-eq-layered-flux} and the algebraic
identity in \eqref{audrev-eq-boundary-principal} give the stable
normal evolution generator \(-kK_{\mathrm{st}}=-k(I_3+N)\).
Since \(N^2=0\), its exponential is exactly the displayed
\(U_f\). It decays for every \(f\), while \(U_f(0)=f\).
The frozen normal characteristic polynomial is
\(\mu^2(\lambda+2\mu)(\rho_{\mathrm{char}}^2+k^2)^3\);
its stable space has complex dimension three, counting the
Jordan structure. Hence these solutions exhaust that space, and
the Dirichlet trace on it is bijective. Rotational invariance
establishes the assertion for every nonzero tangential covector.
\end{proof}

Linearizing the algebraic Riccati equation in a matrix \(X\) gives
\[
\begin{aligned}
 \mathcal S(X)&=(\mathsf Z_0-iR^T)A^{-1}X+XA^{-1}(\mathsf Z_0+iR)\\
                &=K_{\mathrm{st}}^*X+XK_{\mathrm{st}}=2X+\mathscr N(X),\\
 \mathscr N(X)&=N^*X+XN.
\end{aligned}
\]
Here \(*\) is conjugate transpose. Left multiplication by \(N^*\)
and right multiplication by \(N\) commute and have square zero.
Consequently \(\mathscr N^3=0\), and, on all complex
\(3\times3\) matrices, for every integer \(j\geq1\),
\begin{equation}\label{audrev-eq-boundary-sylvester}
 \mathcal S^{-j}(X)=2^{-j}\left(
 X-\frac j2\mathscr N(X)
   +\frac{j(j+1)}8\mathscr N^2(X)\right).
\end{equation}
The finite binomial identity \eqref{audrev-eq-boundary-sylvester}
inverts \(\mathcal S\) at every fixed order despite the repeated
stable root.

\subsection{Recovery of the normal derivatives}
\begin{lemma}
\label{lem:appendix-jet-invertibility}
For the layered model above, the coefficient of order \(k^{1-j}\)
in the asymptotic impedance symbol depends on normal derivatives
of the Lam\'e parameters of order at most \(j\). For each fixed
\(j\geq1\), after subtracting the contributions of lower-order
normal derivatives, its new \(j\)-th derivative part is given by
\eqref{audrev-eq-boundary-highest} and the two scalar observations
\eqref{audrev-eq-boundary-two-observations}. The latter determine
\(\partial_{r_{\mathrm{in}}}^j\mu\) and
\(\partial_{r_{\mathrm{in}}}^j\lambda\) uniquely; their determinant is
\[
 2^{1-2j}\frac{\mu^2}{(\lambda+3\mu)^2}>0.
\]
This includes all strongly convex coefficients, without any
restriction on the sign of \(\lambda\) beyond strong convexity.
\end{lemma}
\begin{proof}
Expand the layered symbol as
\[
 \mathfrak Z(r_{\mathrm{in}},k)\sim k\mathsf Z_0(r_{\mathrm{in}})+\mathsf Z_1(r_{\mathrm{in}})+k^{-1}\mathsf Z_2(r_{\mathrm{in}})+\cdots.
\]
Comparison in \eqref{audrev-eq-layered-riccati} yields
\begin{equation}\label{audrev-eq-boundary-recursion}
 \mathcal S(\mathsf Z_1)=\mathsf Z_0',\qquad
 \mathcal S(\mathsf Z_j)=\mathsf Z_{j-1}'-
 \sum_{\substack{r+t=j\\r,t\geq1}}\mathsf Z_rA^{-1}\mathsf Z_t
 \quad(j\geq2).
\end{equation}
Applying \(\mathcal S^{-1}\) to \(\mathsf Z_0'\) in
\eqref{audrev-eq-boundary-recursion} gives the complete first coefficient
\begin{equation}\label{audrev-eq-boundary-first}
 \begin{split}
 b&=\frac{\mu^2\lambda'+(\lambda^2-2\mu^2)\mu'}{\ell_*^2},\\
 \mathsf Z_1&=\begin{pmatrix}
 \mu'(1-a)+b&0&ib\\
 0&\mu'/2&0\\
 -ib&0&\mu'(1+a)+b
 \end{pmatrix}.
 \end{split}
\end{equation}
An explicit complete matrix formula at the second order is
\begin{equation}\label{audrev-eq-boundary-second}
 \begin{split}
 Q&=\mathsf Z_1'-\mathsf Z_1A^{-1}\mathsf Z_1,\\
 \mathsf Z_2&=\tfrac12Q-\tfrac14(N^*Q+QN)+\tfrac14N^*QN,\\
 (\mathsf Z_2)_{22}&=\tfrac14\mu''-\frac{(\mu')^2}{8\mu}.
 \end{split}
\end{equation}
Since \(\mathsf Z_1\) is given by \eqref{audrev-eq-boundary-first},
\eqref{audrev-eq-boundary-second} contains every term quadratic in first derivatives and
every second-derivative term;
in particular the nonlinear term \(-\mathsf Z_1A^{-1}\mathsf Z_1\) is retained.
After subtraction of terms involving only the values and first
derivatives, the second-order off-diagonal observation is
\[
 (\mathsf Z_2^{\mathrm{top}})_{13}/i
 =\frac{\mu^2\lambda''+(3\lambda^2+4\lambda\mu)\mu''}
        {2\ell_*^2}.
\]

In the recursion at order \(j\), every factor in the sum has
index at most \(j-1\). Differentiating the coefficients of
\(\mathcal S^{-1}\) also produces only normal derivatives of lower order.
Induction therefore shows that the part containing the new
\(j\)-th normal derivatives is exactly
\begin{equation}\label{audrev-eq-boundary-highest}
 \mathsf Z_j^{\mathrm{top}}
 =\mathcal S^{-j}\left((\partial_\lambda\mathsf Z_0)\lambda^{(j)}
                         +(\partial_\mu\mathsf Z_0)\mu^{(j)}\right),
 \qquad \lambda^{(j)}=\partial_{r_{\mathrm{in}}}^j\lambda,
 \quad \mu^{(j)}=\partial_{r_{\mathrm{in}}}^j\mu.
\end{equation}
The inversion needed here is for the two isotropic parameters,
not for freely prescribed derivatives of an arbitrary impedance
matrix. To see its rank explicitly, in the \((1,3)\) plane let
\[
 U_0=\begin{pmatrix}1&i\\-i&1\end{pmatrix}.
\]
In that plane \(\mathsf Z_0=2\mu I_2-\mu(1-a)U_0\),
\(N^*U_0=U_0N=0\), and
\(\partial_\lambda\mathsf Z_0=2\mu^2\ell_*^{-2}U_0\). Hence
\[
 \mathcal S^{-j}(\partial_\lambda\mathsf Z_0)=
 2^{1-j}\mu^2\ell_*^{-2}U_0,
 \qquad (\mathsf Z_j^{\mathrm{top}})_{22}=2^{-j}\mu^{(j)}.
\]
More explicitly, the two scalar observations satisfy
\begin{equation}\label{audrev-eq-boundary-two-observations}
 \begin{pmatrix}
 (\mathsf Z_j^{\mathrm{top}})_{22}\\
 (\mathsf Z_j^{\mathrm{top}})_{13}/i
 \end{pmatrix}
 =\begin{pmatrix}
 2^{-j}&0\\
 \beta_j&2^{1-j}\mu^2/\ell_*^2
 \end{pmatrix}
 \begin{pmatrix}\mu^{(j)}\\\lambda^{(j)}\end{pmatrix},
 \qquad
 \beta_j=2^{-j}\left(j(j+1)a^2-
                \frac{2\mu(2\lambda+3\mu)}{\ell_*^2}\right).
\end{equation}
Its determinant is
\begin{equation}\label{audrev-eq-boundary-jet-determinant}
 2^{1-2j}\frac{\mu^2}{\ell_*^2}>0.
\end{equation}
One recovers \(\mu^{(j)}\) first and then \(\lambda^{(j)}\).
There is no division by \(\lambda\); negative \(\lambda\) and
\(\lambda=0\) allowed by \eqref{eq:convexity} are included.
At normal order zero, recovery begins with the principal
DN symbol through

\begin{equation}\label{audrev-eq-boundary-order-zero}
 \mu=(\mathsf Z_0)_{22},\qquad d=(\mathsf Z_0)_{11},\qquad
 \lambda=\frac{3d\mu-4\mu^2}{2\mu-d},
 \qquad 2\mu-d=\frac{2\mu^2}{\ell_*}>0.
\end{equation}

The highest coefficients also have an energy check with the same
outward sign. For the frozen stable solution the divergence amplitude,
after removal of \(e^{iky_1}\), is
\(k(1-a)e^{-k r_{\mathrm{in}}}(if_1-f_3)\). An infinitesimal perturbation
\(\delta\lambda(r_{\mathrm{in}})=r_{\mathrm{in}}^j\delta\lambda^{(j)}/j!\) in the DN energy
therefore has coefficient
\((1-a)^2/2^{j+1}=2^{1-j}\mu^2/\ell_*^2\) at degree \(k^{1-j}\).
For the transverse solution the corresponding coefficient of
\(\delta\mu^{(j)}\) is \(2^{-j}\). The integral used is
\(\int_0^\infty r_{\mathrm{in}}^je^{-2k r_{\mathrm{in}}}\,d r_{\mathrm{in}}=j!/(2k)^{j+1}\).
This checks the highest linear coefficients; the full nonlinear
recursion is \eqref{audrev-eq-boundary-recursion}.

For an independent check of the nontrivial off-diagonal coefficient
at \emph{every} order, write $t=k r_{\mathrm{in}}$ and use the frozen stable
field $u_f=e^{iky_1-t}(I_3-tN)f$.  After division by
$ke^{iky_1-t}$, its strain amplitude is denoted by $E_f(t)$.
Direct differentiation gives
\begin{equation}\label{auditv-eq-energy-polynomial}
 \frac{2\overline{E_{e_1}(t)}:E_{e_3}(t)}{i}
  =8a^2t^2-8a^2t+3a^2-2a-1.
\end{equation}
For the first variation of the physical elastic energy with respect
to the $j$-th normal derivatives of the Lam\'e parameters, the integral
of this polynomial is
\begin{equation}\label{auditv-eq-energy-beta}
 \frac{1}{j!}\int_0^\infty t^je^{-2t}
  (8a^2t^2-8a^2t+3a^2-2a-1)\,dt
 =2^{-j}\left(j(j+1)a^2-
       \frac{2\mu(2\lambda+3\mu)}{\ell_*^2}\right)=\beta_j.
\end{equation}
The transverse shear amplitude gives the independent coefficient
$2^{-j}$ of $\partial_{r_{\mathrm{in}}}^j\mu$, while the divergence amplitude
$(1-a)(if_1-f_3)$ gives the coefficient
$2^{1-j}\mu^2/\ell_*^2$ of $\partial_{r_{\mathrm{in}}}^j\lambda$.
Thus the rank-two recovery matrix is checked by two genuinely
different mechanisms: Sylvester recursion and physical energy
variation. This independent calculation verifies the \emph{coefficients
of the highest-order unknown normal derivatives}; it does not replace
the full
nonlinear symbol recursion.
\end{proof}

\subsection{The full symbol of the physical Dirichlet-to-Neumann map}
\begin{proposition}
\label{prop:appendix-true-dn-symbol}
For a smooth strongly convex isotropic elasticity operator on a
smooth three-dimensional domain, its outward conormal
Dirichlet-to-Neumann map is a classical tangential pseudodifferential
operator of order one. In an orthonormal frame at a boundary point,
its principal symbol equals \(k\mathsf Z_0\) from
\eqref{audrev-eq-boundary-principal}, \(k=|\xi|\).
For every \(j\geq1\), after terms determined by the normal
derivatives of the Lam\'e parameters of orders at most \(j-1\),
their tangential derivatives and
the known boundary geometry have been subtracted, the new contribution
to the homogeneous symbol of degree \(1-j\) is
\eqref{audrev-eq-boundary-new-jet-symbol}.
The recursively constructed symbol equals the full symbol of the
actual outward traction DN map modulo a smoothing symbol.
\end{proposition}
\begin{proof}
The normal-flux computation and the boundary Poisson parametrix
provide the required identification. We retain the full physical
flux in an inward distance chart as follows.
Put \(n=\nabla r_{\mathrm{in}}\), \(w^\alpha=\nabla y^\alpha\),
and let \(J\) be the positive coordinate volume density. For the
isotropic tensor
\[
 C_{ijk\ell}=\lambda\delta_{ij}\delta_{k\ell}
       +\mu(\delta_{ik}\delta_{j\ell}+\delta_{i\ell}\delta_{jk}),
\]
define
\[
 A_{ik}=C_{ijk\ell}n_jn_\ell,\qquad
 B^\alpha_{ik}=C_{ijk\ell}n_jw^\alpha_\ell,\qquad
 D^{\alpha\beta}_{ik}=C_{ijk\ell}w^\alpha_jw^\beta_\ell.
\]
The minor and major symmetries of \(C\) imply that the tangential
flux is \((B^\alpha)^T\partial_{r_{\mathrm{in}}}u+D^{\alpha\beta}\partial_\beta u\).
Writing \(v=A\partial_{r_{\mathrm{in}}}u+B^\alpha\partial_\alpha u\), the divergence
equation in these coordinates is the exact system
\begin{equation}\label{audrev-eq-curved-flux}
 \begin{split}
 \partial_{r_{\mathrm{in}}}u={}&A^{-1}v-A^{-1}B^\alpha\partial_\alpha u,\\
 \partial_{r_{\mathrm{in}}}v={}&-\frac{\partial_{r_{\mathrm{in}}}J}{J}v
 -J^{-1}\partial_\alpha\big(J(B^\alpha)^TA^{-1}v\big)\\
 &+J^{-1}\partial_\alpha\big(
 J((B^\alpha)^TA^{-1}B^\beta-D^{\alpha\beta})
                 \partial_\beta u\big).
 \end{split}
\end{equation}
Repeated tangential indices are summed. Its coefficients involve
the material values and their tangential derivatives, but no
explicit unknown normal derivatives of the Lam\'e parameters. The term \(\partial_{r_{\mathrm{in}}}J/J\)
is known geometry.

For the stable graph \(v=-\mathcal P(r_{\mathrm{in}})u\), differentiating and
substituting \eqref{audrev-eq-curved-flux} gives the operator Riccati
identity, modulo smoothing terms in a microlocal construction,
\begin{equation}\label{audrev-eq-curved-riccati}
 \begin{split}
 \partial_{r_{\mathrm{in}}}\mathcal P={}&\mathcal PA^{-1}
       (\mathcal P+B^\beta\partial_\beta)
       -\frac{\partial_{r_{\mathrm{in}}}J}{J}\mathcal P\\
 &-J^{-1}\partial_\alpha
       \big(J(B^\alpha)^TA^{-1}\mathcal P\big)\\
 &-J^{-1}\partial_\alpha\big(
 J((B^\alpha)^TA^{-1}B^\beta-D^{\alpha\beta})
                                  \partial_\beta\big).
 \end{split}
\end{equation}
All products in the operator Riccati identity \eqref{audrev-eq-curved-riccati} are operator compositions. For
left quantization, their symbols are expanded by
\[
 p\#q\sim\sum_\gamma
 \frac{1}{i^{|\gamma|}\gamma!}
 (\partial_\xi^\gamma p)(\partial_y^\gamma q).
\]
At a point, in the orthonormal frame fixed above, let
\(k=|\xi|_{g_{r_{\mathrm{in}}}}\). The leading homogeneous symbol of
\(\mathcal P\) is \(k\mathsf Z_0\) from
\eqref{audrev-eq-boundary-principal}. Its leading linearization on
the coefficient of degree \(1-j\) is \(k\mathcal S\).

The orders of normal derivatives entering the complete symbol
recursion follow from \eqref{audrev-eq-curved-riccati}. Let
\(p_{1-r}\) be the homogeneous term of degree \(1-r\); by
induction it contains normal derivatives of the material coefficients of order at most \(r\),
together with known geometry and tangential derivatives of those
traces. At degree \(2-j\), the left-hand side of
\eqref{audrev-eq-curved-riccati} differentiates
\(p_{2-j}\) normally, so it can introduce the new order-\(j\)
material derivative. A composition term with no derivatives of its symbols containing
two lower symbol coefficients has indices adding to \(j\),
each at most \(j-1\). A nonprincipal composition term of length
\(|\gamma|\geq1\) has coefficient indices adding to
\(j-|\gamma|\), and therefore contains normal derivatives of the material coefficients of
order at most \(j-1\). Tangential derivatives do not increase
normal derivatives of the material parameters of order at most $j-1$. Terms with \(\partial_{r_{\mathrm{in}}}J/J\), derivatives of \(k\),
or derivatives of the known moving frame likewise use only
previous normal traces. Finally, differentiation of the
coefficients of a prior \(\mathcal S^{-1}\) introduces first
normal derivatives multiplying derivatives of lower order, and cannot create a new
pure order-\(j\) normal derivative.

It follows inductively that, after subtracting all terms known
from the normal traces of orders \(0,\ldots,j-1\) on the patch
and their tangential derivatives, the new term is precisely

\begin{equation}\label{audrev-eq-boundary-new-jet-symbol}
 k^{1-j}\mathcal S^{-j}
      ((\partial_\lambda\mathsf Z_0)\partial_{r_{\mathrm{in}}}^j\lambda
        +(\partial_\mu\mathsf Z_0)\partial_{r_{\mathrm{in}}}^j\mu).
\end{equation}

Thus \eqref{audrev-eq-boundary-jet-determinant} applies also to the
curved, tangentially variable problem. Subtracting the known lower
normal traces and their tangential derivatives is necessary:
self-adjointness of the DN operator does not make every term of a
left-quantized symbol Hermitian. For example, the anti-plane scalar
equation with coefficient \(\mu(y,r_{\mathrm{in}})\), at positive tangential
frequency, has the first two lower symbols
\[
 \frac{\partial_{r_{\mathrm{in}}}\mu-i\mu_y}{2},\qquad
 \frac{\partial_{r_{\mathrm{in}}}^2\mu+\mu_{yy}}4-
       \frac{(\partial_{r_{\mathrm{in}}}\mu)^2+(\partial_y\mu)^2}{8\mu}.
\]
The known tangential terms are retained and subtracted in the
recursion, rather than suppressed by a symmetry claim.

Identification with the physical DN symbol relies on the classical
smooth Dirichlet--Poisson parametrix principle
for strongly elliptic second-order matrix systems: if the decaying
frozen normal solution space has the same dimension as the boundary
data space and its boundary-value projection is invertible for all
nonzero tangential covectors, the true Dirichlet-to-Neumann operator
has the full classical symbol of the stable conormal graph modulo a
smoothing operator.  Repeated stable characteristic roots are
permitted, provided the stable and unstable spectral subspaces are
separated and vary smoothly.  Here the principal elastic
characteristic determinant is
$\mu^2(\lambda+2\mu)(\rho_{\mathrm{char}}^2+|\xi|^2)^3$;
the stable and unstable roots are $\rho_{\mathrm{char}}=\pm i|\xi|$.
The explicit stable solution
$e^{-|\xi|r_{\mathrm{in}}}(I_3-|\xi|r_{\mathrm{in}}N)f$ has boundary value exactly $f$;
thus the Dirichlet complementing condition is satisfied despite
$N^2=0$ and $N\ne0$.  Strong convexity, Korn's inequality and
smooth Dirichlet regularity supply the required global correction
operator.  In particular, one must identify the formal Riccati
symbol with the \emph{physical} conormal DN map by the Poisson
construction, not by symbol algebra alone.

The stable Poisson parametrix is constructed from the recursively
determined symbol, and its smoothing error is corrected through
the Dirichlet problem. The frozen inward first-order evolution has three stable and three
unstable modes, with normalized eigenvalues \(-1\) and \(+1\),
respectively.
The stable Dirichlet projection is invertible by
\eqref{audrev-eq-boundary-jordan}; the stable and unstable spaces
are separated even when either has a Jordan block. The usual
classical boundary construction therefore solves
\eqref{audrev-eq-curved-riccati} successively by the invertible
\(\mathcal S\), sums its symbol by Borel summation, and constructs
a Poisson parametrix whose boundary trace is the identity modulo
smoothing. Its frozen leading term is
\(e^{-k r_{\mathrm{in}}}(I_3-k r_{\mathrm{in}}N)\). The exponential decay absorbs the polynomial
Jordan factor.

Cutting this parametrix off inside a collar produces, away from
the boundary, a source smoothing in the boundary data; the
near-boundary error is smoothing by the successive construction.
The smooth Dirichlet inverse corrects the collar-cutoff source and the
near-boundary smoothing error. Its existence follows from coercivity of
the physical Dirichlet problem: decompose
the elastic energy into trace-free and trace parts:
\[
 \lambda(\tr E)^2+2\mu|E|^2
 =2\mu\left|E-\tfrac13(\tr E)I_3\right|^2
     +(\lambda+2\mu/3)(\tr E)^2.
\]
The uniform bounds in \eqref{eq:convexity}, zero-trace Korn's
inequality and Poincar\'e's inequality give coercivity; the usual
Dirichlet existence and elliptic regularity give the correcting
inverse. Consequently the true outward DN map and the constructed
\(\mathcal P(0)\) have the same complete classical symbol modulo
smoothing. No normal derivative of the full impedance is assumed
to be observed: successive lower homogeneous terms of the given
DN map are the observations.
\end{proof}

\begin{proof}[Proof of Theorem~\ref{thm:appendix-boundary-jets}]
The principal symbol in
Proposition~\ref{prop:appendix-true-dn-symbol} first determines the
boundary values: in the orthonormal frame,
\(\mu=(\mathsf Z_0)_{22}\) and \(d=(\mathsf Z_0)_{11}\), so
\eqref{audrev-eq-boundary-order-zero} recovers \(\lambda\).
Localizing equal full DN maps on a boundary patch gives equal
complete symbols there. The principal recovery and then
\eqref{audrev-eq-boundary-jet-determinant} prove equality of all
normal traces by induction. Their tangential derivatives agree
because the recovered normal derivatives of $\lambda_j$ and
$\mu_j$ agree as smooth functions on the boundary patch.
The argument is local on every boundary component; interactions
between disjoint patches or distinct components contribute only
smooth kernels and do not alter the local symbol. On a hole
boundary the distance is still taken into \(\Omega\), so the
outward traction remains \(-v\). This proves the equality of all boundary
derivatives without a connected-boundary assumption. The
coefficient difference is therefore flat on \(\partial\Omega\),
as required by the extension construction in
Lemma~\ref{lem:common-extension}.
\end{proof}

\section*{Acknowledgments}

This work was supported by the National Natural Science Foundation of China (Grant No. 12371422) and the Fundamental Research Funds for the Central Universities (Jilin University). The authors used ChatGPT (OpenAI) for exploratory mathematical
discussions and editorial assistance during the preparation of this
manuscript. The mathematical arguments and proofs presented in the
final manuscript were carefully reviewed and independently verified
by the authors, who take full responsibility for the correctness and
mathematical content of the work.

\end{document}